\documentclass[11pt,letterpaper]{amsart}
\usepackage{stmaryrd}
\usepackage{amsfonts}
\usepackage{esint}

\usepackage{amsmath}
\usepackage{amssymb}
\usepackage{latexsym}
\usepackage{amscd}
\usepackage{mathrsfs}

\usepackage{hyperref}
\usepackage{graphicx} 
\usepackage{amsthm}
\usepackage{xypic}
\usepackage{bm}
\usepackage[all]{xy}
\usepackage{color}

\newtheorem{theorem}{Theorem}

\newtheorem{remark}[theorem]{Remark}

\newtheorem{corollary}[theorem]{Corollary}

\newtheorem{definition}[theorem]{Definition}

\newtheorem{lemma}[theorem]{Lemma}
\newtheorem{proposition}[theorem]{Proposition}

\numberwithin{equation}{section} \numberwithin{theorem}{section}

\renewcommand{\oddsidemargin}{5mm}

\def\C{\mathbb C}
\def\R{\mathbb R}

\def\Z{\mathbb Z}

\def\Z{\mathbb Z}
\def\N{\mathbb N}

\def\cal{\mathcal}

\def\na{\nabla}
\def\bn{\overline\nabla}

\def\f#1#2{\frac{#1}{#2}}

\def\a{\alpha}
\def\be{\beta}

\def\r{\Re_{I\!V}}

\def\p#1{\partial #1}

\def\de{\delta}
\def\De{\Delta}
\def\e{\eta}
\def\ep{\epsilon}
\def\vep{\varepsilon}
\def\G{\Gamma}
\def\g{\gamma}
\def\k{\kappa}
\def\la{\lambda}
\def\La{\Lambda}
\def\lan{\langle}
\def\ran{\rangle}

\def\Om{\Omega}
\def\th{\theta}
\def\Th{\Theta}

\def\si{\sigma}
\def\Si{\Sigma}

\def\r{\rho}
\def\z{\zeta}
\def\div{\mathrm{div}}
\def\vth{\vartheta}

\begin{document}

\title
[Hamiltonian stationary Lagrangian currents]
{Hamiltonian stationary Lagrangian currents in Ricci-flat K\"ahler manifolds}
\author{Qi Ding}
\address{Shanghai Center for Mathematical Sciences, Fudan University, Shanghai 200438, China}
\email{dingqi@fudan.edu.cn}

\thanks{The author is partially supported by NSFC (Grant Numbers 12625106, 12371053 and 12526203). The author would like to express his sincere gratitude to Jingyi Chen for valuable discussion.}

\begin{abstract}

In this paper, we investigate Hamiltonian stationary Lagrangian currents with single-valued phases in open sets of Ricci-flat K\"ahler manifolds.
We first establish the relation between phase and mean curvature for a class of integral Lagrangian currents in Ricci-flat K\"ahler manifolds. 
Then we derive several fundamental properties of integral Lagrangian currents with single-valued harmonic phases, including an almost monotonicity formula, an Allard-type regularity theorem, and the structure of limits of currents in this class.
Finally, we apply the above results to study the regularity and rigidity of Hamiltonian stationary Lagrangian graphs in complex Euclidean space under certain conditions. 
\end{abstract}

\maketitle
\tableofcontents

\section{Introduction}

Let $(M,\omega)$ be a Ricci-flat K\"ahler manifold of dimension $n$ with a Riemannian metric $g$ and complex structure $J$ on $M$, and $U$ be an open set of $M$ so that there is a holomorphic $n$-form $\Om$ on $U$ satisfying $(-1)^{n^2/2}\Om\wedge\overline{\Om}=2^{n}\omega^n/n!$.  
For a smooth Lagrangian $L\subset U$,
\begin{equation}\aligned\label{OmThL}
\Omega\big|_L=e^{\sqrt{-1}\th}\mathrm{vol}_L,
\endaligned
\end{equation}
where $\th$ is called \emph{phase} of $L$, and $\mathrm{vol}_L$ is an induced volume form w.r.t. the orientation of $L$.
The Maslov class on $L$ is defined by the 1-form $d\th$, and $L$ has \emph{zero-Maslov} class (in $U$) if and only if $\th$ is single-valued.
Let $\na^L$ denote the Levi-Civita connection of $L$ w.r.t. its induced metric from $(M,g)$,
then the mean curvature $H$ of $L$ satisfies
\begin{equation}\aligned\label{HLJnaTh}
H=J\na^L\th.
\endaligned
\end{equation}
We call $L$ \emph{special Lagrangian} (SL) if and only if its phase $\th$ is a constant, and call $L$ \emph{Hamiltonian stationary Lagrangian} (HSL) if and only if its phase $\th$ is harmonic on $L$.

Special Lagrangian submanifolds are of fundamental importance in the study of mirror symmetry and the Strominger-Yau-Zaslow conjecture \cite{SYZ}. For existence results in certain Calabi-Yau manifolds, we refer to Collins-Jacob-Lin \cite{CJL} and the references therein.
For the construction of SLs in general Calabi-Yau manifolds, it is natural to investigate minimization in Lagrangian homology.
Schoen-Wolfson \cite{SW} developed an existence and regularity
theory for area minimizers among Lagrangian maps. They
proved that each Lagrangian homology group in a compact symplectic 4-manifold is generated by classes representable by Lagrangian Lipschitz maps: these are branched HSL immersions, except at finitely many singular points whose tangent maps are HSL cones in $\C^2$. These cones are now called Schoen-Wolfson cones. Micallef and Wolfson \cite{MW,W} found that there indeed exists an integral homology class in a K3 surface such that the area minimizer among Lagrangian two-spheres representing this class has isolated conical singularities.

In high dimension, the situation becomes more difficult, partly due to the absence of a correspondence between SLs and holomorphic curves that holds	for hyperk\"ahler surfaces.
In fact, it is unknown whether a general Calabi-Yau manifold admits even a single SL. 
Therefore, as a strictly broader class than SLs, HSLs may be regarded as the 'best' representatives of a Hamiltonian isotopy class of Lagrangians, and their study may facilitate a deeper understanding of the family of all Lagrangians (see \S7 in \cite{JLS} by Joyce-Lee-Schoen for further results).

There are numerous HSL examples, e.g., the totally geodesic $\R P^n\subset\C P^n$, the flat tori $\mathbb{S}^1(r_1)\times\cdots\times\mathbb{S}^1(r_n)\subset\C^n$ for constants $r_1,\cdots,r_n>0$, and the corresponding Clifford tori in $\C P^n$ (see Oh \cite{Oh0,Oh1}). There is a large class of smooth HSL surfaces constructed via integrable systems, perturbation and gluing techniques (refer to the bibliography of Chen-Ma \cite{CMa1} for a more comprehensive list of references). 
Recently, Gaia-Orriols-Rivi\`ere \cite{GOR} constructed many HSL surfaces with any finite number of isolated Schoen-Wolfson conical singularities via variational methods. 

The regularity theory for HSLs plays a critical role in the study of their existence in almost K\"ahler manifolds.
Nevertheless, it is difficult to establish the regularity of HSLs under weak conditions, partly because they are high codimensional and governed by a fourth-order PDE. 
In \cite{CW}, Chen-Warren proved a Morrey-type theorem: If a $C^1$ Lagrangian submanifold in $\C^n$ is a critical point of the volume functional under Hamiltonian variations, then it must be real analytic.
The ambient space $\C^n$ can be weakened to abitrary almost K\"ahler manifold by Bhattacharya-Chen-Warren \cite{BCW} (see \cite{BS} for further results).
Moreover, Chen-Warren \cite{CW} first studied the regularity of weak solutions to \emph{geometric Hamiltonian stationary equation}
\begin{equation}\aligned\label{HSL}
\int_L\left\lan\na^L\th,\na^L\phi\right\ran=0\qquad \mathrm{for\ any\ } \phi\in C^\infty_c(\mathbb{U}\times\R^n),
\endaligned
\end{equation}
where $L\subset\C^n$ is a countably $n$-rectifiable Lagrangian graph over an open subset $\mathbb{U}$ of $\R^n$, and the phase $\th\in W^{1,2}(\mathbb{U})$.

Minicozzi \cite{M} successfully established the compactness and smoothness of minimizers for Willmore energy amongst all embedded Lagrangian tori in $\C^2$. Nevertheless, significant challenges remain in studying the compactness and regularity of volume minimizers within Lagrangian homology classes in Ricci-flat K\"ahler manifolds. One of the main difficulties stems from the fact that the $L^2$ integral of mean curvatures of Lagrangian minimizers for volume may be infinite.

Since the volume form $\mathrm{vol}_L$ in \eqref{OmThL} defines an orientation, we consider the class of integral Lagrangian currents in Ricci-flat K\"ahler manifolds within the framework of geometric measure theory, as a natural generalization of orientable smooth Lagrangian submanifolds. In general, an orientable immersed Lagrangian submanifold can locally decompose into embedded components with smooth phases, but its phase function does not belong to $W^{1,1}_{\mathrm{loc}}$ and is merely of locally bounded variation. Accordingly, we consider phases of bounded variation on varifolds associated with integral Lagrangian currents, which locally decompose into components whose phases live in $W^{1,q}$ for some $q>1$.

Cheeger \cite{Ch} and Cheeger-Colding \cite{CCo3} showed that the differential of Lipschitz functions can be well-defined a.e. on rectifiable metric spaces. This enables us to define Cheeger $q$-energy for (multi-valued) phases of integral Lagrangian currents through approximation by Lipschitz functions for each constant $q>1$.
Let $U$ be an open set of the Ricci-flat K\"ahler manifold $(M^n,\omega)$ so that there is a holomorphic $n$-form $\Om$ on $U$ satisfying $(-1)^{n^2/2}\Om\wedge\overline{\Om}=2^{n}\omega^n/n!$,
and $T$ be an integral Lagrangian current in $U$ with phase $\th$ having bounded Cheeger $q$-energy and with zero boundary $\p T$ in $U$.
In general, the Sobolev space $W^{1,q}$ on $L$ (w.r.t. $\mu_T$) is not reflexive. However, we can overcome this difficulty by a calibration argument, and prove that the generalized mean curvature $H$ of the varifold $|T|$ satisfies  (see Theorem \ref{thW12-H})
\begin{equation}\aligned\label{HJnaLth*}
H=J\na^L\th\qquad\mu_T-a.e..
\endaligned
\end{equation}
It's worth noting that every Schoen-Wolfson cone has locally bounded Cheeger $q$-energy of its phase for every $q\in(1,2)$, but has unbounded Cheeger energy (i.e., the case of $q=2$) in any small neighborhood of its vertex. In this situation, the Cheeger energy for $n=2$ is the so-called Willmore energy.

Let $T_i$ be a sequence of integral Lagrangian currents in $U$ with phases $\th_i$ and uniformly bounded mass such that each $\th_i$ has bounded variation on $L_i:=\mathrm{spt}T_i$ w.r.t. $\mu_{T_i}$ (see \eqref{DEF-naphi} for the definition) and the generalized mean curvature $H_i=J\na^{L_i}\th_i$ of $|T_i|$ is $L^q$-integrable for some $q>1$. Here, the condition on $\th_i$ is strictly weaker than requiring uniformly bounded Cheeger $q$-energy (see \eqref{HJnaLth*} and \eqref{thnaLiejthX}).
If $|T_i|$ converges to a varifold $V_*$ with support $L_*$ and $e^{2\sqrt{-1}\th_i}$ converges in the sense of measure to a function $\z$, then
we can prove that $\z$ has bounded variation in $U$ w.r.t. the Radon measure $\mu_{V_*}$ such that the mean curvature $H_*$ of $V_*$ satisfies (see Theorem \ref{zinftyJxiLVinfty})
\begin{equation}\aligned
H_*=-\f{\sqrt{-1}}{2}\bar{\z} J\na^{L_*}\z\qquad\mu_{V_*}-a.e..
\endaligned
\end{equation}
Here, the gradient $\na^{L_*}\z$ is defined in the distribution sense. 
In particular, $|\z|\equiv1$ $\mu_{V_*}$-a.e., and $H_*$ is $L^q$-integrable w.r.t. $\mu_{V_*}$.

Under the finite total curvature condition, almost monotonicity for volume ratios on submanifolds are available, see Simon \cite{S1} and Chen-Warren \cite{CW1} for instance.
In general, however, no such formulas exist for HSLs without this condition.
For instance, the cylinder $\mathbb{S}^1(t)\times\R$ is Hamiltonian stationary in $\C^2$, and clearly fails to admit monotonicity formulas.
In fact, Minicozzi \cite{M} proved that the cylinder minimizes area within its Hamiltonian isotopy class. On the positive side, Schoen-Wolfson \cite{SW} established a monotonicity formula for disk-type weak HSL surfaces that lift to H-minimal Legendrian surfaces, while Rivi\`ere \cite{R} derived an almost monotonicity formula for such surfaces in the Heisenberg group. 
However, to our knowledge, no such formula is available for HSLs of dimension $\ge3$.

Lagrangian submanifolds with
single-valued phases (i.e., the zero-Maslov class) have been extensively studied. Graphical HSL in $\C^n$ are typical  examples admitting single-valued phases. Furthermore, every smooth closed HSL in a Calabi-Yau manifold admits a decomposition into finitely many simply connected components, on each of which the phase is single-valued.
For further references in Lagrangian mean curvature flow with zero-Maslov class, see Chen-Li \cite{CL}, Neves \cite{Ne} and Lotay-Schulze-Sz\'ekelyhidi \cite{LSS} for instance.

Harmonic functions on metric spaces can be defined in the distribution sense, provided the Cheeger energy is bounded. Accordingly, we define \emph{harmonic} phases for a class of integral Lagrangian currents in Ricci-flat K\"ahler manifolds, whose local indecomposable components have, roughly speaking, phases with bounded Cheeger energy (see Definition \ref{Def-harm-curr} for the rigorous definition).
Typical examples include multiplicity one currents supported on oriented immersed HSL submanifolds in Ricci-flat K\"ahler manifolds. In what follows, we study the class of integral Lagrangian currents with single-valued harmonic phases in open subsets of Ricci-flat K\"ahler manifolds, and apply the obtained results to study the regularity and rigidity of HSL graphs in $\C^n$.
We expect our results to be useful for studying limits of certain Lagrangian sequences with uniformly finite fundamental groups and uniformly bounded volumes.

For a bounded open subset $U$ of a Ricci-flat K\"ahler manifold $(M^n,\omega)$,
by scaling we may assume that $U$ is a $2n$-submanifold in $\R^m$ for some integer $m\ge 2n$ with $\mathbf{B}_{R_0}(\mathbf{0})\cap\p U=\emptyset$ for some $R_0>0$ from Nash's isometric embedding theorem, where $\mathbf{B}_r(\mathbf{x})$ denotes the ball in $\R^m$ with radius $r$ and centered at $\mathbf{x}\in\R^m$. Let $\k$ be the smallest constant satisfying $|\mathbf{A}_U(X,X)|\le\k|X|^2/n$ for any $X$ tangent to $U$, where $\mathbf{A}_U$ denotes the second fundamental form of $U$.

Let $T$ be an integral Lagrangian current in $U$ with $\p T\llcorner U=\emptyset$ and with bounded, single-valued, harmonic phase $\th$ satisfying \eqref{OmThL}.
Let $V=|T|$ denote the varifold associated with $T$, and $\mu_V$ be the Radon measure corresponding to $V$. From Sobolev inequality on $V$ by Michael-Simon, 
we can derive a mean value inequality for $\th$, which yields a uniform local lower bound on the volume w.r.t. $\mu_V$.

Let $\ell_*:=\max\{2,2^{n/2}e^{\k R_0}\sup_{L\cap\mathbf{B}_{R_0}(\mathbf{0})}|\th|\}$.
We prove an almost monotonicity formula for the function
\begin{equation}\aligned\label{ThpmuVr0}
\mathbf{m}_{\mathbf{p}}(\mu_V,r):=\f1{r^n}\int_{\mathbf{B}_{r}(\mathbf{p})}|\th|^{\ell_*} d\mu_V+\f{\ell_*^2e^{-\k r}}{2^{n+2}}\int_{r/2}^rs^{1-n}\int_{\mathbf{B}_{s}(\mathbf{p})}|\th|^{\ell_*-2}|\na^L\th|^2 d\mu_Vds
\endaligned
\end{equation}
as follows  (see \eqref{mon-bfTh}):
\begin{theorem}\label{Al-Mono-int}
\begin{equation}\aligned
e^{\k R}&\mathbf{m}_\mathbf{p}(\mu_V,R)\ge e^{\k r}\mathbf{m}_\mathbf{p}(\mu_V,r)+ \f{n-1}{n}\int_{\mathbf{B}_R(\mathbf{p})\setminus \mathbf{B}_r(\mathbf{p})}|\th|^{\ell_*}\f{|\na^N\r|^2}{\r^{n}}d\mu_V\\
&+\f{\ell_*^2}{2^{n+2}}\int_r^{R/2}s^{1-n}\int_{\mathbf{B}_{s}(\mathbf{p})}|\th|^{\ell_*-2}|\na^L\th|^2 d\mu_Vds-\f{1}{nr^n}\int_{\mathbf{B}_r(\mathbf{p})}|\th|^{\ell_*}|\na^N\r|^2d\mu_V,
\endaligned
\end{equation}
whenever $0<r<R$ and $\mathbf{B}_{R}(\mathbf{p})\subset\mathbf{B}_{R_0}(\mathbf{0})$, where $\r(\mathbf{x})=|\mathbf{x}-\mathbf{p}|$ is the distance function from $\mathbf{p}$ on $\R^m$, and $\na^N$ denotes the projection into the normal bundle of $L$ in $\R^m$ $\mu_V$-a.e.. 
\end{theorem}
\begin{remark}
The single-valued phase condition in Theorem \ref{Al-Mono-int} is necessary, as illustrated by cylindrical examples. Monotonicity formulas of Schoen-Wolfson \cite{SW} and Rivi\`ere \cite{R} yield corresponding monotonicity formulas for exact HSLs in $\C^2$. 
\end{remark}
As a consequence, we immediately have that $\mathbf{m}_{\mathbf{p}}(\mu_V,r)$ is uniformly bounded and converges as $r\to0$ to a function $\mathbf{m}_{\mathbf{p}}(\mu_V)$, which is upper semi-continuous on $\mathbf{p}$ (see Proposition \ref{semi-bfTh}).
From Theorem \ref{Al-Mono-int}, the volume doubling property holds w.r.t. $\mu_V$ (see Theorem \ref{Ahlfors}). In particular, we obtain a sharp lower bound (see Corollary \ref{lower}): 
\begin{equation}\aligned\label{sharp-lower}
\liminf_{r\to0}r^{-n}\mu_V(\mathbf{B}_{r}(\mathbf{p}))\ge\omega_n,
\endaligned
\end{equation}
where $\omega_n$ denotes the volume of standard unit $n$-ball in $\R^n$.

In 1972, Allard \cite{A} proved the celebrated regularity theorem for rectifiable $n$-varifolds with generalized mean curvatures in $L^q$ for some $q>n$ (see also Simon \cite{S} and De Philippis-Gasparetto-Schulze \cite{DGS}). Using a new method, we prove Allard's type regularity theorem for HSLs as follows (given in \S 6).  
\begin{theorem}\label{Allard}
For a constant $\be>1$, there is a constant $\de\in(0,1/4]$ depending only on $n,\be$ and $U$ so that if $T=(L,\vth,\vec{T})$ is an integral Lagrangian current in $U$ with single-valued harmonic phase $\th$ and $\p T\llcorner U=0$ so that $\mathbf{0}\in L$, $0\le\th\le\be$, and $V=\vth|L|$ satisfies
\begin{equation}\aligned\label{muVbgTh0}
\mu_V(\mathbf{B}_{r}(\mathbf{0}))\le(1+\de)\omega_nr^n\qquad\mathrm{and}\qquad \mathbf{m}_\mathbf{0}(\mu_V,r)\le(1+\de)\mathbf{m}_\mathbf{0}(\mu_V),
\endaligned
\end{equation}
then $L\cap \mathbf{B}_{\de/9}(\mathbf{0})$ is an embedded Lagrangian submanifold with
$$d_{G}(T_\mathbf{x}L,T_\mathbf{y}L)\le c|\mathbf{x}-\mathbf{y}|\qquad \mathrm{for\ each}\ \mathbf{x},\mathbf{y}\in L\cap \mathbf{B}_{\de/9}(\mathbf{0}),$$
where $c>0$ is a constant depending only on $n,\be$ and $U$.
\end{theorem}
Here, $d_{G}(P_1,P_2)$ denotes the distance between $P_1,P_2$ in the Grassmannian manifold $G_{n,m}$ for each two $n$-planes $P_1,P_2$ (through the origin) in $\R^m$ (see \S 7.1 in \cite{X} for more details). 
The proof of Theorem \ref{Allard} differs from existing arguments due to the absence of uniform Morrey-type estimates for phases. We use the almost monotonicity formula to show a 'quantitative' indecomposablity for integral Lagrangian currents of single-valued harmonic phases in small balls under the condition \eqref{muVbgTh0}.
Then we can derive a Neumann-Poincar\'e inequality on such currents, which yields H\"older estimates of phases. As a result, we complete the proof of Theorem \ref{Allard} directly with the help of Theorem 1.3 in \cite{BV} by Bourni-Volkmann.

From Theorem \ref{Allard}, we immediately have the following regularity result. 
\begin{corollary}\label{Allard-HSL}
If $\liminf_{r\to0}r^{-n}\mu_T(\mathbf{B}_r(\mathbf{p})))\le\omega_n$ for some $\mathbf{p}\in\mathrm{spt}V\cap\mathbf{B}_1(\mathbf{0})$, then $\mathrm{spt}T$ is analytic in a small neighborhood of $\mathbf{p}$.
\end{corollary}
From \cite{Ba,Fc,NV1,Y1}, every
Lipschitz graphical SL cone in $\R^{2n}$ is flat if $n\le4$.
Building on this, Bhattacharya \cite{B} recently proved that any locally Lipschitz HSL graph in $\C^{n}$ is smooth for $n\le4$, where 2-dimensional case has been obtained in \cite{BW} by Bhattacharya-Warren. More recently, Bhattacharya-Orriols-Skorobogatova \cite{BOS} generalized this smoothness theorem from the Euclidean setting to almost K\"ahler manifolds, and using AI tools they constructed a non-flat SL graphical cone in $\R^{10}$.
By a standard blow-up argument and Federer's dimension reduction argument,
Corollary \ref{Allard-HSL} and Bernstein theorem for SL graphs of dimension $\le4$ can directly imply that every locally Lipschitz, multiplicity one HSL $n$-current in a Ricci-flat K\"ahler manifold has smooth support for $n\le4$.

Chen-Warren \cite{CW1} established compactness theorems on the space of all compact immersed HSLs in $\C^n$ with uniform bounds on volume and total (extrinsic) curvature, which has been generalized to ones in almost K\"ahler manifolds by Chen-Ma \cite{CMa2}. For K\"ahler surfaces, the assumption of total curvature can be made on the Willmore energy for HSL tori, and Chen-Ma \cite{CMa1} gave a more precise bubble-tree convergence on the singular part.

We study the singular structure of varifold limits of integral Lagrangian currents with single-valued harmonic phases, 
and use this structure to define an orientation on such limits. Combining Theorem \ref{zinftyJxiLVinfty}, we can prove the following compactness result (see Theorem \ref{current-n}). 
\begin{theorem}\label{Comp-LCHP}
Let $\{T_i\}$ be a sequence of integral Lagrangian currents in $U$ with $\p T_i\llcorner U=0$ so that each $T_i$ has single-valued harmonic phase $\th_i$. If $|\th_i|$ and mass of $T_i$ and $\p T_i$ are uniformly bounded,
then there are a subsequence $\{i'\}\subset\{i\}$ and a Hamiltonian stationary Lagrangian current $T_\infty=(L_\infty,\vth_\infty,\vec{T}_\infty)$ in $U$ with bounded single-valued phase $\tilde{\th}$ of bounded variation w.r.t. $\mu_{T_\infty}$ so that $J\na^{L_\infty}\tilde{\th}$ is the generalized mean curvature of $|T_\infty|$, and  
$|T_{i'}|\to |T_\infty|$, $\th_{i’}\to\th_\infty$ both in the sense of Radon measure, where $\tilde{\th}$ and $\th_\infty$ differ only by a constant on every indecomposable component of $T_\infty$.
\end{theorem}
In Theorem \ref{Comp-LCHP}, the existence of the current $T_\infty$ ensures that the phase is well-defined $\mu_{T_\infty}$-a.e., while the convergence in the sense of measure ensures that $\na^{L_\infty}\tilde{\th}$ is divergence-free in the weak sense.
In general, $T_{i'}$ may have no subsequences converging to $T_\infty$, and may converges to zero (see Remark \ref{TitoTinftyW} for more details).
Although the harmonicity of the phase $\th_\infty$ of $T_\infty$ in the sense of Definition \ref{Def-harm-curr} is not yet known, we are still able to derive several useful properties of $\th_\infty$ in \S7, notably a Sobolev inequality and a Neumann-Poincar\'e inequality under suitable assumptions.

In what follows, we will use our above results to study the regularity and rigidity of HSL graphs in $\C^n$ (given in \S 8).
We consider a locally Lipschitz Lagrangian graph over an open set $\mathbb{U}\subset\R^n$ defined by
$$\{(x,Du(x))\in\R^n\times\R^n|\, x\in \mathbb{U}\}$$
for some function $u\in C^{1,1}_{loc}(\mathbb{U})$. 
For each $r>0$, let $B_r\subset\R^n$ denote the ball of radius $r$ centered at the origin.
In Theorem 1.2 of \cite{CW}, Chen-Warren proved Schauder estimates for $C^{1,1}$ weak solutions to \eqref{HSL} under one of the following conditions: a) the Lagrangian phase
$\th$ satisfies $|\th|\ge\f{n-2}2\pi+\de$; b) the potential $u$ is strongly convex; c) $-(1-\de)\le D^2u\le1-\de$ for some $\de>0$.
The constant $\de$ in condition a) can be removed for Schauder estimates by Bhattacharya-Ogden \cite{BO}. 
Both the above Schauder estimates depend on the Hessian bounds of solutions, while Hessian bounds can not be derived from $C^{1,\f13}$ solutions by the counterexample in Theorem 1.2 of \cite{BO}. 

Using Theorem \ref{Allard} and Theorem \ref{Comp-LCHP}, 
we obtain the following curvature estimate for graphs with certain constraints. 
\begin{theorem}\label{HessHS}
Let $u\in C^{1,1}_{loc}(B_1)$ and $L\subset\R^{2n}$ be a Lagrangian graph of $Du$ over $B_1$ with harmonic phase $\th$.
If one of the following conditions holds:
\begin{itemize}
  \item[i)] $(n-2)\pi/2\le\th<n\pi/2$ a.e. on $B_1$, 
  \item[ii)] $u(x)+\f1{2\sqrt{3}}|x|^2$ is convex on $B_1$,
\end{itemize}
then $L$ is smooth, and the norm of the second fundamental form of $L\cap(B_{1/2}\times\R^n)$ is bounded by a constant depending only on $n$ and $\sup_{B_1}|Du|$.
\end{theorem}
The curvature estimate in Theorem \ref{HessHS} immediately yields a rigidity theorem for entire smooth HSL graphs with condition i) or ii). 
The condition i) is said to be \emph{critical} or \emph{supercritical} for the constant phase (see Yuan \cite{Y2}). 
The condition ii) means $D^2u\ge-1/\sqrt{3}$ a.e. on $B_1$, which is equivalent to $-\sqrt{3}\le D^2u\le\sqrt{3}$ a.e. by Lewy-Yuan rotation.
The regularity of SL graphs under the condition i) or ii) has been obtained in \cite{D1,WdY,WmY0,WmY1}.

For the case ii), the curvature estimate in Theorem \ref{HessHS} actually depends only on $n$.
For the case i), the difficulty relies that there are no Neumann-Poincar\'e inequalities with uniformly coefficients on $L$, since the lack of the Hessian bounds of solutions. We overcome this difficulty using geometric measure theory through a delicate argument by contradiction combining results obtained in \S 5 and \S 7.
In general, we cannot derive the Hessian estimates for the solutions in view of Theorem 1.2 of \cite{BO} and Lewy-Yuan rotation.

\section{Preliminary}
\subsection{Hamiltonian Stationary Lagrangian submanifolds}\

Let $(M,g,J,\omega)$ be a $2n$-dimensional symplectic manifold with almost K\"ahler structure given by a Riemannian metric $g$, compatible with a symplectic 2-form $\omega$, and an almost complex structure $J$ on $M$. 
Namely, $\omega(X,Y)=g(JX,Y)$ for arbitrary smooth vector fields $X, Y$ on $M$. 
For an open set $U\subset M$,
a smooth submanifold $L\subset U$ is \emph{Lagrangian} if 
$$\omega\big|_L=0\qquad \mathrm{and}\qquad \mathrm{dim}\,L=n.$$
Let $H$ denote the mean curvature vector of $L$,
and $\div_L$ denote the divergence on $L$. $L$ is said to be \emph{Hamiltonian stationary} if the volume is critical for compactly
supported smooth Hamiltonian variations, i.e.,
\begin{equation}\aligned\label{Def-HS}
\div_L(JH)=0,
\endaligned
\end{equation}
which is equivalent to that the 1-form $\omega_H:=\omega(H,\cdot)$ is coclosed on $L$.
Moreover, $d\omega_H=0$ for the K\"ahler-Einstein $M$ (see Appendix A in \cite{SW} by Schoen-Wolfson).

Now we further assume that $(M,\omega)$ is Ricci-flat K\"ahler with complex structure $J$ (see Yau \cite{Yau} or Tian \cite{T} for more details). 
By choosing $U$ smaller (if needed), there is a holomorphic $n$-form $\Om$ on $U$ satisfying $(-1)^{n^2/2}\Om\wedge\overline{\Om}=2^{n}\omega^n/n!$. Here, $\Om$ has the unit norm and $d\Om=0$ from a Bochner formula. In particular, Calabi-Yau manifolds admit global nowhere zero parallel holomorphic $n$-forms.
For an oriented smooth Lagrangian $L\subset U$ with an induced volume form $\mathrm{vol}_L$ w.r.t the orientation of $L$, there is a smooth multi-valued function $\th$ on $L$ so that 
\begin{equation}\aligned\label{OmLth}
\Omega\big|_L=e^{\sqrt{-1}\th}\mathrm{vol}_L,
\endaligned
\end{equation}
which can be derived analog to the argument in Euclidean space (see Harvey-Lawson \cite{HL} or Xin \cite{X} for more details). Here, $\th$ is said to be a \emph{phase} of $L$. 
Taking differential of \eqref{OmLth} gives (see Dazor \cite{Da})
\begin{equation}\aligned\label{HL-JnaTh}
H=J\na^L\th,
\endaligned
\end{equation}
where $\na^L$ denotes the Levi-Civita connection of $L$. 
In this situation, $\omega_H=d\th$ on $L$, and $[d\th]\in H^1(L,\R)$ is called the \emph{Maslov class} of $L$. 
$L$ has \emph{zero-Maslov} in $U$ if 
the integration of $d\th$ vanishes on each one-cycle in $U$, which is equivalent to that the phase $\th$ is single-valued. From \eqref{Def-HS}, $L$ is Hamiltonian stationary if and only if the phase $\th$ is harmonic on $L$.
For any $p\in L$ there holds
\begin{equation}\aligned\label{PhaseThDEF}
\mathrm{Re}\left(e^{-\sqrt{-1}\th(p)}\Om\right)(\mathbf{v})\le|\mathbf{v}|\qquad \ \mathrm{for\ any}\ n\mathrm{-vector}\ \mathbf{v},
\endaligned
\end{equation}
where the equality occurs if and only if $\mathbf{v}$ represents a Lagrangian plane with the phase $\th(p)$ (mod $2\pi)$. Hence, the $n$-form $\mathrm{Re}(e^{-\sqrt{-1}\th}\Om)$ is a calibration for the constant $\th$ (see Harvey-Lawson \cite{HL}).

\subsection{Differentiability on metric and rectifiable spaces}\

Let $(\mathcal{X},d_\mathcal{X},\mu_\mathcal{X})$ be a complete metric space with Borel regular measure $\mu_\mathcal{X}$.
The measure $\mu_{\mathcal{X}}$ is said to be \emph{Ahlfors $k$-regular} at $x\in \mathcal{X}$,
if there exists a constant $K_x\ge1$ such that
\begin{equation}\aligned\label{Ahlfors-def}
K_x^{-1}r^k\le\mu_\mathcal{X}(B_r(x))\le K_xr^k
\endaligned
\end{equation} 
for all $0<r\le 1$.
The space $\mathcal{X}$ is said to be \emph{$\mu_\mathcal{X}$-rectifiable} (see Cheeger-Colding \cite{CCo3} for instance), if there exist an
integer $k_0$, a countable collection of Borel subsets, $\mathcal{C}_{k,i}\subset \mathcal{X}$ with $k\le k_0$ and bi-Lipschitz maps $\phi_{k,i}:\, \mathcal{C}_{k,i}\rightarrow\phi_{k,i}(\mathcal{C}_{k,i})\subset\R^k$ such that
\begin{itemize}
  \item [i)] $\mu_\mathcal{X}\left(\mathcal{X}\setminus\cup_{k,i}\mathcal{C}_{k,i}\right)=0$;
  \item [ii)] $\mu_\mathcal{X}$ is Ahlfors $k$-regular at each $x\in \mathcal{C}_{k,i}$.
\end{itemize}
For a Lipschitz function $f$ on $\mathcal{X}$, one can define the pointwise Lipschitz constant
\begin{equation}\aligned\label{DefLip}
\mathrm{Lip}\, f(x):=\limsup_{r\rightarrow0}\sup_{x'\in\p B_r(x)}\f{|f(x)-f(x')|}{r}=\limsup_{\mathcal{X}\ni x'\rightarrow x, x'\neq x}\f{|f(x)-f(x')|}{d_\mathcal{X}(x,x')}
\endaligned
\end{equation} 
for each $x\in \mathcal{X}$. 
Let $\mathbf{Lip}\, f=\sup_{x\in \mathcal{X}}\mathrm{Lip}\, f(x)$ denote the Lipschitz constant of $f$ on $\mathcal{X}$.

Cheeger \cite{Ch} and Cheeger-Colding \cite{CCo3} have established the differential of Lipschitz functions on the rectifiable metric space $\mathcal{X}$.
For a Borel function $f:\, \mathcal{X}\to R$, one can define $f_{\phi_{k,i}}=f\circ\phi_{k,i}^{-1}: \phi_{k,i}(\mathcal{C})\to\R$ for each $k,i$ so that
the compatibility condition
$$f_{\phi_{k,i}}\circ\phi_{k,i}=f_{\phi_{k,j}}\circ\phi_{k,j}\qquad\mathrm{holds\ on}\ \mathcal{C}_{k,i}\cap \mathcal{C}_{k,j}\qquad (\mu_{\mathcal{X}}-a.e.).$$
On the contrary, the compatibility condition for $f_{\phi_{k,i}}$ can determine a unique Borel function $f:\, \mathcal{X}\to R$ satisfying $f\circ\phi_{k,i}^{-1}=f_{\phi_{k,i}}$ for each $k,i$.
By Rademacher's theorem and Lipschitz extension ("MacShane's lemma"), for a Lipschitz $f$ on $\mathcal{X}$
one can define \emph{differential} of $f$, denoted by $d_\mathcal{X}f$, which is a $\mu_\mathcal{X}$-a.e. well-defined $L^\infty$ section of $T^*\mathcal{X}$. Here, $T^*\mathcal{X}$ is a finite dimensional cotangent vector bundle on $\mathcal{X}$ (see  \cite{Ch,CCo3}).
Correspondingly, we use $T\mathcal{X}$ to denote a finite dimensional tangent vector bundle on $\mathcal{X}$.

Now we further require the above bi-Lipschitz maps $\phi_{i,k}$ satisfying (see p. 60 in \cite{CCo3})
\begin{itemize}
  \item [iii)] For all $x\in\cup_{k,i}\mathcal{C}_{k,i}$ and all $\la>0$, there exists $\mathcal{C}_{k,i}$ such that
$x\in\mathcal{C}_{k,i}$ and the map $\phi_{k,i}:\, \mathcal{C}_{k,i}\rightarrow\phi_{k,i}(\mathcal{C}_{k,i})\subset\R^k$ is $e^{\pm\la}$-bi-Lipschitz.
\end{itemize}
Then there is a Borel set $\mathcal{X}^*_{f}\subset \mathcal{X}^*$ such that $\mu_\mathcal{X}(\mathcal{X}\setminus \mathcal{X}^*_{f})=0$, and $|d_\mathcal{X}f|(x)=\mathrm{Lip}\, f(x)$ for every $x\in \mathcal{X}^*_{f}$.
Obviously, the differential here satisfies the Leibniz rule. 
For Lipschitz functions $f,f'$ on $\mathcal{X}$, one can define the pointwise inner product $\lan d_\mathcal{X}f,d_\mathcal{X}f'\ran$ $\mu_\mathcal{X}$-a.e.  so that the parallelogram rule holds.

For a Borel set $\mathcal{B}\subset \mathcal{X}$, let $\mathbf{Lip}_c(\mathcal{B})$ denote Lipschitz functions on $\mathcal{B}$ with supports compactly contained in $\mathcal{B}$.
For a constant $q\in(0,\infty]$,
$L^q_{\mu_\mathcal{X}}(\mathcal{B})$ denote the space, equipped with norm $||\cdot||_{L^q}$, containing every Borel function $f$ with $L^q$-intergable on $\mathcal{B}$ w.r.t. $\mu_\mathcal{X}$. 
For each $f\in L^q_{\mu_\mathcal{X}}(\mathcal{B})$, we set
\begin{equation}\aligned\label{Def-W12}
||f||_{W^{1,q}}:=||f||_{L^q}+\inf_{\{f_i\}}\liminf_{i\to\infty}||\mathrm{Lip}\, f_i||_{L^q},
\endaligned
\end{equation}
where the $'\inf'$ is taken over all sequences of Lipschitz $\{f_i\}$ with $f_i\to f$ in $L^q$ sense.  
We define Cheeger $q$-energy by
$$\mathrm{Ch}_q(f):=\f12\inf_{\{f_i\}}\liminf_{i\to\infty}\int|\mathrm{Lip}\, f_i|^qd\mu_{\mathcal{X}}.$$
Let $W^{1,q}_{\mu_\mathcal{X}}(\mathcal{B})$ denote the subspace of $L^q_{\mu_\mathcal{X}}(\mathcal{B})$ consisting of functions, $f$, for which $||f||_{W^{1,q}}<\infty$, equipped with the norm $||\cdot||_{W^{1,q}}$. It's easy to see that the Sobolev space $W^{1,q}_{\mu_\mathcal{X}}(\mathcal{B})$ is complete, and it's a Banach space for all $q>1$ (see \cite{Ch}).

This definition is the same as the one defined by (weak) upper gradient.
According to the definitions of cotangent module, the differential of Lipschitz functions can be extended to each function $f\in W^{1,q}_{\mu_\mathcal{X}}(\mathcal{B})$ for $q>1$, and we can define a unique 'differential' of $f$, denoted by $d_\mathcal{X}f$ (see Gigli \cite{G} or Gigli-Pasqualetto \cite{GP}) .
Let $f_i$ be a sequence of Lipschitz functions converging to $f$ in $L^q_{\mu_\mathcal{X}}$ sense satisfying $||f||_{W^{1,q}}=\lim_{i\to\infty}||f_i||_{W^{1,q}}$. If there is a 1-form $\e_f\in L^q_{\mu_\mathcal{X}}(T^*\mathcal{B})$ so that $d_\mathcal{X}f_{i}\rightharpoonup \e_f$, then $d_\mathcal{X}f=\e_f$ (see [\cite{G}, Theorem 2.2.9]).
By Mazur's theorem, up to a choice of $f_i$ we can assume $d_\mathcal{X}f_i\to d_\mathcal{X}f$ in the strong  $L^q_{\mu_\mathcal{X}}$ sense. 
We write $(f_i,d_\mathcal{X}f_i)\to(f,d_\mathcal{X}f)$ in (the strong) $L^q_{\mu_\mathcal{X}}$ sense if $f_i\to f$ and $d_\mathcal{X}f_i\to d_\mathcal{X}f$ both in (the strong) $L^q_{\mu_\mathcal{X}}$ sense.
Obviously, we can extend the above norms and differentials from real functions to complex functions. However, when we talk about a function, we always mean it real unless we emphasize it complex. 
\begin{remark}
In general, $W^{1,q}_{\mu_\mathcal{X}}(\mathcal{B})$ is not reflexive (see Theorem 2.1.5 in \cite{G}). As a result, the above sequence $d_\mathcal{X}f_i$ may not be weakly compact. Hence, we need assume $d_\mathcal{X}f_{i}\rightharpoonup \e_f$.
\end{remark}

For a multi-valued function $u$ on $\mathcal{B}$, there exists a smallest constant $k>0$ so that $e^{\sqrt{-1}ku}$ is single-valued.
For $e^{\sqrt{-1}ku}\in W^{1,2}_{\mu_\mathcal{X}}(\mathcal{B})$,
we say $u$ \emph{harmonic} on $\mathcal{B}$ (w.r.t. $\mu_\mathcal{X}$) if 
\begin{equation}\aligned\label{dz_idphi=0}
\int_{\mathcal{X}}e^{-\sqrt{-1}ku}\lan d_\mathcal{X}e^{\sqrt{-1}ku},d_\mathcal{X}\phi\ran d\mu_\mathcal{X}=0\qquad\qquad \mathrm{for\ each}\ \phi\in \mathbf{Lip}_c(\mathcal{B}).
\endaligned
\end{equation}
If $u$ is single-valued, then \eqref{dz_idphi=0} reduces to
\begin{equation}\aligned\label{du_idphi=0}
\int_{\mathcal{X}}\lan d_\mathcal{X}u,d_\mathcal{X}\phi\ran d\mu_\mathcal{X}=0\qquad\qquad \mathrm{for\ each}\ \phi\in \mathbf{Lip}_c(\mathcal{B}).
\endaligned
\end{equation}

\subsection{Lagrangian currents and varifolds}\

Let $M$ be a $q$-dimensional complete Riemannian manifold, and $U\subset M$ be a bounded open subset.
From Nash's theorem, $U$ can be properly isometrically embedded in $\R^{m}$ for some integer $m\ge q$ (see G\"unther \cite{Gu} or Andrews \cite{Ben} for further results on $m$).
For a set $S$ in $M\cap U$ and an integer $k\in[0,q]$, $S$ is said to be \emph{countably $k$-rectifiable}
if $S\subset S_0\cup\bigcup_{j=1}^\infty Q_j(F_j(\R^k))$, where $\mathcal{H}^k(S_0)=0$,  $F_j:\, \R^k\rightarrow \R^{m}$ are Lipschitz mappings for all integers $j\ge1$, and $Q_j$ are orthogonal transformations on $\R^{m}$.
Let $G_{k,m}$ denote the Grassmann manifold including all the $k$-dimensional (unoriented) subspaces of $\R^{m}$. 

A $k$-varifold $V$ in $U$ is a Radon measure on
$$G_{k}(U)=\{(p,P)|p\in U, P\in G_{k,m}\cap T_pM\}.$$
Let $\Pi:\R^{m}\times G_{k,m}\to\R^{m}$ denote the projection $\Pi(p,P)=p$. There corresponds a Radon measure $\mu_V$ on $U$ defined by
$\mu_V(\mathbf{W})=V(\Pi^{-1}(\mathbf{W}))$ for each Borel set $\mathbf{W}\subset G_{k}(U).$
We further assume $\mathcal{H}^k(S)<\infty$. Let $\vartheta$ be a positive locally $\mathcal{H}^k$-integrable function on $S$.
The associated varifold $\vartheta|S|$ is called a \emph{rectifiable $k$-varifold} with 
a Radon measure $\mu_\vth$ defined by $\mathcal{H}^k\llcorner\vartheta$, namely,
$$\mu_\vth(W)=\int_{W\cap S}\vartheta d\mathcal{H}^k\qquad \mathrm{for\ each\ Borel\ set}\ W\subset U.$$ 
For $\vartheta|S|$, there is a corresponding $k$-varifold $V$ in $U$ defined by
$$V(\mathbf{W})=\mu_\vth(\Pi(TS\cap\mathbf{W}))\qquad \mathrm{for\ each\ Borel\ set}\ \mathbf{W}\subset G_{k}(U).$$
The above $\vartheta$ is called the \emph{multiplicity} function of $V$, and $V$ is said to be an \emph{integral varifold} if $\vartheta$ is integer-valued. We call $|S|$ the varifold associated with the set $S$, and the multiplicity of $|S|$ is equal to one on $S$. If $S$ is the smallest relatively closed set in $U$ so that $\mu_\vth(W)>0$ for any open $W$ with $W\cap S\neq\emptyset$, then $S$ is called the \emph{support} of $V$.

Let $\mathcal{D}^k(U)$ denote the set including all smooth $k$-forms on $U$ with compact supports in $U$. Let $\mathcal{D}_k(U)$ be the set of $k$-currents in $U$, which are continuous linear functionals on $\mathcal{D}^k(U)$. 
$T\in\mathcal{D}_k(U)$ is said to be an \emph{integer multiplicity current} if it can be expressed as
$$T(\e)=\int_S\vth\lan \e,\vec{T}\ran d\cal{H}^k\qquad \mathrm{for\ each}\ \e\in \mathcal{D}^k(U),$$
where $S$ is a countably $k$-rectifiable subset of $U$, $\vth$ is a locally $\mathcal{H}^k$-integrable positive integer-valued function on $S$, and $\vec{T}$ is an orientation on $S$, i.e.,  $\vec{T}(x)$ is a $k$-vector representing the approximate tangent space $T_xS$ for $\mathcal{H}^k$-a.e. $x$. Let $|T|$ denote the varifold associated with $T$, i.e., $|T|=\vth|S|$, and $\mu_T$ be the Radon measure $\mathcal{H}^k\llcorner\vartheta$. We further assume that $S$ is the support of $|T|$, and denote $T=(S,\vth,\vec{T})$.

For a Borel set $W$ in $U$, one defines $T\llcorner W\in\mathcal{D}_k(U)$ by
\begin{equation}\aligned
T\llcorner W(\e)=\int_{W}\lan\e,\vec{T}\ran d\mu_T\qquad \mathrm{for\ any}\ \e\in\mathcal{D}^{k}(U),
\endaligned
\end{equation}
and the mass of $T$ on $W$ by
$$\mathbf{M}(T\llcorner W)=\sup_{|\e|_U\le1,\e\in\mathcal{D}^k(U)}
T\llcorner W(\e)$$
with $|\e|_U=\sup_{U}\lan \e,\e\ran^{1/2}$. 
For a sequence of currents $\{T_i\}\subset \mathcal{D}_{k}(U)$, we say $T_i\rightharpoonup T$ if $T_i(\e)\to T(\e)$ for each $\e\in\mathcal{D}^{k}(U)$.
Let $\p T$ be the boundary of $T$ defined by $\p T(\e')=T(d\e')$ for any $\e'\in \mathcal{D}^{k-1}(U)$. 
$T$ is said to be \emph{integral} if both $T$ and $\p T$ are integer multiplicity currents.
Morover, we can extend $T$ on complex smooth $k$-forms on $U$ with compact supports in $U$ by letting
\begin{equation}\aligned
T(\e_1+\sqrt{-1}\e_2)=T(\e_1)+\sqrt{-1}T(\e_2) \qquad\qquad \mathrm{for\ each}\ \e_1,\e_2\in\mathcal{D}^k(U).
\endaligned
\end{equation}

In \cite{BG}, Bombieri-Giusti introduced a concept 'indecomposable' for integral currents as follows.
\begin{definition}\label{B-Gindec}
For an integer $k\ge1$, an integral current $T\in\mathcal{D}_k(M)$ is \emph{decomposable} in $U$ if there exist integral currents $T_1,T_2\in\mathcal{D}_k(U)$ with $T_1\llcorner U,T_2\llcorner U\neq0$ such that
\begin{equation}\aligned\nonumber
\mathbf{M}(T\llcorner W)=\mathbf{M}(T_1\llcorner W)+\mathbf{M}(T_2\llcorner W),\quad \mathbf{M}(\p T\llcorner W)=\mathbf{M}(\p T_1\llcorner W)+\mathbf{M}(\p T_2\llcorner W)
\endaligned
\end{equation}
for any $W\subset\subset U$. Here, $T_1,T_2$ are called \emph{components} of $T\llcorner U$. On the contrary, $T$ is said to be \emph{indecomposable} in $U$.
\end{definition}

Let $\tilde{M}$ be a $\tilde{q}$-dimensional Riemannian manifold with $\tilde{q}\ge q$, and $f:\ U\rightarrow \tilde{M}$ be a $C^1$-mapping. 
Let $f_*\vec{T}$ denote the push-forward of $\vec{T}$, which is an orientation of $f(S)$ in $\tilde{M}$. Then one can define $f_\#T\in \mathcal{D}_k(\tilde{M})$ by letting
\begin{equation}\label{PushforwfT}
f_\#T(\e)=\int_S\vth\lan \e\circ f,f_*\vec{T}\ran d\mathcal{H}^k
\end{equation}
for each $\e\in \mathcal{D}^k(\tilde{M})$.
It's clear that $f_\#T$ is an integer multiplicity current in $\tilde{M}$.
Moreover, for a rectifiable $k$-varifold $V=\vth|S|$ in $U$, we define the image varifold $f_\#V$ by $\tilde{\vth}|f(S)|$, where $\tilde{\vth}(y)$ is defined on $f(S)$ by $\int_{f^{-1}(y)\cap S}\vth d\mathcal{H}^0$.

For a constant $\ep>0$ and a compact set $K\subset U$, let $\{F_t\}_{-\ep<t<\ep}$ denote a 1-parameter family of smooth diffeomorphism of $U$ so that $F_0=1_{U}$ and $F_t\big|_{U\setminus K}=1\big|_{U\setminus K}$ for any $t\in(-\ep,\ep)$.
Hence, there is a smooth vector field $X$ on $U$ with compact support satisfying
\begin{equation*}\aligned
\f{d}{dt}\bigg|_{t=0}F_t=X.
\endaligned
\end{equation*}
From Simon \cite{S}, for a rectifiable $k$-varifold $V=\vth|S|\subset U$ we have
\begin{equation}\aligned\label{FirstVar}
\mathbf{M}(F_{t\#}(V))=\int_{U} (1+t\mathrm{div}_S X+t^2\Phi_t)d\mu_V,
\endaligned
\end{equation}
where $\mathrm{div}_S$ denotes the divergence restricted on $S$ $\mu_V$-a.e., and $\sup_{U,|t|<\ep}|\Phi_t|$ is bounded by a constant independent of $t$, but depends on the geometry of $U$ (in a small neighborhood of $S$). 
For a locally $\mu_V$-integrable vector field $H$ on $U$,
we say that $V$ has (generalized) mean curvature $H$ if (see \cite{LYa,S} for the Euclidean space)
\begin{equation}\aligned\label{Def-GMC}
\int_U \mathrm{div}_SY d\mu_V=-\int_U\lan H,Y\ran d\mu_V
\endaligned
\end{equation}
for any $C^1$ tangent vector field $Y$ on $U$ with compact support.

For each $\phi\in L^1_{\mu_V}(U)$, we say that $\phi$ has \emph{bounded variation} in $U$ w.r.t. $\mu_V$ if there is a constant $c>0$ so that
\begin{equation}\aligned\label{DEF-naphi}
\int_U\phi(\lan H,X\ran+\div_SX)d\mu_V\le c\sup|X|
\endaligned
\end{equation}
for each smooth vector field $X$ on $U$ with compact support, where $S$ is the support of $V$. 
Let $BV_{\mu_V}(U)$ denote the space of all functions in $L^1_{\mu_V}(U)$ with bounded variation in $U$ w.r.t. $\mu_V$.
From Riesz Representation Theorem, there are a Radon measure $\mu_\phi$ and a $\mu_\phi$-measurable vector-valued function $\nu_\phi$ with $|\nu_\phi|=1$ $\mu_\phi$-a.e. so that
\begin{equation}\aligned\label{phinuphina}
\int_U\phi(\lan H,X\ran+\div_SX)d\mu_V=-\int_U \lan X,\nu_\phi\ran d\mu_\phi.
\endaligned
\end{equation}
We denote $\na^S\phi\, d\mu_V=\nu_\phi d\mu_\phi$, then $d\mu_\phi=|\na^S\phi| d\mu_V$. In particular, if $\phi\in C^1(U)$, then \eqref{Def-GMC} infers that $\na^S\phi$ is just the gradient of $\phi$ restricted on $S$ $\mu_V$-a.e..

Let $(M,\omega)$ be a $2n$-dimensional symplectic manifold with almost K\"ahler structure given by a Riemannian metric $g$, compatible with a symplectic 2-form $\omega$ and an almost complex structure $J$ on $M$. For an $n$-varifold $V$ in $U$, we say $V$  \emph{Lagrangian} if 
\begin{equation}\aligned\label{Lag-varifold}
\int_{G_n(U)}|\lan\omega\wedge\e,P\ran|(p) dV(p,P)=0\qquad \mathrm{for\ any}\ \e\in\mathcal{D}^{n-2}(U).
\endaligned
\end{equation}
An $n$-current $T$ in $U$ is said to be \emph{Lagrangian} in $U$ if $|T|$ is Lagrangian. Namely,
\begin{equation}\aligned\label{Lag-current}
T(\omega\wedge\e)=0\qquad \mathrm{for\ any}\ \e\in\mathcal{D}^{n-2}(U).
\endaligned
\end{equation}
Given a varifold $V$ in $U$,
we say $V$ \emph{Hamiltonian stationary Lagrangian} (HSL) if $V$ is Lagrangian and stationary under variations by Hamiltonian vector fields with compact supports, i.e.,
\begin{equation}\aligned\label{HS-varifold}
\int_{G_n(U)} \mathrm{div}_P(J\na f)dV(p,P)=0\qquad \mathrm{for\ every}\ f\in C^\infty_c(U).
\endaligned
\end{equation}
\begin{remark}
For the rectifiable case in $\C^n$, 
the definitions of Lagrangian varifolds and Lagrangian currents are the same as in \cite{Ne} by Neves. For a rectifiable $n$-varifold $V=\vth|S|$ in $U$, \eqref{Lag-varifold} is equivalent to $\omega\big|_{S}=0$ $\mathcal{H}^n$-a.e.. 
\end{remark}

Now we further assume that $(M,\omega)$ is a Ricci-flat K\"ahler manifold with a holomorphic $n$-form $\Om$ on $U$ satisfying $(-1)^{n^2/2}\Om\wedge\overline{\Om}=2^{n}\omega^n/n!$. 
For an integral Lagrangian current $T=(L,\vth,\vec{T})$ in $U$, compared with \eqref{OmLth} there is a measurable (multi-valued) function $\th$ on $L$ so that 
\begin{equation}\aligned\label{OmLthae}
\Omega\big|_L=e^{\sqrt{-1}\th}\mathrm{vol}_L\qquad \mu_T-a.e.,
\endaligned
\end{equation}
where $\th$ is called a \emph{phase} of $T$, and $\mathrm{vol}_L$ is the dual $n$-form of $\vec{T}$. 
\begin{definition}\label{Def-harm-curr}
We say that $T=(L,\vth,\vec{T})$ has a \emph{harmonic phase} $\th$ if for every $x\in L$ there is an open subset $U'$ of $U$ with $x\in U'$ so that each indecomposable component $T'$ of $T\llcorner U'$ satisfies that
$e^{\sqrt{-1}\th}\in W^{1,2}_{\mu_T}(\mathrm{spt} T')$ and $\th$ is harmonic on $\mathrm{spt}\, T'$ w.r.t. $\mu_{T}$.
\end{definition} 
This definition gives a generalization of oriented immersed Hamiltonian stationary Lagrangians in $M$.
In particular, Schoen-Wolfson cones do not satisfy Definition \ref{Def-harm-curr}.
Let $\Si$ be a smooth oriented submanifold or a local Lipschitz graph in $U$, we use $[|\Si|]$ to denote the current $(\Si,\chi_{_\Si},\vec{\Si})$ with an orientation $\vec{\Si}$.
We say that $\Si$ has a \emph{harmonic phase} if $[|\Si|]$ has a harmonic phase.

$\mathbf{Notational\ conventions}.$ 
Let $M$ be a complete Riemannian manifold, and $U\subset M$ be open. We use $\mathfrak{X}(U)$ to denote all $C^1$-continuous vector fields defined on $U$ valued in the tangent space $T_xU=T_xM$ for each $x\in U$. Let $\mathfrak{X}_c(U)$ denote a subset of $\mathfrak{X}(U)$ so that every $X\in \mathfrak{X}(U)$ has compact support in $U$.
For a subset $K\subset M$, let $B_r(K)$ denote $r$-tubular neighborhood of $K$ in $M$. 
For each integer $k>0$, let $\omega_k$ denote the volume of $k$-dimensional unit Euclidean ball, and $\mathcal{H}^k$ denote the $k$-dimensional Hausdorff measure.
When we write an integation on a Borel set $\mathcal{B}$ w.r.t. some Radon measure, we sometimes omit $\mathcal{B}$ if it is the maximal set where the measure is defined.

\section{Phase and mean curvature for Lagrangian currents}

Let $U$ be an open set of a Ricci-flat K\"ahler manifold $(M^n,g,J,\omega)$, and $\Om$ be a holomorphic $n$-form in $U$ satisfying $\Om\wedge\overline{\Om}=(-1)^{-n^2/2}2^{n}\omega^n/n!$.
Let $T=(L,\vth,\vec{T})$ be an integer multiplicity Lagrangian current in $U$ with phase $\th$ so that $e^{\sqrt{-1}\th}\in W^{1,q}_{\mu_T}(L)$ for some $q>1$. 
Here, $\mu_T$ denotes the volume element of $T$. Without loss of generality, we may assume $\p T\llcorner U=0$ after shrinking $U$. 
Let $\na^L$ denote Levi-Civita connection on $L$ $\mu_T$-a.e. since $L$ is countably $n$-rectifiable.

For each vector field $Y\in \mathfrak{X}_c(U)$, there are
a constant $\vep>0$, a compact set $K\subset U$ and a 1-parameter family of smooth diffeomorphisms $\{F_t\}_{-\ep<t<\ep}$ so that $F_0=1_{U}$, $F_t\big|_{U\setminus K}=1\big|_{U\setminus K}$ for any $t\in(-\vep,\vep)$ and
\begin{equation}\aligned
\f{d}{dt}\bigg|_{t=0}F_t=Y.
\endaligned
\end{equation}
Put $F(t,\cdot)=F_t$. 
We consider an $n$-form $\omega_{F}$ defined on $(-\ep,\ep)\times L$ by
$$\omega_{F}:=\mathrm{Re}\left(e^{-\sqrt{-1}\th}F^*\Om\right).$$
From Lemma \ref{Lip-to-smooth} in Appendix I and Mazur's theorem,
there exists a sequence of complex smooth functions $\tilde{\z}_\ell$ on $U$ with 
$\tilde{\z}_\ell\to e^{\sqrt{-1}\th}$ in the strong $L^q_{\mu_T}$ sense and 
$$\limsup_{i\to\infty}\left(\int_U|\na^L\tilde{\z}_\ell|^qd\mu_T\right)^{1/q}=||e^{\sqrt{-1}\th}||_{W^{1,q}}-\mu_T(U)<\infty.$$
We define a complex function $\z_\ell=e^{\sqrt{-1}\arg\tilde{\z}_\ell}$. 
It's clear that $|\na^L\z_\ell|\le|\na^L\tilde{\z}_\ell|$ and
\begin{equation}\aligned
\lim_{\ell\to\infty}\int\left|\z_\ell-e^{\sqrt{-1}\th}\right|^qd\mu_T=\lim_{\ell\to\infty}\int\left|\tilde{\z}_\ell-e^{\sqrt{-1}\th}\right|^qd\mu_T=0.
\endaligned
\end{equation}

For each $\ell>0$, we define a smooth $n$-form 
$$\omega_{F,\ell}:=\mathrm{Re}\left(\bar{\z}_{\ell}F^*\Om\right)\qquad\mathrm{on}\ (-\ep,\ep)\times U.$$
Let $\mathbf{T}_t:=[0,t]\times T$ and $T_\tau=\{\tau\}\times T$ for each $0\le\tau\le t<\ep$, 
then $\mathbf{T}_t$ is an integer multiplicity one $(n+1)$-current in $\R\times U$ with
$$\p \mathbf{T}_t=T_t-T_0+[0,t]\times\p T.$$
By the definition of $F$, it infers
\begin{equation}\aligned\label{bTtdwFlep}
\mathbf{T}_\tau(d\omega_{F,\ell})=\p\mathbf{T}_\tau(\omega_{F,\ell})=T_\tau(\omega_{F,\ell})-T_0(\omega_{F,\ell}).
\endaligned
\end{equation}
Let $\vec{T}_\tau$ denote the orientation of $T_\tau$, and $\p_t$ denote the orientation of $[0,1]$. 
By the definition of $\omega_{F,\ell}$, we have
\begin{equation}\aligned\label{TtdwFlep}
T_\tau(\omega_{F,\ell})=&\int\lan\omega_{F,\ell},\vec{T}_\tau\ran d\mu_{T_\tau}=\int\mathrm{Re}\left(\bar{\z}_{\ell}\lan F_\tau^*\Om,\vec{T}_\tau\ran\right) d\mu_{T_\tau}\\
=&\int\mathrm{Re}\left(\bar{\z}_{\ell}\left\lan \Om\big|_{F_\tau},(F_\tau)_*\vec{T}_\tau\right\ran\right) d\mu_{T_\tau}\\
\le&\int\left|(F_\tau)_*\vec{T}_\tau\right| d\mu_{T_\tau}=\mathbf{M}(F_{\#}T_\tau).
\endaligned
\end{equation}

There exists a zero $\mu_T$-measure set $L_0\subset L$ so that the approximate tangent space of $T_x$ exists, and is an $n$-plane with multiplicity $\vth(x)$ for each $x\in L\setminus L_0$. In particular, all such planes live in $G_{n,m}$ since $U$ is properly embedded in $\R^m$. 
Hence, there is a countable collection of Borel sets $\{\mathcal{C}_k\}\subset L\setminus L_0$ satisfying $\mu_T(L\setminus\cup_k\mathcal{C}_k)=0$, and there is a bi-Lipschitz map $\phi_{k}:\, \mathcal{C}_{k}\rightarrow\phi_{k}(\mathcal{C}_{k})\subset\R^n$ with Lipschitz norm $\le1$ for each $k$.
Here, the distance function on $L$ is defined by
\begin{equation*}\aligned
d_L(x,y):=\lim_{\ep\to0}\inf\left\{\sum_{i=1}^jd(x_i,x_{i-1}):\, \{x_i\}_{i=1}^j\subset L, d(x_{i-1},x_i)<\ep,x_0=x,x_j=y\right\}
\endaligned
\end{equation*}
for each $x,y\in L$,
where $d$ is the distance function on $M$.
Hence, the metric space $(L,d_L,\mu_T)$ is $\mu_T$-rectifiable (in the intrinsic sense). 

\begin{lemma}
There is a constant $c_F>0$ depending only on $n$ and the $C^2$-norm of $F$ on $M$ so that for each $t\in[0,\ep)$
\begin{equation}\aligned\label{WtdwLLt}
\left|\f{\mathbf{T}_t(d\omega_{F,\ell})}t+\int\left\lan Y,\mathrm{Re}\left(e^{\sqrt{-1}\th}\left(\na^L\bar{\z}_{\ell}+\sqrt{-1} J\na^L\bar{\z}_{\ell}\right)\right)\right\ran d\mu_T\right|\le\f{c_Ft}2\int |\na^L\z_{\ell}| d\mu_T.
\endaligned
\end{equation}
\end{lemma}
\begin{proof}
Given a point $p\in \mathcal{C}_k$, we choose an orthonormal basis $\{e_i\}_{1\le i\le n}$ of $T_pL$ so that $\vec{T}\big|_p=e_1\wedge\cdots\wedge e_n$. 
There is a local Lipschitz coordinate system $(\mathbf{x}_1,\cdots,\mathbf{x}_n):\mathcal{C}_k \to\R^n$ so that $\f{\p}{\p\mathbf{x}_i}=e_i$ at $p$.
There are a small $\ep_0>0$, and a complex coordinate chart $(B_{\ep_0}(p),Z)$ so that $B_{\ep_0}(p)\subset U$, $Z=(z_1,\cdots,z_n): B_{\ep_0}(p)\to\Phi(B_{\ep_0}(p))\subset\C^n$ is bi-holomorphic with $z_i=(x_i,y_i)$ and $\f{\p}{\p z_i}=\f{\p}{\p x_i}+\sqrt{-1}\f{\p}{\p y_i}$. 
Moreover, we can assume $\f{\p}{\p x_i}=e_i$ at $p$ since $T_pL$ is a Lagrangian plane in $\C^n$.

From the definition of $\th$, there is a holomorphic function $h_\Om$ on $U$ with $h_\Om(p)=1$ so that
$$\Om=e^{\sqrt{-1}\th(p)}h_\Om dz_1\wedge\cdots\wedge dz_n\qquad \text{on}\ B_{\ep_0}(p).$$
Let $(\mathbf{z}_1,\cdots,\mathbf{z}_{n})=Z(F(t,\mathbf{x}_1,\cdots,\mathbf{x}_n))\in Z(B_{\ep_0}(p))$, then
\begin{equation}\aligned
d\mathbf{z}_j\bigg|_{(p,0)}=\left\lan \f{d}{dt}\bigg|_{(p,0)}F_t,\f{\p}{\p z_j}\right\ran dt+\sum_i\f{\p\mathbf{z}_j}{\p \mathbf{x}_i}\bigg|_{(p,0)} d\mathbf{x}_i=\left\lan Y,\f{\p}{\p z_j}\right\ran dt+d\mathbf{x}_j.
\endaligned
\end{equation}
From
\begin{equation}\aligned
F^*\Om\big|_{\mathcal{C}_k\times(-\ep,\ep)}=e^{\sqrt{-1}\th(p)}h_\Om\circ F d\mathbf{z}_1\wedge\cdots\wedge d\mathbf{z}_n\qquad \text{for each } k,
\endaligned
\end{equation}
on $\cup_k\mathcal{C}_k\times[0,\ep)$ we get
\begin{equation}\aligned\label{dwFt}
\lan d\omega_{F,\ell},\p_t\wedge\vec{T}\ran=&\left\lan\mathrm{Re}\left(d\bar{\z}_{\ell}\wedge F^*\Om\right),\p_t\wedge \vec{T}\right\ran\\
=&\mathrm{Re}\left(e^{\sqrt{-1}\th(p)}h_\Om\circ F\left\lan d\bar{\z}_{\ell}\wedge d\mathbf{z}_1\wedge\cdots\wedge d\mathbf{z}_n,\p_t\wedge \vec{T}\right\ran\right).
\endaligned
\end{equation}

Put $\p_{x_j}=\f{\p}{\p x_j}$ and $\p_{y_j}=\f{\p}{\p y_j}$ for each $j$. At $(p,0)$ we have 
\begin{equation}\aligned\label{dwF}
\lan d\omega_{F,\ell},\p_t\wedge\vec{T}\ran=&\sum_j(-1)^{j-1}\mathrm{Re}\Big(e^{\sqrt{-1}\th}\lan Y,\p_{x_j}+\sqrt{-1}\p_{y_j}\ran\\ 
& \left\lan d\bar{\z}_{\ell}\wedge dt\wedge d\mathbf{x}_1\wedge\cdots \wedge\widehat{d\mathbf{x}_j}\wedge\cdots\wedge d\mathbf{x}_n,\p_t\wedge e_1\wedge\cdots\wedge e_n\right\ran\Big)\\
=&-\sum_j\mathrm{Re}\Big(e^{\sqrt{-1}\th}\lan Y,\p_{x_j}+\sqrt{-1}\p_{y_j}\ran \p_{x_j}\bar{\z}_{\ell} \Big)\\
=&-\mathrm{Re}\left(e^{\sqrt{-1}\th}\lan Y,\na^L\bar{\z}_{\ell}+\sqrt{-1} J\na^L\bar{\z}_{\ell}\ran\right)\\
=&-\left\lan Y,\mathrm{Re}\left(e^{\sqrt{-1}\th}\left(\na^L\bar{\z}_{\ell}+\sqrt{-1} J\na^L\bar{\z}_{\ell}\right)\right)\right\ran.
\endaligned
\end{equation}
Let $\Psi_{\ell}$ be a function defined on $\cup_k\mathcal{C}_k\times[0,\ep)$ by
$$\Psi_{\ell}(p,\tau)=\lan d\omega_{F,\ell},\p_t\wedge\vec{T}\ran\big|_{(p,\tau)}-\lan d\omega_{F,\ell},\p_t\wedge\vec{T}\ran\big|_{(p,0)}.$$
From \eqref{dwFt}, we conclude that 
\begin{equation}\aligned\label{Est-Psi}
|\Psi_{\ell}|(p,\tau)\le\int_0^\tau\left|\f{\p}{\p s}\lan d\omega_{F,\ell},\p_t\wedge\vec{T}\ran\big|_{(p,s)}\right|ds\le c_F\tau|\na^L\z_{\ell}|(p),
\endaligned
\end{equation}
where $c_F>0$ is a constant depending only on $n$ and the $C^2$-norm of $F$ on $M$.
From
\begin{equation}\aligned
\mathbf{T}_t(d\omega_{F,\ell})=&\int_0^t\int\lan d\omega_{F,\ell},\p_t\wedge\vec{T}\ran d\mu_Td\tau
\endaligned
\end{equation}
and \eqref{dwF}, we get
\begin{equation}\aligned
&\left|\f1t\mathbf{T}_t(d\omega_{F,\ell})+\int\left\lan Y,\mathrm{Re}\left(e^{\sqrt{-1}\th}\left(\na^L\bar{\z}_{\ell}+\sqrt{-1} J\na^L\bar{\z}_{\ell}\right)\right)\right\ran d\mu_T\right|\\
=&\f1t\left|\int_0^t\int \Psi_{\ell}(p,\tau) d\mu_Td\tau\right|
\le\f{c_F}t\int_0^t\tau\int |\na^L\z_{\ell}| d\mu_Td\tau=\f{c_Ft}2\int |\na^L\z_{\ell}| d\mu_T.
\endaligned
\end{equation}
This completes the proof.
\end{proof}

Analog to the smooth case, we also use $d_L$ to denote the differential on $L$. Let
$d_L\th$ be the differential of $\th$, and $\na^L\th\in L^q_{\mu_T}(TL)$ be the dual vector field (on $L$) of $d_L\th$, i.e., $d_L\th(Y)=\lan \na^L\th,Y\ran$ for each vector field $Y$.
Since it's unclear that $W^{1,q}_{\mu_T}(L)$ is reflexive or not (see Theorem 2.1.5 in \cite{G}),
in the following we prove
\begin{equation}\aligned\label{con*th*}
d_L\arg\tilde{\z}_\ell\to d_L\th\qquad \text{or} \qquad d_L\tilde{\z}_\ell\to d_Le^{\sqrt{-1}\th}\qquad \text{in the weak}\  L^q_{\mu_T}\ \text{sense}
\endaligned
\end{equation}
using a calibration argument.

\begin{theorem}\label{thW12-H}
If $T=(L,\vth,\vec{T})$ is an integer multiplicity Lagrangian current  in $U$ with $\p T\llcorner U=0$ and phase $\th$  satisfies $e^{\sqrt{-1}\th}\in W^{1,q}_{\mu_T}(L)$, then $|T|$ has generalized mean curvature $H=J\na^L\th\in L^q_{\mu_T}(TL)$.
\end{theorem}
\begin{proof}
From \eqref{FirstVar}, \eqref{bTtdwFlep} and \eqref{TtdwFlep}, for each $Y\in \mathfrak{X}_c(U)$ it follows that 
\begin{equation}\aligned
\mathbf{M}(F_\#T_t)=\int (1+t\mathrm{div}_L Y+t^2\Phi_t)d\mu_T
\ge T_t(\omega_{F,\ell})=T_0(\omega_{F,\ell})+\mathbf{T}_t(d\omega_{F,\ell}),
\endaligned
\end{equation}
which implies
\begin{equation}\aligned
\int (\mathrm{div}_L Y+t\Phi_t)d\mu_T\ge \f1t\limsup_{\ell\to\infty}\mathbf{T}_t(d\omega_{F,\ell}).
\endaligned
\end{equation}
From the definition $\z_\ell=e^{\sqrt{-1}\arg\tilde{\z}_\ell}$, we have
\begin{equation}\aligned
\na^L\bar{\z}_{\ell}+\sqrt{-1} J\na^L\bar{\z}_{\ell}=e^{-\sqrt{-1}\arg\tilde{\z}_\ell}\left(-\sqrt{-1}\na^L\arg\tilde{\z}_\ell+ J\na^L\arg\tilde{\z}_\ell\right)\qquad \mu_T-a.e..
\endaligned
\end{equation}
From H\"older inequality and the choice of $\tilde{\z}_\ell$, it follows that
\begin{equation}\aligned
&\liminf_{\ell\to\infty}\int\left\lan Y,\mathrm{Re}\left(e^{\sqrt{-1}\th}\left(\na^L\bar{\z}_{\ell}+\sqrt{-1} J\na^L\bar{\z}_{\ell}\right)\right)\right\ran d\mu_T\\
=&\liminf_{\ell\to\infty}\int\left\lan Y,J\na^L\arg\tilde{\z}_\ell\right\ran d\mu_T.
\endaligned
\end{equation}
Combining \eqref{WtdwLLt}, $t\to0$ infers that
\begin{equation}\aligned
\int \mathrm{div}_L Y d\mu_T \ge&\limsup_{t\to0}\left(\f1t\limsup_{\ell\to\infty}\mathbf{T}_t(d\omega_{F,\ell})\right)\\
=&-\liminf_{\ell\to\infty}\int\left\lan Y,J\na^L\arg\tilde{\z}_\ell\right\ran d\mu_T.
\endaligned
\end{equation}
By taking $-Y$ instead of $Y$, we have
\begin{equation}\aligned
\int \mathrm{div}_L(-Y) d\mu_T \ge&
-\liminf_{\ell\to\infty}\int\left\lan -Y,J\na^L\arg\tilde{\z}_\ell\right\ran d\mu_T\\
=&\limsup_{\ell\to\infty}\int\left\lan Y,J\na^L\arg\tilde{\z}_\ell\right\ran d\mu_T.
\endaligned
\end{equation}
The above two inequalities give
\begin{equation}\aligned\label{divJnaLL}
\int \mathrm{div}_L Y d\mu_T=-\lim_{\ell\to\infty}\int\left\lan Y,J\na^L\arg\tilde{\z}_\ell\right\ran d\mu_T.
\endaligned
\end{equation}
By H\"older inequality, 
\begin{equation}\aligned
\left|\int \mathrm{div}_L Yd\mu_T\right|\le&\limsup_{\ell\to\infty}\left(\int|\na^L\arg\tilde{\z}_\ell|^qd\mu_T\right)^{1/q}\left(\int|Y|^{q'}d\mu_T\right)^{1/q'}\\
\le&||e^{\sqrt{-1}\th}||_{W^{1,q}}\left(\int|Y|^{q'}d\mu_T\right)^{1/q'},
\endaligned
\end{equation}
where $q'$ is the constant satisfying $1/q+1/{q'}=1$.
From Riesz Representation Theorem, there is a vector-valued $L^q$-function $H$ on $L$ so that
\begin{equation}\aligned\label{GMC}
\int \mathrm{div}_L Yd\mu_T=-\int \lan H,Y\ran d\mu_T.
\endaligned
\end{equation}
Here, $H$ is called generalized mean curvature of $L$. With \eqref{divJnaLL}, we get
\begin{equation}\aligned
\lim_{\ell\to\infty}\int\left\lan (H-J\na^L\arg\tilde{\z}_\ell),Y\right\ran d\mu_T=0\qquad \text{for any}\ Y\in \mathfrak{X}_c(U).
\endaligned
\end{equation}
In other words, $\na^L\arg\tilde{\z}_\ell$ converges to $-JH$ in the weak $L^q_{\mu_T}$ sense. 
From \cite{G,GP} and $\z_\ell=e^{\sqrt{-1}\arg\tilde{\z}_\ell}$, we conclude that 
\begin{equation}\aligned\label{dLargxielldLth}
d_L\arg\tilde{\z}_\ell=-\sqrt{-1}e^{-\sqrt{-1}\arg\tilde{\z}_\ell}d_L\z_\ell\to-\sqrt{-1}e^{-\sqrt{-1}\th}d_Le^{\sqrt{-1}\th}=d_L\th
\endaligned
\end{equation}
 in $L^q_{\mu_T}$ sense.
This means $\na^L\th=-JH$ $\mu_T$-a.e..
We complete the proof.
\end{proof}
\begin{remark}
The converse of Theorem \eqref{thW12-H} may not be true. For example, let $\mathfrak{C}$ be the Lagrangian current $[|\{(x,y)\in\R^2:\, y=x\}|]+[|\{(x,y)\in\R^2:\, y=-x\}|]$, then the varifold $|\mathfrak{C}|$ is stationary in $\R^2$, but the phase of $\mathfrak{C}$ is not in $W^{1,q}_{\mu_T}$ for each $q>1$.
\end{remark}
From \eqref{dLargxielldLth}, $\na^{L}\z_\ell^j\to\na^{L}e^{j\sqrt{-1}\th}$ in $L^q_{\mu_T}$ sense for each $j\in\Z$.
For each $X\in \mathfrak{X}_c(U)$, we replace $Y$ by $\z_\ell^jX$ in \eqref{GMC}, then letting $\ell\to\infty$ gives
\begin{equation}\aligned\label{thnaLiejthX}
j\sqrt{-1}\int_{U} e^{j\sqrt{-1}\th}\lan\na^{L}\th,X\ran d\mu_{T}=-\int_Ue^{j\sqrt{-1}\th}\left(\lan H,X\ran+\mathrm{div}_{L}X\right)d\mu_{T}.
\endaligned
\end{equation}
The above equality means that $e^{j\sqrt{-1}\th}$ has bounded variation in $U$ w.r.t. $\mu_{T}$.

From \eqref{Def-GMC} and Theorem \ref{thW12-H}, we immediately have the following equivalence.
\begin{corollary}
Let $T=(L,\vth,\vec{T})$ be an integer multiplicity Lagrangian current in $U$ with $\p T\llcorner U=0$ and phase $\th$ satisfying $e^{\sqrt{-1}\th}\in W^{1,2}_{\mu_T}(L)$. Then $\th$ is harmonic on $L$ w.r.t. $\mu_T$ if and only if $|T|$ is HSL.
\end{corollary}

\begin{proposition}
The 1-form $\omega_{H}:=\omega(H,\cdot)$ is closed on $T$.
\end{proposition}
\begin{proof}
Let $\th_\ell=-\sqrt{-1}\log\z_\ell$, which is a multi-valued function. However, $\na^L\th_\ell$ is well-defined a.e. w.r.t. $\mu_T$.
Let $K_t=\{x\in L:\,|\na^L\th|(x)\le t\}$ for each $t>0$, then 
\begin{equation}\aligned\label{dzltoeth}
&\liminf_{\ell\to\infty}\int\left|e^{\sqrt{-1}\th}\na^L\th-\z_\ell\na^L\th\right|^qd\mu_T
\le\liminf_{\ell\to\infty}\int\left|e^{\sqrt{-1}\th}-\z_\ell\right|^q|\na^L\th|^qd\mu_T\\
\le&\limsup_{\ell\to\infty}\int_{K_t}\left|e^{\sqrt{-1}\th}-\z_\ell\right|^q|\na^L\th|^qd\mu_T+\liminf_{\ell\to\infty}\int_{U\setminus K_t}\left|e^{\sqrt{-1}\th}-\z_\ell\right|^q|\na^L\th|^qd\mu_T\\
\le&\liminf_{\ell\to\infty}\int_{U\setminus K_t}\left|e^{\sqrt{-1}\th}-\z_\ell\right|^q|\na^L\th|^qd\mu_T\le2^q\int_{U\setminus K_t}|\na^L\th|^qd\mu_T\to0\ (\text{as}\ t\to\infty).
\endaligned
\end{equation}
Combining Triangle inequality and H\"older inequality, we have
\begin{equation}\aligned\label{naLthlnaLth}
&\liminf_{\ell\to\infty}\int\left|\na^L\th_\ell-\na^L\th\right|^qd\mu_T=\liminf_{\ell\to\infty}\int\left|\sqrt{-1}\na^L\z_\ell+\z_\ell\na^L\th\right|^qd\mu_T\\
\le&2^q\liminf_{\ell\to\infty}\int\left(\left|\sqrt{-1}\na^L\z_\ell+e^{\sqrt{-1}\th} \na^L\th\right|^q+\left|e^{\sqrt{-1}\th}\na^L\th-\z_\ell\na^L\th\right|^q\right)d\mu_T\\
=&2^q\liminf_{\ell\to\infty}\int\left|\sqrt{-1}\na^L\z_\ell+e^{\sqrt{-1}\th} \na^L\th\right|^qd\mu_T=0.
\endaligned
\end{equation}
From Stokes' formula, it follows that
\begin{equation}\aligned
0=-T(d(\th_{\ell}\wedge d\e))=T(d\th_{\ell}\wedge d\e)
=\int_U\lan d_L\th_{\ell}\wedge d\e,\vec{T}\ran d\mu_T
\endaligned
\end{equation}
for any $\e\in\cal{D}^{n-2}(U)$. Combining \eqref{naLthlnaLth}, letting $\ell\to\infty$ in the above equality implies
\begin{equation}\aligned
0=\lim_{\ell\to\infty}\int_U\lan d_L\th_\ell\wedge d\e,\vec{T}\ran d\mu_T=\int_U\lan d_L\th\wedge d\e,\vec{T}\ran d\mu_T=\int_U\lan\omega_{H}\wedge d\e,\vec{T}\ran d\mu_T.
\endaligned
\end{equation}
This gives the closeness of $\omega_{H}$ on $T$.
\end{proof}

\begin{lemma}\label{duiduikvar}
Let $T=(L,\vth,\vec{T})$ be an integer multiplicity Lagrangian current in $U$ with single-valued phase $\th$ in $U$ and $\p T\llcorner U=0$. If $\th\in W^{1,2}_{\mu_T}(L)$ is bounded harmonic on $L$ w.r.t. $\mu_T$, then for any sequence of Lipschitz functions $\th_i$ on $U$ with $\th_i\to\th$ in $L^2_{\mu_T}$ sense and $||\th_i||_{W^{1,2}}\to ||\th||_{W^{1,2}}$,
we have
\begin{equation}\aligned\label{iinftyLapthi}
\lim_{i\to\infty}\int\lan \na^L\th_i,\na^L(|\th_i|^{k}\th_i\varphi)\ran d\mu_T=0\qquad \mathrm{for\ each}\ \varphi\in \mathbf{Lip}_c(L)\ \mathrm{and\ each}\ k\ge0.
\endaligned
\end{equation}
\end{lemma}
\begin{proof}
Without loss of generality, we assume that there is a constant $\La>1$ so that $|\th_i|\le\La$ on $L$ for each $i$. For all integers $i,j\ge1$,
\begin{equation}\aligned
\na^L(|\th_i|^{k}\th_i-|\th_j|^{k}\th_j)=&(k+1)|\th_i|^{k}\na^L\th_i-(k+1)|\th_j|^{k}\na^L\th_j\\
=&(k+1)|\th_i|^{k}\na^L(\th_i-\th_j)+(k+1)(|\th_i|^{k}-|\th_j|^{k})\na^L\th_j.
\endaligned
\end{equation}
For each $\varphi\in \mathbf{Lip}_c(L)$, we have
\begin{equation}\aligned\label{thikthi}
&\bigg|\int\lan \na^L\th_i,\na^L((|\th_i|^{k}\th_i-|\th_j|^{k}\th_j)\varphi)\ran d\mu_T\bigg|\\
\le&\left|\int(|\th_i|^{k}\th_i-|\th_j|^{k}\th_j)\lan \na^L\th_i,\na^L\varphi\ran d\mu_T\right|+\left|\int\varphi\lan \na^L\th_i,\na^L(|\th_i|^{k}\th_i-|\th_j|^{k}\th_j)\ran d\mu_T\right|\\
\le&c_k\La^{k}\mathbf{Lip}\,\varphi\int|\th_i-\th_j|\cdot|\na^L\th_i|d\mu_T+c_k\La^{k}\max|\varphi|\int|\na^L\th_i|\cdot|\na^L(\th_i-\th_j)|d\mu_T\\
&+c_k\max|\varphi|\int\left||\th_i|^{k}-|\th_j|^{k}\right||\na^L\th_i|\cdot|\na^L\th_j|d\mu_T,
\endaligned
\end{equation}
where $c_k$ is a constant depending only on $k$.
For any $\ep>0$, there is a measurable set $L(\ep)\subset L$ so that $\th_i\to\th$ uniformly on $L(\ep)$, and $\mu_T(L\setminus L(\ep))\to0$ as $\ep\to0$. Hence,
\begin{equation}\aligned
\lim_{i,j\to\infty}\int\left||\th_i|^{k}-|\th_j|^{k}\right||\na^L\th_i|\cdot|\na^L\th_j|d\mu_T\le2\La^{k}\int_{L\setminus L(\ep)}|\na^L\th|^2d\mu_T.
\endaligned
\end{equation}
Since $|\na^L\th|\in L^2_{\mu_T}(L)$, it follows that
\begin{equation}\aligned
\lim_{i,j\to\infty}\int\left(\left||\th_i|^{k}-|\th_j|^{k}\right||\na^L\th_i|\cdot|\na^L\th_j|\right)d\mu_T=0.
\endaligned
\end{equation}
Letting $i,j\to\infty$ in \eqref{thikthi} gives
\begin{equation}\aligned\label{thijtoinfty}
\lim_{i,j\to\infty}\bigg|\int\lan \na^L\th_i,\na^L((|\th_i|^{k}\th_i-|\th_j|^{k}\th_j)\varphi)\ran d\mu_T\bigg|=0.
\endaligned
\end{equation}
Clearly,
\begin{equation}\aligned\label{thitoinfty}
\lim_{i\to\infty}\bigg|\int\lan \na^L\th_i,\na^L(|\th_j|^{k}\th_j\varphi)\ran d\mu_T\bigg|=0.
\endaligned
\end{equation}
Combining \eqref{thijtoinfty} and \eqref{thitoinfty}, we complete the proof.
\end{proof}
\begin{remark}\label{Remth+}
If we further assume that $\th$ has a positive lower bound, then \eqref{iinftyLapthi} holds for each $\varphi\in \mathbf{Lip}_c(L)$ and each $k\in\R$.
\end{remark}

\section{Mean curvature of varifold limits of Lagrangian currents}

Let $(M^n,g,J,\omega)$ be a Ricci-flat K\"ahler manifold with a holomorphic $n$-form $\Om$ on an open $U\subset M$ satisfying $(-1)^{n^2/2}\Om\wedge\overline{\Om}=2^{n}\omega^n/n!$. 
Let $T_i=(L_i,\vth,\vec{T}_i)$ be a sequence of integral Lagrangian currents in $U$ with phase $\th_i$ and $\p T_i\llcorner U=0$ so that $e^{2\sqrt{-1}\th_i}$ has bounded variation in $U$ w.r.t. $\mu_{T_i}$.
Let $V_i=|T_i|$ be the varifold associated with $T_i$. We assume that $V_i$ has generalized mean curvature $H_i:=J\na^{L_i}\th_i$, and 
\begin{equation}\aligned\label{Ass-nathiLa}
\sup_i\left(\mathbf{M}(T_i)+\mathbf{M}(\p T_i)+||H_i||_{L^q}\right)\le \La \qquad \text{for some }q>1,
\endaligned
\end{equation}
where $\La$ is a positive constant and $||H_i||_{L^q}$ is the $L^q$ integral of mean curvature $H_i$ on $L_i$ w.r.t. $\mu_{T_i}$.

By compactness theorem for $n$-varifolds (see \cite{S}), up to a subsequence $V_i$ converges (in the sense of Radon measure) to an integral varifold $V_\infty=\vth_\infty|L_\infty|$ of bounded first variation on $U$ with $\vth_\infty(x)\ge1$ for $\mu_{V_\infty}$-a.e. $x\in U$. In particular, 
$$\mathbf{M}(V_\infty)\le\liminf_{i\to\infty}\mathbf{M}(V_i)\le\La.$$ 
For each vector field $X\in \mathfrak{X}_c(U)$, 
\begin{equation}\aligned
-\int_{U} \mathrm{div}_{L_i}(JX) d\mu_{V_i}=&\int_{U} \lan H_i,JX\ran d\mu_{V_i}\\
=&\int_{U} \lan J\na^{L_i}\th_i,JX\ran d\mu_{V_i}=\int_{U} \lan \na^{L_i}\th_i,X\ran d\mu_{V_i}.
\endaligned
\end{equation}
Hence, we can define a linear functional $\mathcal{F}$ on $\mathfrak{X}_c(U)$ by
\begin{equation}\aligned
\mathcal{F}(X)=\lim_{i\to\infty}\int_{U} \lan\na^{L_i}\th_{i},X^{L_i}\ran d\mu_{V_i}=-\int_{U} \mathrm{div}_{L_\infty}(JX) d\mu_{V_\infty}.
\endaligned
\end{equation}
By H\"older inequality, for $q'=q/(q-1)$ we have
\begin{equation}\aligned
|\mathcal{F}(X)|\le
\liminf_{i\to\infty}\left(\int_U|\na^{L_i}\th_i|^qd\mu_{V_i}\right)^{1/q}\left(\int_U|X^{L_\infty}|^{q'}d\mu_{V_\infty}\right)^{1/q'}.
\endaligned
\end{equation}
From Riesz Representation Theorem, there is a measurable (normal) vector field $\xi$ with support in $L_\infty$ such that
\begin{equation}\aligned\label{semi-xi2}
\int_{U}|\xi|^qd\mu_{V_\infty}\le \liminf_{i\to\infty}\int_U|\na^{L_i}\th_i|^qd\mu_{V_i},
\endaligned
\end{equation}
and
\begin{equation}\aligned\label{limitxi}
\mathcal{F}(X)=\int_U\lan \xi,X\ran d\mu_{V_\infty}=\int_{U} \lan \xi,X^{L_\infty}\ran d\mu_{V_\infty}=-\int_U\mathrm{div}_{L_\infty}(JX) d\mu_{V_\infty}.
\endaligned
\end{equation}
Hence, $\xi\in L^q_{\mu_{V_\infty}}(TL_\infty)$, and $J\xi$ is the generalized mean curvature of $V_\infty$.

Up to choosing a subsequence, there is a complex measurable function $\z_\infty$ in $U$ so that 
\begin{equation}\aligned\label{App-zinfty}
\lim_{i\to\infty}\int_U fe^{2\sqrt{-1}\th_i} d\mu_{V_i}=\int_U f\z_\infty d\mu_{V_\infty}\qquad\mathrm{for\ each}\ f\in C_c(U).
\endaligned
\end{equation}

\begin{lemma}\label{zinftyViVinftymass}
$|\z_\infty|\equiv1$ $\mu_{V_\infty}$-a.e..
\end{lemma}
\begin{proof}
For any $\ep>0$, there is a countably collection of Borel sets $\{\mathcal{C}_k\}_k$ with $\mathcal{C}_k\subset L_\infty$ satisfying $\mu_{V_\infty}(L_\infty\setminus\cup_k\mathcal{C}_k)=0$, where there is a bi-Lipschitz map $\phi_{k}:\, \mathcal{C}_{k}\rightarrow\phi_{k}(\mathcal{C}_{k})\subset\R^n$ with Lipschitz constants $<1+\ep/n$ for each $k$. 
Up to a refinement of $\mathcal{C}_k$, there is a point $x_k\in\mathcal{C}_k$ so that
\begin{equation}\aligned\label{Ckzinftyxkep}
\sup_{\mathcal{C}_{k}}|\z_\infty-\z_\infty(x_k)|<\ep.
\endaligned
\end{equation}
For each $k$, there exists a sequence of Borel set $\{\mathcal{C}^i_k\}_k$ with $\mathcal{C}^i_k\subset L_i$ so that $\mu_{V_i}(L_i\setminus\cup_{k}\mathcal{C}^i_k)=0$ and $\mu_{V_i}\llcorner \mathcal{C}^i_k\to\mu_{V_\infty}\llcorner \mathcal{C}_k$ as $i\to\infty$ in the sense of measure, and a bi-Lipschitz map $\phi_{i,k}:\, \mathcal{C}_{k}^i\rightarrow\phi_{i,k}(\mathcal{C}_{k}^i)\subset\R^n$ with Lipschitz constants $<1+\ep/n$. 
Combining \eqref{Ckzinftyxkep}, it follows that
\begin{equation}\aligned
\limsup_{i\to\infty}\f1{\mu_{V_i}(\mathcal{C}^i_k)}\int_{\mathcal{C}^i_k}\left|e^{2\sqrt{-1}\th_i}-\z_\infty(x_k)\right|d\mu_{V_i}<\ep.
\endaligned
\end{equation}
This implies that $|\z_\infty|\equiv1$ $\mu_{V_\infty}$-a.e. since $\ep$ is arbitrary.
\end{proof}
\begin{remark}\label{xinotinW1q}
Even if $e^{\sqrt{-1}\th_i}\in W^{1,q}_{\mu_{V_i}}(L_i)$ for all $i$, $\z_\infty$ may still fail to live in $W^{1,q}_{\mu_{V_\infty}}(L_\infty)$. For example,  
let 
\begin{equation*}\aligned
L_i^+:=\{(x_1,\cdots,x_n,y_1,\cdots,y_n):\, y_1=x_1,\, y_2=\cdots=y_n=1/i\}\\
L^-:=\{(x_1,\cdots,x_n,y_1,\cdots,y_n):\, y_1=-x_1,\, y_2=\cdots=y_n=0\}
\endaligned\,,
\end{equation*}
and $T_i:=[|L_i^+\cup L^-|]$ with phase $\th_i$, then up to an orientation $\th_i=\pi/4\, (\mathrm{mod}\ 2\pi)$ on $L_i^+$, and $\th_i=3\pi/4\, (\mathrm{mod}\ 2\pi)$ on $L^+$. Since $L_i^+\cap L^-=\emptyset$, it follows that $\th_i\in W^{1,q}_{\mu_{T_i}}(L_i^+\cup L^-)$ for each $q>1$. Clearly, $T_i\rightharpoonup T:=[|L^+\cup L^-|]$ with
 \begin{equation*}\aligned
L^+:=\{(x_1,\cdots,x_n,y_1,\cdots,y_n):\, y_1=x_1,\, y_2=\cdots=y_n=0\},
\endaligned
\end{equation*}
and the phase of $T$ does not belong to $W^{1,q}_{\mu_{T}}(L^+\cup L^-)$ for any $q>1$.
\end{remark}

\begin{theorem}\label{zinftyJxiLVinfty}
$\z_\infty\in BV_{\mu_{V_\infty}}(U)$ and
\begin{equation}\aligned\label{nazinftyxi}
\na^{L_\infty}\z_\infty=2\sqrt{-1}\z_\infty\xi\qquad \mu_{V_\infty}-a.e..
\endaligned
\end{equation}
\end{theorem}
\begin{proof}
From the definition \eqref{phinuphina}, we only need to show
\begin{equation}\aligned
\int_U\z_{\infty}\left(\lan J\xi+2\sqrt{-1}\xi,X\ran+\mathrm{div}_{L_\infty}X\right)d\mu_{V_\infty}=0
\endaligned
\end{equation}
for each vector field $X\in \mathfrak{X}_c(U)$.
Since $e^{2\sqrt{-1}\th_i}$ has bounded variation in $U$ w.r.t. $\mu_{V_i}$, it follows that
\begin{equation}\aligned\label{thinaLithiX**}
2\sqrt{-1}\int_{U} e^{2\sqrt{-1}\th_i}\lan\na^{L_i}\th_{i},X\ran d\mu_{V_i}=-\int_Ue^{2\sqrt{-1}\th_i}\left(\lan H_i,X\ran+\mathrm{div}_{L_i}X\right)d\mu_{V_i}.
\endaligned
\end{equation}

Given a small $\ep>0$, let $\{\mathcal{C}^{i}_k\}_k$ and $\{\mathcal{C}_k\}_k$ be as in the proof of Lemma \ref{zinftyViVinftymass}.
Up to a refinement of $\mathcal{C}_k$, there is a complex constant $\Th_k$ with $|\Th_k|=1$ so that
\begin{equation}\aligned\label{thinftyThk}
\sup_{\mathcal{C}_{k}}|\z_\infty-\Th_k|<\ep.
\endaligned
\end{equation}
For every $\ell>>1$, there are integers $k^*_\ell>k_\ell>2$ so that $\mu_{V_\infty}(\cup_{k\ge k_\ell}\mathcal{C}_{k})<\ell^{-1}/2$ and 
$$\mu_{V_i}(U)\le \mu_{V_\infty}(U)+\f1{4\ell},\qquad \mu_{V_i}(\cup_{k< k_\ell}\mathcal{C}^{i}_{k})\ge\mu_{V_\infty}(\cup_{k< k_\ell}\mathcal{C}_{k})-\f1{4\ell}\qquad \text{ for each } i\ge k^*_\ell.$$
This implies
\begin{equation}\aligned\label{Ckinell}
\mu_{V_i}(\cup_{k\ge k_\ell}\mathcal{C}^{i}_{k})=\mu_{V_i}(U)-\mu_{V_i}(\cup_{k< k_\ell}\mathcal{C}^{i}_{k})<\f1{2\ell}+\mu_{V_\infty}(U)-\mu_{V_\infty}(\cup_{k< k_\ell}\mathcal{C}_{k})<\f1{\ell}.
\endaligned
\end{equation}
Hence, there are Borel subsets $\tilde{\mathcal{C}}^{i}_k\subset \mathcal{C}^{i}_k$ with $\lim_{i\to\infty}\mu_{V_i}(\mathcal{C}^{i}_k\setminus\tilde{\mathcal{C}}^{i}_k)=0$ so that 
\begin{equation}\aligned\label{supthiThk}
\sup_{\tilde{\mathcal{C}}^{i}_{k}}\left|e^{2\sqrt{-1}\th_i}-\Th_{k}\right|<\ep\qquad\text{for all }k\ge1, \text{ and large}\ i.
\endaligned
\end{equation}
From \eqref{limitxi}, we have
\begin{equation}\aligned\label{CkiCkThksum}
&\lim_{i\to\infty}\int_{\mathcal{C}^{i}_k}\Th_{k}\left(\lan J\na^{L_i}\th_{i}+2\sqrt{-1}\na^{L_i}\th_{i},X\ran+\mathrm{div}_{L_i}X\right)d\mu_{V_i}\\
=&\int_{\mathcal{C}_k}\Th_k\left(\lan J\xi+\sqrt{-1}\xi,X\ran+\mathrm{div}_{L_\infty}X\right)d\mu_{V_\infty}.
\endaligned
\end{equation}

On the other hand, combining \eqref{thinaLithiX**}\eqref{supthiThk} we get
\begin{equation}\aligned\label{ThLithiX}
&\left|\sum_{k}\int_{\mathcal{C}^{i}_k}\Th_{k}\left(\lan J\na^{L_i}\th_{i}+2\sqrt{-1}\na^{L_i}\th_{i},X\ran+\mathrm{div}_{L_i}X\right)d\mu_{V_i}\right|\\
\le&\sum_{k}\int_{\mathcal{C}^{i}_k}\left|e^{2\sqrt{-1}\th_i}-\Th_{k}\right|\left|\lan J\na^{L_i}\th_{i}+2\sqrt{-1}\na^{L_i}\th_{i},X\ran+\mathrm{div}_{L_i}X\right|d\mu_{V_i}\\
\le&\ep\sum_{k}\int_{\tilde{\mathcal{C}}^{i}_k}\left|\lan J\na^{L_i}\th_{i}+2\sqrt{-1}\na^{L_i}\th_{i},X\ran+\mathrm{div}_{L_i}X\right|d\mu_{V_i}\\
&+2\sum_{k}\int_{{\mathcal{C}}^{i}_k\setminus\tilde{\mathcal{C}}^{i}_k}\left|\lan J\na^{L_i}\th_{i}+2\sqrt{-1}\na^{L_i}\th_{i},X\ran+\mathrm{div}_{L_i}X\right|d\mu_{V_i}.
\endaligned
\end{equation}
From H\"older inequality, for any subset $W\subset U$ and $\de\in(0,1)$ it follows that
\begin{equation}\aligned\label{deSththi}
&\int_{W}\left|\lan J\na^{L_i}\th_{i}+2\sqrt{-1}\na^{L_i}\th_{i},X\ran+\mathrm{div}_{L_i}X\right|d\mu_{V_i}\\
\le&\de^q\int_W|\na^{L_i}\th_i|^qd\mu_{V_i}+\f{c_q}{\de^{q'}}\int_W|X|^{q'}d\mu_{V_i}+\int_W|\na X|d\mu_{V_i},
\endaligned
\end{equation}
where $q'=q/(q-1)$ and $c_q$ is a constant depending only on $q$.
Let $\la_X:=\sup_U|X|^{q'}+\sup_U|\na X|$. From \eqref{Ckinell}, there is an integer $k_*>>1$ so that $\sum_{k\ge k_*}\mu_{V_i}({\mathcal{C}}^{i}_k)\le\de^{q+q'}$ for all large $i$. 
Combining $\lim_{i\to\infty}\mu_{V_i}(\mathcal{C}^{i}_k\setminus\tilde{\mathcal{C}}^{i}_k)=0$ and \eqref{Ass-nathiLa}\eqref{deSththi}, we have
\begin{equation}\aligned\label{limisumkCkide}
&\limsup_{i\to\infty}\sum_{k}\int_{{\mathcal{C}}^{i}_k\setminus\tilde{\mathcal{C}}^{i}_k}\left|\lan J\na^{L_i}\th_{i}+2\sqrt{-1}\na^{L_i}\th_{i},X\ran+\mathrm{div}_{L_i}X\right|d\mu_{V_i}\\
=&\limsup_{i\to\infty}\sum_{k}\left(\de^q\int_{{\mathcal{C}}^{i}_k\setminus\tilde{\mathcal{C}}^{i}_k}|\na^{L_i}\th_i|^qd\mu_{V_i}+\f{c_q\la_X}{\de^{q'}}\mu_{V_i}({\mathcal{C}}^{i}_k\setminus\tilde{\mathcal{C}}^{i}_k)\right)\\
=&\limsup_{i\to\infty}\sum_{k\ge k_*}\left(\de^q\int_{{\mathcal{C}}^{i}_k\setminus\tilde{\mathcal{C}}^{i}_k}|\na^{L_i}\th_i|^qd\mu_{V_i}+\f{c_q\la_X}{\de^{q'}}\mu_{V_i}({\mathcal{C}}^{i}_k\setminus\tilde{\mathcal{C}}^{i}_k)\right)\\
\le&\de^q\La+\f{c_q\la_X}{\de^{q'}}\limsup_{i\to\infty}\sum_{k\ge k_*}\mu_{V_i}({\mathcal{C}}^{i}_k\setminus\tilde{\mathcal{C}}^{i}_k)\le \de^q\La+c_q\de^q\la_X.
\endaligned
\end{equation}
Letting $\de\to0$ in \eqref{limisumkCkide} gives
\begin{equation}\aligned\label{limisumkCki0}
\limsup_{i\to\infty}\sum_{k}\int_{{\mathcal{C}}^{i}_k\setminus\tilde{\mathcal{C}}^{i}_k}\left|\lan J\na^{L_i}\th_{i}+2\sqrt{-1}\na^{L_i}\th_{i},X\ran+\mathrm{div}_{L_i}X\right|d\mu_{V_i}=0.
\endaligned
\end{equation}
Combining \eqref{Ass-nathiLa}, \eqref{CkiCkThksum} and \eqref{ThLithiX}, we obtain
\begin{equation}\aligned\label{sumkCkThkep}
&\left|\sum_k\int_{\mathcal{C}_k}\Th_k\left(\lan J\xi+2\sqrt{-1}\xi,X\ran+\mathrm{div}_{L_\infty}X\right)d\mu_{V_\infty}\right|\\
\le&\ep\limsup_{i\to\infty}\sum_{k}\left(\int_{\mathcal{C}^{i}_k}|\na^{L_i}\th_i|^qd\mu_{V_i}+\int_{\mathcal{C}^{i}_k}|X|^{q'}d\mu_{V_i}+\int_{\mathcal{C}^{i}_k}|\na X|d\mu_{V_i}\right)\\
\le&\ep\limsup_{i\to\infty}\left(\int_{U}|\na^{L_i}\th_i|^qd\mu_{V_i}+\la_X\mu_{V_i}(U)\right)\le\ep\left(\La+\la_X\mu_{V_\infty}(U)\right).
\endaligned
\end{equation}
Combining \eqref{thinftyThk} and \eqref{sumkCkThkep}, it follows that
\begin{equation}\aligned\label{CkthinftyThk}
&\left|\sum_k\int_{\mathcal{C}_k}\z_\infty\left(\lan J\xi+2\sqrt{-1}\xi,X\ran+\mathrm{div}_{L_\infty}X\right)d\mu_{V_\infty}\right|\\
\le&\left|\sum_k\int_{\mathcal{C}_k}\Th_k\left(\lan J\xi+2\sqrt{-1}\xi,X\ran+\mathrm{div}_{L_\infty}X\right)d\mu_{V_\infty}\right|\\
&+\ep\sum_k\int_{\mathcal{C}_k}\left|\lan J\xi+2\sqrt{-1}\xi,X\ran+\mathrm{div}_{L_\infty}X\right|d\mu_{V_\infty}\\
\le&\ep\left(\La+\la_X\mu_{V_\infty}(U)\right)+\ep\int_U\left|\lan J\xi+2\sqrt{-1}\xi,X\ran+\mathrm{div}_{L_\infty}X\right|d\mu_{V_\infty}.
\endaligned
\end{equation}
By the definition of $\mathcal{C}_k$, forcing $\ep\to0$ in \eqref{CkthinftyThk} completes the proof.
\end{proof}

From \eqref{semi-xi2}, Lemma \ref{zinftyViVinftymass} and Theorem \ref{zinftyJxiLVinfty}, we immediately have
\begin{equation}\aligned\label{seminaeith}
\int_U|\na^{L_\infty}\z_\infty|^qd\mu_{V_\infty}
\le 2^q\liminf_{i\to\infty}\int_U |\na^{L_i}\th_i|^qd\mu_{V_i}.
\endaligned
\end{equation}
From Federer-Fleming compactness theorem, there is an integral current $T=(L,\vth,\vec{T})$ in $U$ so that $T_i\rightharpoonup T$ up to choosing a subsequence. Hence,
\begin{equation}\aligned
T(\omega\wedge\e)=\lim_{i\to\infty}T_i(\omega\wedge\e)=\int\lan\omega\wedge\e,\vec{T}_i\ran d\mu_{T_i}=0\qquad \mathrm{for\ any}\ \e\in\mathcal{D}^{n-2}(U),
\endaligned
\end{equation}
which means that $T$ is a Lagrangian current. There is a multi-valued measurable function $\th$ on $L$ so that
$$\Omega\big|_L=e^{\sqrt{-1}\th}\mathrm{vol}_L\qquad \mu_T-a.e.$$ with an induced volume form $\mathrm{vol}_L$ of $L$ w.r.t. $\vec{T}$. For each $f\in C^0(U)$,
\begin{equation}\aligned\label{z_itozC0}
\lim_{i\to\infty}\int_U fe^{\sqrt{-1}\th_i} d\mu_{V_i}=&\lim_{i\to\infty}\int_U f\lan\Om,\vec{T}_i\ran d\mu_{T_i}=\lim_{i\to\infty}T_i(f\Om)\\
=&T(f\Om)=\int_U f\lan\Om,\vec{T}\ran d\mu_{T}=\int_U fe^{\sqrt{-1}\th} d\mu_T.
\endaligned
\end{equation}
In the case of $V_\infty=|T|$, we get $e^{\sqrt{-1}\th_i}\to e^{\sqrt{-1}\th}$ in the sense of measure from \eqref{z_itozC0}, then $\z_\infty=e^{2\sqrt{-1}\th}$ combining \eqref{App-zinfty}. In general, $\f12\arg\z_\infty$ may not be the phase of $T$. We will study it further under the condition of single-valued harmonic phases in \S6.

\section{Lagrangian currents with single-valued harmonic phases: volume estimates}

Let $(M^n,\omega)$ be a Ricci-flat K\"ahler manifold with a holomorphic $n$-form $\Om$ on $U$ satisfying $(-1)^{n^2/2}\Om\wedge\overline{\Om}=2^{n}\omega^n/n!$. 
Let $T=(L,\vth,\vec{T})$ be an integral Lagrangian current in $U$ with $\p T\llcorner U=0$ and locally bounded, single-valued harmonic phase $\th$. Denote $V=\vth|L|$.

We first prove the following Caccioppoli type inequality under a further condition $\th\in W^{1,2}_{\mu_V}(L)$. 
\begin{lemma}\label{caccith}
For each constant $\ell\ge2$ and each Lipschitz function $\phi$ on $U$ with compact support, there holds
\begin{equation}\aligned\label{caccith*}
\int |\th|^{\ell-2}\phi^2|\na^L \th|^2d\mu_V\le\left(\f2{\ell-1}\right)^2\int |\th|^{\ell}|\na^L\phi|^2d\mu_V.
\endaligned
\end{equation}
\end{lemma}
\begin{proof}
Let $\th_i$ be a sequence of Lipschitz functions on $(L,d^L,\mu_V)$ with $(\th_i,\na^L\th_i)\to(\th,\na^L\th)$ in $L^2_{\mu_V}(L)$ sense.
For each constant $\ell\ge2$, from Lemma \ref{duiduikvar} we have
\begin{equation}\aligned\label{thinanathi}
0=&\lim_{i\to\infty}\int\lan\na^L (|\th_i|^{\ell-2}\th_i\phi^2),\na^L \th_i\ran d\mu_V\\
=&(\ell-1)\int |\th|^{\ell-2}\phi^2|\na^L \th|^2d\mu_V+2\int |\th|^{\ell-2}\th\phi\lan\na^L \th,\na^L\phi\ran d\mu_V\\
\ge&\f{\ell-1}2\int |\th|^{\ell-2}\phi^2|\na^L \th|^2d\mu_V-\f2{\ell-1}\int_M |\th|^{\ell}|\na^L\phi|^2d\mu_V,
\endaligned
\end{equation}
where we have used Cauchy-Schwartz inequality in the last line of \eqref{thinanathi}.
\end{proof}

From Nash's isometric embedding theorem,
we assume that $U$ is a $2n$-submanifold in $\R^m$ with $m\ge 2n$. 
Let $\mathbf{A}_U$ denote the second fundamental form of $U$ in $\R^m$.
Let $\bn,\na$ denote the Levi-Civita connections of $\R^m$ and $M$, respectively.
For a smooth vector field $X$ on $\R^m$, 
let $X^U$ denote the restriction of $X$ into $TU$. 
At a considered point $p\in S$ with $T_pS$ being an $n$-plane,
let $\{e_i\}_{i=1}^n$ be an orthonormal basis of $T_pS$, and $\mathrm{tr}_S\mathbf{A}_{U}:=\sum_{1\le i\le n}\mathbf{A}_U(e_i,e_i)$. Clearly, $\mathrm{tr}_S\mathbf{A}_{U}$ is independent of the choice of $\{e_i\}$.
By a direct computation (see page 58 in \cite{S} for instance), at $p$
\begin{equation}\aligned
\div_S X=&\div_S X^U+\sum_i\lan\bn_{e_i}(X-X^U),e_i\ran\\
=&\div_S X^U-\sum_i\lan(X-X^U),\bn_{e_i}e_i\ran=\div_S X^U-\lan X,\mathrm{tr}_S\mathbf{A}_{U}\ran.
\endaligned
\end{equation}
We can also regard $V$ as a varifold in $\R^m$.
If $X^U$ has compact support in $U$, then
from \eqref{Def-GMC} it follows that
\begin{equation}\aligned\label{bfHHAM}
\int_{\R^m} \mathrm{div}_SX d\mu_V=-\int_{\R^m}\lan H+\mathrm{tr}_S\mathbf{A}_{U},X\ran d\mu_V.
\endaligned
\end{equation}
Thus, $V$ has (generalized) mean curvature $H+\mathrm{tr}_S\mathbf{A}_{U}$ in $U\subset\R^m$.
Let $\k$ be the upper bound of the maximum absolute value of the principal curvature of $U\subset\R^m$. If we regard $V$ as an $n$-varifold in $\R^m$, then the norm of mean curvature of $V$ in $\R^m$ is bounded by $|\na^L\th|+\k$ from \eqref{bfHHAM} and Theorem \ref{thW12-H}.

Let $\mathbf{B}_r(\mathbf{x})$ denote the ball in $\R^m$ with radius $r$ and centered at $\mathbf{x}\in\R^m$. Denote $\mathbf{B}_r=\mathbf{B}_r(\mathbf{0})$ for short.
We assume $\mathbf{B}_{R_0}\cap\p U=\emptyset$ for some $R_0>0$.
Recalling the isoperimetric inequality on varifolds \cite{MS} by Michael-Simon:
There is a constant $a_n>0$ depending only on $n$ so that 
\begin{equation}\aligned\label{Sob}
\left(\int|f|^{\f{n}{n-1}}d\mu_{V}\right)^{\f{n-1}{n}}\le a_n\int\left(|\na^L f|+(|\na^L\th|+\k)f\right)d\mu_V
\endaligned
\end{equation}
for each $f\in C^1_c(\mathbf{B}_{R_0})$. Noting \eqref{monofHkdf} in Appendix II holds with $f$ equal to $\th^\ell \varphi$ with $\ell\ge1$. From the proof of \eqref{Sob} (see \cite{MS} or Theorem 3.11 in \cite{CM1}), \eqref{Sob} holds for functions in the form of $\th^\ell \varphi$ with $\ell\ge1$ and $\varphi\in C^1_c(\mathbf{B}_{R_0})$.

For each measurable function or vector field $X$ on $\mathbf{B}_{R_0}\cap L$ and each constant $q>0,r\in(0,R_1]$, we write
$$||X||_{q,r}:=\left(\int_{\mathbf{B}_{r}}|X|^q d\mu_V\right)^{1/q}.$$

\begin{lemma}\label{supthB12p}
For each $q>n$, let $\La_1,\La_2$ denote the positive constants so that $\mu_V(\mathbf{B}_{R_0})\le\La_1 R_0^n$ and $||\th||_q\le\La_2 R_0^{n/q}$, then
$$\sup_{\mathbf{B}_{R_0/2}}|\th|\le c\left(n,q,\k R_0,\La_1,\La_2\right).$$
\end{lemma}
\begin{proof}
By scaling, we only need to prove the case of $R_0=1$.
For each $1-\tau\ge r\ge\tau>0$ and $r\in(0,1/2]$, let $\phi$ be a function defined by $\phi\equiv1$ on $\mathbf{B}_{r}$, $\phi=\f1\tau\left(r+\tau-|\cdot|\right)$ on $\mathbf{B}_{r+\tau}\setminus \mathbf{B}_{r}$,
$\phi\equiv0$ outside $\mathbf{B}_{r+\tau}$.
Then $|\na^L\phi|\le1/\tau$. 
From Lemma \ref{caccith}, we get
\begin{equation}\aligned\label{naellle2}
&\int\left|\na^L (\th^\ell\phi^2)\right|d\mu_V\le\ell\int|\th|^{\ell-1}\phi^2\left|\na^L \th\right| d\mu_V+2\int|\th|^{\ell}\phi\left|\na^L \phi\right| d\mu_V\\
\le&\f1{\tau}\int|\th|^{\ell}\phi^{2}d\mu_V+\f{\tau\ell^2}4\int|\th|^{\ell-2}\phi^2\left|\na^L \th\right|^2 d\mu_V+2\int|\th|^{\ell}\phi\left|\na^L \phi\right| d\mu_V\\
\le&\f1{\tau}\int|\th|^{\ell}\phi^{2}d\mu_V+\f{\tau\ell^2}{(\ell-1)^2}\int|\th|^{\ell}|\na^L\phi|^2 d\mu_V+2\int|\th|^{\ell}\phi\left|\na^L \phi\right| d\mu_V,
\endaligned
\end{equation}
and similarly
\begin{equation}\aligned\label{naellle2**}
\int|\th|^{\ell}\phi^2\left|\na^L \th\right| d\mu_V
\le\f1{(\ell+1)\tau}\int|\th|^{\ell+1}\phi^{2}d\mu_V+\f{\tau(\ell+1)}{\ell^2}\int|\th|^{\ell+1}|\na^L\phi|^2 d\mu_V.
\endaligned
\end{equation}
Combining Sobolev inequality \eqref{Sob} and \eqref{naellle2}\eqref{naellle2**}, we have
\begin{equation}\aligned\label{thelltau+tau}
||\th^{\ell}||_{\f {n}{n-1},r}\le& ||\th^{\ell}\phi^2||_{\f {n}{n-1},1}\le a_n||\na^L(\th^{\ell}\phi^2)||_{1,1}+a_n||\th^{\ell}\phi^2\na^L\th||_{1,1}+\k a_n||\th^{\ell}\phi^2||_{1,1}\\
\le& c_{n,\k}\left(\f1\tau||\th^{\ell}||_{1,r+\tau}+\f{1}{\ell\tau}||\th^{\ell+1}||_{1,r+\tau}\right).
\endaligned
\end{equation}
Namely,
\begin{equation}\aligned\label{thln1r}
||\th||_{\f {n\ell}{n-1},r}^\ell\le \f{c_{n,\k}}\tau||\th||_{\ell,r+\tau}^\ell+\f{c_{n,\k}}{\ell\tau}||\th||_{\ell+1,r+\tau}^{\ell+1}.
\endaligned
\end{equation}
Let $\ell_0=(n-1)q/n>n-1$ and $\ell_i=(\f {n}{n-1})^{i/2}\ell_0$
for all integers $i\ge1$. From H\"older inequality, 
\begin{equation}\aligned\label{thlili+1}
||\th||_{\ell_i,r+\tau}^{\ell_i}\le||\th||_{\ell_{i+1},r+\tau}^{\ell_i}\left(\mu_V(B_1(p))\right)^{1-\ell_i/\ell_{i+1}}\le c(n,q,\La_1)||\th||_{\ell_{i+1},r+\tau}^{\ell_i}.
\endaligned
\end{equation}
Let $[\cdots]$ be the floor function for real numbers, and we fix an integer
$$k_n:=2+2\left[\f{\log(1+\sqrt{2})}{\log n-\log(n-1)}\right]\ge\max\left\{2,\f{\log(1+\sqrt{2})}{\log n-\log(n-1)}\right\}.$$
It's easy to see $\ell_i+1\le\ell_{i+1}$ for all $i\ge k_n$. From H\"older inequality, 
\begin{equation}\aligned\label{thelli+1i+1**}
||\th||_{\ell_i+1,r+\tau}^{\ell_i+1}\le||\th||_{\ell_{i+1},r+\tau}^{\ell_i+1}\left(\mu_V(B_1(p))\right)^{1-(\ell_i+1)/\ell_{i+1}}\le c(n,q,\La_1)||\th||_{\ell_{i+1},r+\tau}^{\ell_i+1}.
\endaligned
\end{equation}
By iteration and H\"older inequality, we get
\begin{equation}\aligned
||\th||_{{\ell_{k_n+1}},8/9}\le c(n,q,\k,\La_1,\La_2).
\endaligned
\end{equation}
From \eqref{thln1r}, \eqref{thlili+1} and \eqref{thelli+1i+1**}, for all $i\ge k_n$
\begin{equation}\aligned
||\th||_{\ell_{i+2},r}^{\ell_i}\le c_*||\th||_{\ell_i,r+\tau}^{\ell_i}+c_*||\th||_{\ell_i+1,r+\tau}^{\ell_i+1}
\le c_*||\th||_{\ell_{i+1},r+\tau}^{\ell_i}+c_*|||\th||_{\ell_{i+1},r+\tau}^{\ell_i+1},
\endaligned
\end{equation}
where $c_*\ge1$ is a constant depending only on $n,q,\tau,\k,\La_1,\La_2$. 
For all $j\ge0$, we set $\tau_j=2^{-(2+j)}r$ and
$r_{j+1}=r_j-\tau_j$ with $r_0=r$.
Then
$$r_{j+1}=r_0-\sum_{k=0}^{j}\tau_k=\f r2+\tau_j\le r,$$
and $\lim_{j\to\infty}r_j=r/2$.
Set $a_i=\log\max\{1,||\th||_{\ell_{i},r_i}\}$ for all $i\ge k_n$. Then
\begin{equation}\aligned
e^{\ell_ia_{i+2}}\le \f{c_*}2e^{\ell_ia_{i+1}}+\f{c_*}2e^{(\ell_i+1)a_{i+1}}\le c_*e^{(\ell_i+1)a_{i+1}},
\endaligned
\end{equation}
which implies
\begin{equation}\aligned
a_{i+2}\le \f{c_*}{\ell_i}+\left(1+\f{1}{\ell_i}\right)a_{i+1}.
\endaligned
\end{equation}
Set $b_i=\max\{c_*,a_{i+1}\}$. By the definition of $\ell_i$, it follows that
\begin{equation}\aligned
b_{i+1}\le \left(1+\f{2}{\ell_i}\right)b_i\le b_{k_n}\prod_{j=k_n}^i\left(1+\f{2}{\ell_j}\right)\le c_nb_{k_n}\qquad\mathrm{for\ all}\ i\ge k_n.
\endaligned
\end{equation}
Letting $i\to\infty$ completes the proof.
\end{proof}

Now let us give a more refined estimate compared with Lemma \ref{supthB12p}.
Let $c_\th$ denote a general constant depending only on $n$ and the upper bound of $\th$ on $L\cap\mathbf{B}_1$. 
From \eqref{thelltau+tau}, it follows that
\begin{equation}\aligned\label{v2lnr}
||\th^{\ell}||_{\f {n}{n-1},r}\le& c_n\left(\f1\tau||\th^{\ell}||_{1,r+\tau}+\f{1}{\ell\tau}||\th^{\ell+1}||_{1,r+\tau}\right)\le\f{c_\th}{\tau}||\th^{\ell}||_{1,r+\tau}.
\endaligned
\end{equation}
Let $\ell_i=(\f {n}{n-1})^i$ for each integer $i\ge0$.
Then \eqref{v2lnr} gives
\begin{equation}\aligned\label{ite}
||\th||_{\ell_{i+1},r}=||\th^{\ell_i}||_{\f {n}{n-1},r}^{1/\ell_i}\le c_\th^{\f1{\ell_i}}\tau^{-\f1{\ell_i}}||\th||_{\ell_i,r+\tau}.
\endaligned
\end{equation}
For any $\si\in(0,1)$, we set $\tau_i=2^{-(1+i)}(1-\si)r$ and
$r_{i+1}=r_{i}-\tau_{i}$ with $r_0=r\le1/2$ as before.
By iterating \eqref{ite}, for each $i\ge0$ we have
\begin{equation}\aligned
||\th||_{\ell_{i+1},r_{i+1}}\le c_\th^{\f1{\ell_i}}\tau_i^{-\f1{\ell_i}}||\th||_{\ell_{i},r_{i}}=\left(\f{c_\th}{(1-\si)r}\right)^{(\f{n-1}{n})^i}2^{(1+i)(\f{n-1}{n})^i}||\th||_{\ell_{i},r_{i}}.
\endaligned
\end{equation}
Since 
$
\sum_{j=0}^\infty(j+1)\a^{-j}=\f{\a^2}{(\a-1)^2}$ for $\a>1$,
we have
\begin{equation}\aligned
||\th||_{\ell_{i+1},r_{i+1}}\le&\left(\f{c_\th}{(1-\si)r}\right)^{\sum_{j=0}^i(\f{n-1}{n})^j}2^{\sum_{j=0}^i(1+j)(\f{n-1}{n})^j}||\th||_{\ell_0,r_{0}}\\
\le&\left(\f{c_\th}{(1-\si)r}\right)^{n}2^{n^2}||\th||_{1,r}.
\endaligned
\end{equation}
Letting $i\rightarrow\infty$, it follows that
\begin{equation}\aligned\label{Phiinftyder}
||\th||_{\infty,\si r}\le\left(\f{c_\th}{(1-\si)r}\right)^{n}||\th||_{1,r}.
\endaligned
\end{equation}
Up to adding a constant to $\th$, \eqref{Phiinftyder} immediately gives the following lower bound. 
\begin{theorem}\label{Low-V-Brp}
There exists a constant $c>0$ depending only on $n,q,\k R_0,\La_1,\La_2$ such that 
$$\mathbf{M}(V\llcorner \mathbf{B}_r)\ge cr^n\qquad \mathrm{for\ each}\ 0<r\le R_0.$$
\end{theorem}

Now, we consider the general case of single-valued harmonic phase $\th$ without the assumption $\th\in W^{1,2}_{\mu_V}(L)$.
From Definition \ref{Def-harm-curr}, there is a countable cover $\{U_k\}_{k\in \N}$ of $U$ so that for every $k$, $T\llcorner U_k$ has indecomposable components $T_{k,1},T_{k,2},\cdots$ satisfying that
the phase $\th_{k,j}$ of $T_{k,j}$ belongs to $W^{1,2}_{\mu_V}(L_{k,j})$, and is harmonic on each $L_{k,j}:=\text{spt}\, T_{k,j}$ w.r.t. $\mu_V$.
For any compact set $K\subset U$, we can require that there are only finite elements in $\{U_k\}_{k\in \N}$, whose  intersections with $K$ are nonempty. Let $U'_k\subset\subset U_k$ be open for each $k$ so that $\{U'_k\}_{k\in \N}$ is still a cover of $U$.
From Theorem \ref{Low-V-Brp}, for each fixed $k$ there are only finite elements in $\{T_{k,j}\}_j$, denoted by $\{T_{k,j}\}_{1\le j\le m_k}$ for some integer $m_k>1$, whose intersections with $U'_k$ are nonempty.

With loss of generality, we can assume $\th_{k,j}=\th|_{L_{k,j}}$ on each $L_{k,j}$.
There is a sequence of nonnegative Lipschitz functions $\{\phi_k\}_{k\ge1}$ on $U$ with spt$\phi_k\subset U'_k$ for each $k$ so that $\sum_{k\ge1}\phi_k\equiv1$ on $U$. Then
$$J\sum_{k\ge1}\sum_{1\le j\le m_k}\phi_{k}\na^{L_{k,j}}\th_{k,j}$$
is the generalized mean curvature $H_V$ of $V$. From the linearity of $\phi$ in \eqref{phinuphina} and Theorem \ref{thW12-H}, we conclude that $\th\in BV_{\mu_V}(U)$ with $J\na^L\th=H_V$.

From \cite{G} or \cite{GP}, 
\begin{equation}\aligned\label{naLthk}
\na^{L_{k,j}}\th_{k,j}=\na^{L_{k',j'}}\th_{k',j'}\qquad\mu_V-a.e.\ \text{on }L_{k,j}\cap L_{k',j'}.
\endaligned
\end{equation}
Noting $\sum_k\phi_k\equiv1$ on $U$. Then
\begin{equation}\aligned\label{naLthk2}
|\na^L\th|^2=\sum_{k,j}\phi_k\bigg\lan\na^{L_{k,j}}\th_{k,j},\sum_{k',j'}\phi_{k'}\na^{L_{k',j'}}\th_{k',j'}\bigg\ran=\sum_{k,j}\phi_k\big|\na^{L_{k,j}}\th_{k,j}\big|^2,
\endaligned
\end{equation}
and
\begin{equation}\aligned\label{naLthkphik}
0=&\sum_k|\th|^{\ell-2}\th\lan\na^L\th,\na^L\phi_k\ran=\sum_{k,k',j}|\th|^{\ell-2}\th\phi_{k'}\big\lan\na^{L_{k',j}}\th_{k',j},\na^L\phi_k\big\ran\\
=&\sum_{k,j}|\th_{k,j}|^{\ell-2}\th_{k,j}\big\lan\na^{L_{k,j}}\th_{k,j},\na^L\phi_k\big\ran,
\endaligned
\end{equation}
where $k,k'\ge1$, $1\le j\le m_k$ and $1\le j'\le m_{k'}$ in the above sums.

Put $T'_{k,j}=T_{k,j}\llcorner U'_k$.
There is a sequence of Lipschitz functions $\{\th_{k,j,i}\}_i$ on $(L_{k,j},d^{L_{k,j}},\mu_{T'_{k,j}})$ so that $(\th_{k,j,i},\na^{L}\th_{k,j,i})\to(\th_{k,j},\na^{L_{k,j}}\th_{k,j})$ in $L^2_{\mu_{T'_{k,j}}}(L_{k,j})$ sense. Let $\phi$ be a nonnegative Lipschitz function on $U$ with compact support.
For each constant $\ell\ge2$, from Lemma \ref{duiduikvar} we have
\begin{equation}\aligned
0=&\lim_{i\to\infty}\int\lan\na^L (|\th_{k,j,i}|^{\ell-2}\th_{k,j,i}\phi^2\phi_k),\na^L \th_{k,j,i}\ran d\mu_{T'_{k,j}}\\
=&(\ell-1)\int |\th_{k,j}|^{\ell-2}\phi^2|\na^{L_{k,j}} \th_{k,j}|^2\phi_k d\mu_{T'_{k,j}}\\
&+2\int |\th_{k,j}|^{\ell-2}\th_{k,j}\phi\lan\phi_k\na^{L_{k,j}} \th_{k,j},\na^L\phi\ran d\mu_{T'_{k,j}}\\
&+\int |\th_{k,j}|^{\ell-2}\th_{k,j}\phi^2\lan\na^{L_{k,j}} \th_{k,j},\na^L\phi_k\ran d\mu_{T'_{k,j}}.
\endaligned
\end{equation}
From \eqref{naLthk2} and \eqref{naLthkphik}, summing the above equality over $k\ge1,1\le j\le m_k$ yields
\begin{equation}\aligned\label{thinanathi*}
0=(\ell-1)\int |\th|^{\ell-2}\phi^2|\na^L \th|^2d\mu_V+2\int |\th|^{\ell-2}\th\phi\lan\na^L \th,\na^L\phi\ran d\mu_V.
\endaligned
\end{equation}
With Cauchy-Schwartz inequality, we can derive Lemma \ref{caccith}.
Correspondingly, Lemma \ref{supthB12p} and Theorem \ref{Low-V-Brp} also hold under the condition of single-valued harmonic phase.
\begin{remark}\label{Remth+*}
From Remark \ref{Remth+}, \eqref{caccith*} holds for each $\ell\in\R$ provided $\th$ is positive.
\end{remark}

\begin{proof}[Proof of Theorem \ref{Al-Mono-int}]
Given $\mathbf{p}\in \mathbf{B}_{R_0}$, we set $\r(\mathbf{x})=|\mathbf{x}-\mathbf{p}|$ and $R_1=R_0-|\mathbf{p}|>0$.
From \eqref{thinanathi*} with $\phi^2=s^2-\r^2$, it follows that
\begin{equation}\aligned\label{BrxnaVthell0}
0=(\ell-1)\int_{\mathbf{B}_s(\mathbf{p})} |\th|^{\ell-2}(s^2-\r^2)|\na^L \th|^2d\mu_V-\int_{\mathbf{B}_s(\mathbf{p})} |\th|^{\ell-2}\th\lan\na^L \th,\na^L\r^2\ran d\mu_V
\endaligned
\end{equation}
for each $\ell\ge2$ and $0<s<R_1$.
Then
\begin{equation}\aligned\label{BrxnaVthell}
\int_{\mathbf{B}_s(\mathbf{p})}\lan \r\na^L\r,\na^L |\th|^\ell\ran d\mu_V
=&\f{\ell}2\int_{\mathbf{B}_s(\mathbf{p})}|\th|^{\ell-2}\th\lan \na^L\r^2,\na^L \th\ran d\mu_V\\
=&\f{\ell(\ell-1)}2\int_{\mathbf{B}_s(\mathbf{p})}(s^2-\r^2)|\th|^{\ell-2}|\na^L\th|^2 d\mu_V.
\endaligned
\end{equation}
Combining Theorem \ref{thW12-H}, for $0<r<R\le R_1$
\begin{equation}\aligned\label{BrxnaVthell*}
&\int_r^R\f{e^{\k s}}{s^{n+1}}\int_{\mathbf{B}_s(\mathbf{p})}\r\lan \na\r,\na^L |\th|^\ell\ran d\mu_Vds\\
=&\f{\ell(\ell-1)}2\int_r^R\f{e^{\k s}}{s^{n+1}}\int_{\mathbf{B}_s(\mathbf{p})}(s^2-\r^2)|\th|^{\ell-2}|H_V|^2 d\mu_Vds\\
\ge&\f{\ell^2}8\int_r^R\int_{\mathbf{B}_{s/2}(\mathbf{p})}\f{e^{\k s}}{s^{n-1}}|\th|^{\ell-2}|H_V|^2 d\mu_Vds\\
=&\f{\ell^2}{2^{n+1}}\int_{r/2}^{R/2}\int_{\mathbf{B}_{s}(\mathbf{p})}\f{e^{2\k s}}{s^{n-1}}|\th|^{\ell-2}|H_V|^2 d\mu_Vds.
\endaligned
\end{equation}
For every $\ell\ge2$, we set $\varphi_\ell:=\r\left|\na^N\r\right|\cdot|H_V|\cdot|\th|^\ell$ and
\begin{equation}\aligned
\mathcal{W}_{\mathbf{p},\ell}(s):=s^{1-n}\int_{\mathbf{B}_{s}(\mathbf{p})}|\th|^{\ell-2}|H_V|^2 d\mu_V\qquad \mathrm{for\ each}\ s\in(0,R].
\endaligned
\end{equation}
From \eqref{Almost-mono-Vinfty} (with $f=|\th|^{\ell}$) in Appendix II and \eqref{BrxnaVthell*}, we get
\begin{equation}\aligned\label{Almost-mono-thell}
&\f{e^{\k R}}{R^n}\int_{\mathbf{B}_R(\mathbf{p})}|\th|^{\ell} d\mu_{V}-\f{e^{\k r}}{r^n}\int_{\mathbf{B}_r(\mathbf{p})}|\th|^{\ell}d\mu_{V}-\int_{\mathbf{B}_R(\mathbf{p})\setminus \mathbf{B}_r(\mathbf{p})}\f{e^{\k \r}}{\r^n}\left|\na^N\r\right|^2|\th|^{\ell} d\mu_V\\
\ge&\int_r^R\f{e^{\k s}}{s^{n+1}}\int_{\mathbf{B}_s(\mathbf{p})}\lan |\th|^{\ell}H_V+\na^L |\th|^{\ell},\r\na\r\ran d\mu_Vds\\
\ge&\f{\ell^2}{2^{n+1}}\int_{r/2}^{R/2}\mathcal{W}_{\mathbf{p},\ell}(s)ds-\int_r^R\f{e^{\k s}}{s^{n+1}}\int_{\mathbf{B}_s(\mathbf{p})}\varphi_{\ell} d\mu_Vds.
\endaligned
\end{equation}
By Fubini's theorem, for a function $\varphi\in L^1_{\mu_V}(\mathbf{B}_R(\mathbf{p}))$ with $0<r<R\le R_1$ we have
\begin{equation}\aligned\label{xHVthell}
&\int_r^R\f{e^{\k s}}{s^{n+1}}\int_{\mathbf{B}_s(\mathbf{p})}\varphi d\mu_Vds\\
=&\int_r^R\f{e^{\k s}}{s^{n+1}}\int_{ \mathbf{B}_r(\mathbf{p})}\varphi d\mu_Vds+\int_r^R\f{e^{\k s}}{s^{n+1}}\int_{\mathbf{B}_s(\mathbf{p})\setminus  \mathbf{B}_r(\mathbf{p})}\varphi d\mu_Vds\\
=&\int_r^R\f{e^{\k s}}{s^{n+1}}ds\int_{ \mathbf{B}_r(\mathbf{p})}\varphi d\mu_V+\int_{\mathbf{B}_R(\mathbf{p})\setminus  \mathbf{B}_r(\mathbf{p})}\varphi\int_\r^R\f{e^{\k s}}{s^{n+1}}ds d\mu_V.
\endaligned
\end{equation}

Now, we fix a constant 
\begin{equation}\aligned\label{ell*req}
\ell_*:=\max\left\{2,2^{n/2}e^{\k R_0}\sup_{L\cap\mathbf{B}_{R_0}}|\th|\right\}.
\endaligned
\end{equation}
From Cauchy-Schwartz inequality,
\begin{equation}\aligned\label{C-S-Phiell}
\varphi_{\ell_*}=\r\left|\na^N\r\right|\cdot|H_V|\cdot|\th|^{\ell_*}\le e^{-\k R}\left|\na^N\r\right|^2|\th|^{\ell_*}+\f{e^{\k R}}{4}\r^2|H_V|^2|\th|^{\ell_*}.
\endaligned
\end{equation}
From \eqref{xHVthell}\eqref{ell*req} it follows that
\begin{equation}\aligned\label{Phil*est}
&\int_r^R\f{e^{\k s}}{s^{n+1}}\int_{\mathbf{B}_s(\mathbf{p})}\varphi_{{\ell_*}} d\mu_Vds\le\f{e^{\k R}}{4}\int_r^R\f{e^{\k s}}{s^{n+1}}\int_{\mathbf{B}_s(\mathbf{p})}\r^2|H_V|^2|\th|^{\ell_*} d\mu_Vds\\
&+\int_r^R\f{1}{s^{n+1}}ds\int_{ \mathbf{B}_r(\mathbf{p})}\left|\na^N\r\right|^2|\th|^{\ell_*} d\mu_V+\int_{\mathbf{B}_R(\mathbf{p})\setminus  \mathbf{B}_r(\mathbf{p})}\left|\na^N\r\right|^2|\th|^{\ell_*}\int_\r^R\f{ds}{s^{n+1}} d\mu_V\\
\le&\f{\ell_*^2}{2^{n+2}}\int_r^R\mathcal{W}_{\mathbf{p},\ell_*}(s)ds+\int_{ \mathbf{B}_r(\mathbf{p})}\f{\left|\na^N\r\right|^2}{nr^n}|\th|^{\ell_*} d\mu_V+\int_{\mathbf{B}_R(\mathbf{p})\setminus \mathbf{B}_r(\mathbf{p})}\f{|\na^N\r|^2}{n\r^{n}}|\th|^{\ell_*}d\mu_V.
\endaligned
\end{equation}
Combining \eqref{Almost-mono-thell} and \eqref{Phil*est}, we obtain
\begin{equation}\aligned\label{mon-thell*}
&\f{e^{\k R}}{R^{n}}\int_{\mathbf{B}_R(\mathbf{p})}|\th|^{\ell_*} d\mu_V-\f{e^{\k r}}{r^{n}}\int_{ \mathbf{B}_r(\mathbf{p})}|\th|^{\ell_*} d\mu_V\\
\ge& \f{\ell_*^2}{2^{n+2}}\int_{r/2}^{R/2}\mathcal{W}_{\mathbf{p},\ell_*}(s)ds-\f{1}{nr^n}\int_{\mathbf{B}_r(\mathbf{p})}|\th|^{\ell_*}|\na^N\r|^2d\mu_V\\
&+ \f{n-1}n\int_{\mathbf{B}_R(\mathbf{p})\setminus \mathbf{B}_r(\mathbf{p})}|\th|^{\ell_*}\f{|\na^N\r|^2}{\r^{n}}d\mu_V-\f{\ell_*^2}{2^{n+2}}\int_{R/2}^R\mathcal{W}_{\mathbf{p},\ell_*}(s)ds.
\endaligned
\end{equation}
From \eqref{mon-thell*}, it's clear that
\begin{equation}\aligned\label{Wpl*inf}
\int_0^{R}\mathcal{W}_{\mathbf{p},\ell_*}(s)ds+\int_{\mathbf{B}_R(\mathbf{p})}|\th|^{\ell_*}\f{|\na^N\r|^2}{\r^{n}}d\mu_V<\infty,
\endaligned
\end{equation}
which implies 
\begin{equation}\aligned\label{Wpl*to0}
\lim_{r\to0}\f{1}{r^n}\int_{\mathbf{B}_r(\mathbf{p})}|\th|^{\ell_*}|\na^N\r|^2d\mu_V\le\lim_{r\to0}\int_{\mathbf{B}_r(\mathbf{p})}|\th|^{\ell_*}\f{|\na^N\r|^2}{\r^{n}}d\mu_V=0.
\endaligned
\end{equation}

Given the constant $\ell_*$, for each $0<r\le R\le R_1$ we define $\mathbf{m}_{\mathbf{p}}(\mu_V,r)$ as \eqref{ThpmuVr0}, i.e.,
\begin{equation}\aligned\label{DEF-bfTh}
\mathbf{m}_{\mathbf{p}}(\mu_V,r)=\f1{r^n}\int_{\mathbf{B}_{r}(\mathbf{p})}|\th|^{\ell_*} d\mu_V+\f{\ell_*^2e^{-\k r}}{2^{n+2}}\int_{r/2}^r\mathcal{W}_{\mathbf{p},\ell_*}(s)ds.
\endaligned
\end{equation}
From \eqref{mon-thell*}, we get
\begin{equation}\aligned\label{mon-bfTh}
e^{\k R}\mathbf{m}_\mathbf{p}(\mu_V,R)\ge&e^{\k r}\mathbf{m}_\mathbf{p}(\mu_V,r)+ \f{n-1}{n}\int_{\mathbf{B}_R(\mathbf{p})\setminus \mathbf{B}_r(\mathbf{p})}|\th|^{\ell_*}\f{|\na^N\r|^2}{\r^{n}}d\mu_V\\
&+\f{\ell_*^2}{2^{n+2}}\int_r^{R/2}\mathcal{W}_{\mathbf{p},\ell_*}(s)ds-\f{1}{nr^n}\int_{\mathbf{B}_r(\mathbf{p})}|\th|^{\ell_*}|\na^N\r|^2d\mu_V.
\endaligned
\end{equation}
This completes the proof.
\end{proof}
Noting $|\na^N\r|\le1$. 
From \eqref{Wpl*inf} and \eqref{Wpl*to0}, the limit $\lim_{r\to0}\mathbf{m}_{\mathbf{p}}(\mu_V,r)$ exists, denoted by $\mathbf{m}_{\mathbf{p}}(\mu_V)$.
Letting $r\to0$ in \eqref{mon-bfTh} gives
\begin{equation}\aligned\label{Integrable-W}
\f{\ell_*^2}{2^{n+2}}\int_{0}^{R/2}\mathcal{W}_{\mathbf{p},\ell_*}(s)ds+\f12\int_{\mathbf{B}_R(\mathbf{p})}\f{|\na^N\r|^2}{\r^{n}}|\th|^{\ell_*} d\mu_V\le e^{\k R}\mathbf{m}_\mathbf{p}(\mu_V,R)-\mathbf{m}_\mathbf{p}(\mu_V).
\endaligned
\end{equation}

\begin{proposition}\label{semi-bfTh}
The function $\mathbf{m}_\mathbf{p}(\mu_V)$ is upper semi-continuous on $\mathbf{p}$.
\end{proposition}
\begin{proof}
The proof is similar to the proof of Corollary 1.14 in Colding-Minicozzi's book \cite{CM1}.
Let $\mathbf{p}_i$ be a sequence of points converging to $\mathbf{p}$.
For a small $r>0$, we denote $r_i=r-|\mathbf{p}-\mathbf{p}_i|$. From \eqref{Integrable-W} and the definition of $\mathbf{m}$, we have
\begin{equation}\aligned
&\mathbf{m}_{\mathbf{p}_i}(\mu_V)\le e^{\k r_i}\mathbf{m}_{\mathbf{p}_i}(\mu_V,r_i)\\
\le&\f{e^{\k r_i}}{r_i^n}\int_{\mathbf{B}_{r}(\mathbf{p})}|\th|^{\ell_*} d\mu_V+\f{\ell_*^2}{2^{n+2}}\int_{r_i/2}^{r_i}s^{1-n}\int_{\mathbf{B}_{s+|\mathbf{p}-\mathbf{p}_i|}(\mathbf{p})}|\th|^{\ell_*-2}|H_V|^2 d\mu_Vds\\
\le&\f{e^{\k r_i}}{r_i^n}\int_{\mathbf{B}_{r}}|\th|^{\ell_*} d\mu_V+\f{\ell_*^2}{2^{n+2}}\left(\f{r+|\mathbf{p}-\mathbf{p}_i|}{r_i}\right)^{n-1}\int_{r/2}^{r}t^{1-n}\int_{\mathbf{B}_{t}(\mathbf{p})}|\th|^{\ell_*-2}|H_V|^2 d\mu_Vdt\\
\le&\left(1+2\f{|\mathbf{p}-\mathbf{p}_i|}{r_i}\right)^n e^{\k r}\mathbf{m}_\mathbf{p}(\mu_V,r).
\endaligned
\end{equation}
This completes the proof by choosing suitably small $r>0$.
\end{proof}
\begin{remark}\label{tildebfTh}
Let $\tilde{\th}=\sup_{x\in L}\th(x)-\th$, and 
\begin{equation}\aligned
\tilde{\mathbf{m}}_{\mathbf{p}}(\mu_V,r)=\f1{r^n}\int_{\mathbf{B}_{r}(\mathbf{p})}\tilde{\th}^{\ell_*} d\mu_V+\f{\ell_*^2e^{-\k r}}{2^{n+2}}\int_{r/2}^rs^{1-n}\int_{\mathbf{B}_{s}(\mathbf{p})}\tilde{\th}^{\ell_*-2}|H_V|^2 d\mu_Vds.
\endaligned
\end{equation}
From \eqref{BrxnaVthell0} and \eqref{BrxnaVthell}, we have
\begin{equation}\aligned\label{tildethiell}
\int_{\mathbf{B}_s(\mathbf{p})}\lan \r\na^L\r,\na^L \tilde{\th}^{\ell_*}\ran d\mu_V
=\f{\ell(\ell-1)}2\int_{\mathbf{B}_s(\mathbf{p})}(s^2-\r^2)\tilde{\th}^{\ell_*-2}|H_V|^2 d\mu_V.
\endaligned
\end{equation}
Analog to the above argument, the limit $\lim_{r\to0}\tilde{\mathbf{m}}_{\mathbf{p}}(\mu_V,r)$ exists, denoted by $\tilde{\mathbf{m}}_{\mathbf{p}}(\mu_V)$. Compared with \eqref{Integrable-W}, we have
\begin{equation}\aligned
\tilde{\mathbf{m}}_\mathbf{p}(\mu_V)\le e^{\k R}\tilde{\mathbf{m}}_\mathbf{p}(\mu_V,R).
\endaligned
\end{equation}
\end{remark}
For each $0<r\le R_1$, we define the density function
\begin{equation}\aligned
\Th_{\mathbf{p}}(\mu_V,r):=\f{\mu_V(\mathbf{B}_{r}(\mathbf{p}))}{\omega_nr^n}.
\endaligned
\end{equation}
We further assume $\th\ge1$.
From \eqref{Almost-mono-Vinfty} in Appendix II and Cauchy-Schwartz inequality, for all $0<r\le R\le R_1$ we have
\begin{equation}\aligned\label{almost-mono-Th}
&e^{\k R}\Th_\mathbf{p}(\mu_V,R)-e^{\k r}\Th_\mathbf{p}(\mu_V,r)\\
\ge&\int_{\mathbf{B}_R\setminus \mathbf{B}_r}\f{e^{\k \r}}{\r^n}\left|\na^N\r\right|^2 d\mu_V-\int_r^R\f{e^{\k s}}{s^{n+1}}\int_{\mathbf{B}_s}\left(|\na^N\r|^2+\f{\r^2}{4}|H_V|^2\right) d\mu_Vds\\
\ge& -c_ne^{\k R}(e^{2\k R}\mathbf{m}_\mathbf{p}(\mu_V,2R)-\mathbf{m}_\mathbf{p}(\mu_V)),
\endaligned
\end{equation}
where we obtain the last inequality in \eqref{almost-mono-Th} from \eqref{Integrable-W}, and
 $c_n$ is a positive constant  depending only on $n$. Clearly, \eqref{almost-mono-Th} implies that the limit $\lim_{r\to0}\Th_{\mathbf{p}}(\mu_V,r)$ exists, denoted by $\Th_{\mathbf{p}}(\mu_V)$.
Combining Theorem \ref{Low-V-Brp}, we get the following volume estimates.
\begin{theorem}\label{Ahlfors}
There is a constant $c^*\ge1$ depending only on $n,q,\k R_0,\La_1,\La_2$ so that
\begin{equation*}\aligned
\f{1}{c^*}\le \Th_{x}(\mu_V,r)\le c^*\qquad \mathrm{for\ all}\ x\in L\cap \mathbf{B}_{R_0},\ 0<r<R_0-|x|.
\endaligned
\end{equation*}
\end{theorem}

From \eqref{Integrable-W}, every tangent cone of $V$ is a minimal Lagrangian cone $C\subset\C^n$ with integer multiplicity and vertex at the origin $0^{2n}$. In particular,
\begin{equation}\aligned
\mathbf{M}(C\llcorner B_r(0^{2n}))\ge\omega_nr^{n}.
\endaligned
\end{equation}
By the convergence of Radon measure, we immediately have the following sharp lower bound on the density.
\begin{corollary}\label{lower}
For each $x\in \mathrm{spt}V\cap \mathbf{B}_{R_0}$, there holds $\Th_x(\mu_V)\ge1$.
\end{corollary}

\section{Lagrangian currents with single-valued harmonic phases: regularity}

Let $(M^n,g,J,\omega)$ be a Ricci-flat K\"ahler manifold with a non-zero holomorphic $n$-form $\Om$ in a bounded open subset of $U\subset M$ satisfying $\Om\wedge\overline{\Om}=(-1)^{-n^2/2}2^{n}\omega^n/n!$. Up to scaling,
we assume that $U$ is isometrically embedded in $\R^m$ with 
$\p U\cap\mathbf{B}_1=\emptyset$ for the unit ball $\mathbf{B}_1\subset\R^m$. 
Let $\mathbf{A}_U$ denote the second fundamental form of $U$, and $\k$ be a constant satisfying 
\begin{equation}\aligned\label{def-k}
|\mathbf{A}_U(X,X)|\le\k|X|^2/n\qquad \mathrm{for\ any\ vector\ field}\ X\ \mathrm{tangent\ to}\ U. 
\endaligned
\end{equation}
Let $T=(L,\vth,\vec{T})$ be an integral Lagrangian current in $U$ with single-valued harmonic phase $\th$ and $\p T\llcorner U=0$.
We assume $0\le\th\le\tilde{\La}$ for some constant $\tilde{\La}>1$.
Denote $V=\vth|L|$.

For an $n$-dimensional subspace $P$ in $\R^m$,
let $\Pi_P$ denote the projection into $P$.
We define the excess of $V$ w.r.t. $P$ at $\mathbf{p}\in\mathbf{B}_1$ by
\begin{equation}\aligned
E(\mu_V,\mathbf{p},r,P)=r^{-n}\int_{\mathbf{B}_r(\mathbf{p})}\left|\Pi_{T_\mathbf{p}L}-\Pi_P\right|^2 d\mu_V.
\endaligned
\end{equation}
Put $\mathbf{R}^n:=\R^n\times\{0^{m-n}\}\subset\R^m$.
Let $\mathbf{x}=(\mathbf{x}_1,\cdots,\mathbf{x}_m)=\sum_{1\le i\le m}\mathbf{x}_iE_i$ denote the position vector in $\R^m$ with a standard basis $\{E_i\}$ of $\R^m$, then (see (4.2) in \cite{S})
\begin{equation}\aligned\label{EmuVprRn}
E(\mu_V,\mathbf{p},r,\mathbf{R}^n)=2r^{-n}\int_{\mathbf{B}_r(\mathbf{p})}\sum_{j=n+1}^m\left|\na^L\mathbf{x}_j\right|^2 d\mu_V.
\endaligned
\end{equation}

Let $\tau_{\mathbf{p},t}(\mathbf{x})=\f1t(\mathbf{x}-\mathbf{p})$ for all $\mathbf{x},\mathbf{p}\in\R^m$ and $t>0$.
\begin{lemma}\label{smallde1}
Given any $\de_0>0$, there are 3 constants $\de_1,\varrho,\tilde{\varrho}\in(0,\f14)$ with $\tilde{\varrho}<\de_1\varrho/4$ depending only on $n,\de_0,\tilde{\La}$ and $U$ so that if $\mathbf{0}\in L$,
\begin{equation}\aligned\label{Th0Em0Cond}
\Th_{\mathbf{0}}(\mu_V,\varrho)\le2-\de_0,\quad E(\mu_V,\mathbf{0},\varrho,\mathbf{R}^n)<\de_1\quad \text{and}\quad \mathbf{m}_\mathbf{0}(\mu_V,\varrho)\le(1+\de_1)\mathbf{m}_\mathbf{0}(\mu_V),
\endaligned
\end{equation}
then
\begin{equation}\aligned\label{Th_0=th}
\Th_{x}(\mu_V)=1\qquad \text{for each}\ x\in L\cap\mathbf{B}_{\tilde{\varrho}}.
\endaligned
\end{equation}
\end{lemma}
\begin{proof}
Let us first prove $\Th_{\mathbf{0}}(\mu_V)=1$ by contradiction. We suppose that there are 
a sequence $\varrho_i\to0$, a sequence of Lagrangian currents $T_i=(L_i,\vth_i,\vec{T}_i)$ in $U$ with single-valued harmonic phase $\th_i$ satisfying $1\le\th_i\le\tilde{\La}$ and $\de_*>0$ so that 
\begin{equation}\aligned
\Th_{\mathbf{0}}(\mu_{V_i},\varrho_i)\le2-\de_0,\quad E(\mu_{V_i},\mathbf{0},\varrho_i,\mathbf{R}^n)<\f1i,\quad \mathbf{m}_\mathbf{0}(\mu_{V_i},\varrho_i)\le\f{1+i}i\mathbf{m}_\mathbf{0}(\mu_{V_i}),
\endaligned
\end{equation}
and
\begin{equation}\aligned
\Th_{\mathbf{0}}(\mu_{V_i})\ge1+\de_*.
\endaligned
\end{equation}
From $\mathbf{m}_\mathbf{0}(\mu_{V_i},\varrho_i)\le\f{1+i}i\mathbf{m}_\mathbf{0}(\mu_{V_i})$ and \eqref{almost-mono-Th},
we have
\begin{equation}\aligned
\Th_{\mathbf{0}}(\mu_{V_i},\varrho_i/2)\ge1+\de_*/2\qquad
\mathrm{for\ all\ suitably\ large}\ i.
\endaligned
\end{equation}
Let $U_i:=\tau_{\mathbf{0},\varrho_i^{-1}}(U\cap \mathbf{B}_1)=\varrho_i^{-1}(U\cap \mathbf{B}_1)$ 
and $\mathbf{A}_i$ denote the second fundamental form of $U_i$, 
then from \eqref{def-k} it follows that
\begin{equation}\aligned
|\mathbf{A}_i(X,X)|\le\k\varrho_i|X|^2/n\qquad \mathrm{for\ any\ vector\ field}\ X\ \mathrm{tangent\ to}\ U_i. 
\endaligned
\end{equation}

Combining \eqref{Integrable-W} and the compactness of varifolds, up to choosing a subsequence we assume that $(\tau_{\mathbf{0},\varrho_i^{-1}})_{\#}V_i$ converges to a stationary varifold $V_\infty$ in $\R^{m}$ in the sense of Radon measure.
From the convergence of Radon measure, it follows that 
$$\Th_{\mathbf{0}}(\mu_{V_\infty},1)\le2-\de_0,\qquad E(\mu_{V_\infty},\mathbf{0},1,\mathbf{R}^n)=0\qquad \text{and}\qquad \Th_{\mathbf{0}}(\mu_{V_\infty},1/2)\ge1+\de_*/2.$$
Clearly, $E(\mu_{V_\infty},\mathbf{0},1,\mathbf{R}^n)=0$ implies that spt$V_\infty$ is flat. By Constancy Theorem for the stationary $V_\infty$ (see \cite{S} for instance), $V_\infty=k|\mathbf{R}^n|$ for some integer $k\ge1$. However,
it's a contradiction since $\Th_{\mathbf{0}}(\mu_{V_\infty},1)\le2-\de_0$ and
$\Th_{\mathbf{0}}(\mu_{V_\infty},1/2)\ge1+\de_*/2$. Hence, we obtain $\Th_{\mathbf{0}}(\mu_V)=1$.

Given a point $\mathbf{p}\in\mathbf{B}_{\de_1\varrho}$, we set $r:=\varrho/2-|\mathbf{p}|$.
From the definition of $\tilde{\mathbf{m}}$ in Remark \ref{tildebfTh}, we have
\begin{equation}\aligned\label{tildebfThpmuVr}
&\tilde{\mathbf{m}}_\mathbf{p}(\mu_V)\le e^{\k r}\tilde{\mathbf{m}}_\mathbf{p}(\mu_V,r)\\
\le&\f{e^{\k r}}{r^n}\int_{\mathbf{B}_{\varrho/2}}\tilde{\th}^{\ell_*} d\mu_V+\f{\ell_*^2}{2^{n+2}}\int_{r/2}^{r}s^{1-n}\int_{\mathbf{B}_{s+|\mathbf{p}|}}\tilde{\th}^{\ell_*-2}|H_V|^2 d\mu_Vds\\
\le&\f{e^{\k r}}{r^n}\int_{\mathbf{B}_{\f \varrho2}}\tilde{\th}^{\ell_*} d\mu_V+\f{\ell_*^2}{2^{n+2}}\sup_{\f r2\le t-|\mathbf{p}|\le r}\left(\f t{t-|\mathbf{p}|}\right)^{n-1}\int_{\f \varrho4}^{\f \varrho2}t^{1-n}\int_{\mathbf{B}_{t}}\tilde{\th}^{\ell_*-2}|H_V|^2 d\mu_Vdt\\
\le&\left(1+\f{2|\mathbf{p}|}{r}\right)^n e^{\k \varrho/2}\tilde{\mathbf{m}}_\mathbf{0}(\mu_V,\varrho/2).
\endaligned
\end{equation}
On the other hand, from \eqref{Integrable-W} it follows that
\begin{equation}\aligned
\f{\ell_*^2}{2^{n+2}}\int_{0}^{\varrho/2}\mathcal{W}_{\mathbf{0},\ell_*}(s)ds+\f12\int_{\mathbf{B}_\varrho}\f{|\na^N\r|^2}{\r^{n}}\th^{\ell_*} d\mu_V\le e^{\k \varrho}\mathbf{m}_\mathbf{0}(\mu_V,\varrho)-\mathbf{m}_\mathbf{0}(\mu_V).
\endaligned
\end{equation}
Here, $\r(\mathbf{x})=|\mathbf{x}|$ for every $\mathbf{x}\in\R^m$.
Combining \eqref{Almost-mono-Vinfty**} in Appendix II and \eqref{tildethiell}\eqref{Th0Em0Cond}, we have
\begin{equation}\aligned\label{tthl*tmomuV}
&\f{2^ne^{-\k\varrho/2}}{\varrho^n}\int_{\mathbf{B}_{\varrho/2}}\tilde{\th}^{\ell_*}d\mu_{V}-\tilde{\mathbf{m}}_\mathbf{0}(\mu_V)\\
\le&\int_{\mathbf{B}_{\varrho/2}}\f{\left|\na^N\r\right|^2}{\r^n}\tilde{\th}^{\ell_*}d\mu_V
+\int_0^{\varrho/2}\f{e^{-\k s}}{s^{n+1}}\int_{\mathbf{B}_s}\lan \tilde{\th}^{\ell_*}H+\na^L \tilde{\th}^{\ell_*},\r\na\r\ran d\mu_Vds\\
\le&\f{c_n\ell_*^2}{2^{n+2}}\int_{0}^{\varrho/2}\mathcal{W}_{\mathbf{0},\ell_*}(s)ds+\f{c_n}2\int_{\mathbf{B}_{\varrho/2}}\f{|\na^N\r|^2}{\r^{n}}\th^{\ell_*} d\mu_V\\
\le& c_n\left(e^{\k \varrho}\mathbf{m}_\mathbf{0}(\mu_V,\varrho)-\mathbf{m}_\mathbf{0}(\mu_V)\right)\le c_n\left((e^{\k \varrho}-1)\mathbf{m}_\mathbf{0}(\mu_V,\varrho)+\de_1\mathbf{m}_\mathbf{0}(\mu_V)\right).
\endaligned
\end{equation}
Here, $c_n$ is a constant depending only on $n$.  
Combining \eqref{tildebfThpmuVr} and \eqref{tthl*tmomuV}, there is a general function $\psi(\de_1,\varrho)$ depending on $n,\tilde{\La}$ with $\lim_{\de_1,\varrho\to0}\psi(\de_1,\varrho)=0$ so that
\begin{equation}\aligned\label{tmptm0psider}
\tilde{\mathbf{m}}_\mathbf{p}(\mu_V)\le \tilde{\mathbf{m}}_\mathbf{0}(\mu_V)+\psi(\de_1,\varrho).
\endaligned
\end{equation}
Noting $\th\ge0$. It's clear that
\begin{equation}\aligned\label{mp+mtiltep*}
\left(\f{\mathbf{m}_\mathbf{p}(\mu_V)}{\omega_n}\right)^{1/\ell_*}+\left(\f{\tilde{\mathbf{m}}_\mathbf{p}(\mu_V)}{\omega_n}\right)^{1/\ell_*}=\Th_\mathbf{p}(\mu_V)\sup_L\th.
\endaligned
\end{equation}
Combining \eqref{tmptm0psider} and \eqref{mp+mtiltep*}, it follows that
\begin{equation}\aligned
\Th_\mathbf{p}(\mu_V)\sup_L\th-\left(\f{\mathbf{m}_\mathbf{p}(\mu_V)}{\omega_n}\right)^{1/\ell_*}\le \Th_\mathbf{0}(\mu_V)\sup_L\th-\left(\f{\mathbf{m}_\mathbf{0}(\mu_V)}{\omega_n}\right)^{1/\ell_*}+\psi(\de_1,\varrho).
\endaligned
\end{equation}
From Corollary \ref{lower} and $\Th_{\mathbf{0}}(\mu_V)=1$, we get
\begin{equation}\aligned
\left(\f{\mathbf{m}_\mathbf{0}(\mu_V)}{\omega_n}\right)^{1/\ell_*}\le \left(\f{\mathbf{m}_\mathbf{p}(\mu_V)}{\omega_n}\right)^{1/\ell_*}+\psi(\de_1,\varrho).
\endaligned
\end{equation}
From \eqref{mon-bfTh} (with $\mathbf{p}$ at the origin) and \eqref{tildebfThpmuVr} (with $\tilde{\mathbf{m}}$ replaced by $\mathbf{m}$), it follows that
\begin{equation}\aligned\label{BfThpmurpsi}
\mathbf{m}_\mathbf{p}(\mu_V,r)\le&\left(1+2\f{|\mathbf{p}|}{r}\right)^n e^{\k \varrho/2}\mathbf{m}_\mathbf{0}(\mu_V,\varrho/2)\le(1+\psi(\de_1,\varrho))\mathbf{m}_\mathbf{0}(\mu_V,\varrho)\\
\le&(1+\psi(\de_1,\varrho))(1+\de_1)\mathbf{m}_\mathbf{0}(\mu_V)\le(1+\psi(\de_1,\varrho))\mathbf{m}_\mathbf{p}(\mu_V),
\endaligned
\end{equation}
where $r=\varrho/2-|\mathbf{p}|$ as before. Using the above proof for $\Th_{\mathbf{0}}(\mu_V)=1$, we complete the proof of \eqref{Th_0=th}.
\end{proof}
\begin{remark}
$\de_0>0$ is necessary in Lemma \ref{smallde1} since the limit in \eqref{Th_0=th} may be 2 if $T=(B_1(0^n),2,E_1\wedge\cdots\wedge E_n)$ in $\R^{2n}$.
\end{remark}

Compared with Definition \ref{B-Gindec}, we define a quantitative version for decomposable components as follows.
\begin{definition}\label{decom}
Let $\mathcal{M}$ denote a complete Riemannian manifold.
For each $\ep\ge0$, $\de\in(0,1/2]$ and an integer $0\le k\le\mathrm{dim}\mathcal{M}$, an integral current $T\in\mathcal{D}_k(\mathcal{M})$ is said to be $(\ep,\de)$-\emph{decomposable} in an open $\mathcal{U}\subset\mathcal{M}$ if there exist integral currents $T_1,T_2\in\mathcal{D}_k(\mathcal{U})$  such that
\begin{equation}\aligned\nonumber
&\mathbf{M}(T\llcorner W)=\mathbf{M}(T_1\llcorner W)+\mathbf{M}(T_2\llcorner W)\\
&\mathbf{M}(\p T\llcorner W)+\ep\ge\mathbf{M}(\p T_1\llcorner W)+\mathbf{M}(\p T_2\llcorner W)\\
&\min\{\mathbf{M}(T_1\llcorner W),\mathbf{M}(T_2\llcorner W)\}\ge\de\mathbf{M}(T\llcorner W)
\endaligned
\end{equation}
for any $W\subset\subset \mathcal{U}$. Here, $T_1,T_2$ are called $(\ep,\de)$-\emph{components} of $T\llcorner \mathcal{U}$. 
\end{definition}
We further define decomposable components for sequences as follows.
\begin{definition}\label{seq-decom}
For a sequence of integral currents $\{T_\ell\}\subset\mathcal{D}_k(\mathcal{M})$, we say $\{T_\ell\}$ \emph{decomposable} in $\mathcal{U}$
if there exists a sequence $\ep_\ell\to0$, $\de\in(0,1/2]$ and a sequence of integral currents $\{T_\ell^+\},\{T_\ell^-\}\subset\mathcal{D}_k(\mathcal{U})$ so that $T_\ell^+,T_\ell^-$ are $(\ep_\ell,\de)$-\emph{components} of $T_\ell\llcorner \mathcal{U}$. We call $\{T_\ell^+\},\{T_\ell^-\}$ sequences of \emph{components} of $\{T_\ell\llcorner \mathcal{U}\}$. 
On the contrary, we say $\{T_\ell\}$ \emph{indecomposable} in $\mathcal{U}$.
\end{definition}

From Lemma \ref{smallde1}, we have the following indecomposable results.
\begin{lemma}\label{Indecomp}
Given the sufficiently small $\de_1,\varrho$ in Lemma \ref{smallde1}, for any $\de\in(0,1/2]$ there are two constant $\ep_\de>0$, $\varrho_\de\in(0,\tilde{\varrho})$ depending only on $n,\de_0,\tilde{\La},\de,U$ so that
the current $T\llcorner \mathbf{B}_{r}(\mathbf{x})$ is $(\ep_\de r^{n-1},\de)$-indecomposable in $\mathbf{B}_{r}(\mathbf{x})$ for each $\mathbf{x}\in \mathbf{B}_{\varrho_\de}$ and $r\in(0,\varrho_\de]$.
\end{lemma}
\begin{proof}
We assume that there are a constant $\de>0$, two sequences of numbers $r_\ell,\ep_\ell\to0$, and a sequence of points $\mathbf{x}_\ell\in \mathbf{B}_{\varrho_*}$ so that $T^+_\ell,T^-_\ell$ are two $(\ep_\ell,\de)$-components of $T\llcorner \mathbf{B}_{r_\ell}(\mathbf{x}_\ell)$ for each $\ell$. Let 
$$\tilde{T}_\ell^\pm=(\tau_{\mathbf{x}_\ell,r_\ell})_\#T_\ell^\pm\qquad \text{and}\qquad \tilde{T}_\ell=\tilde{T}_\ell^++\tilde{T}_\ell^-,$$
then
\begin{equation}\aligned\label{tildeTellpm}
&\mathbf{M}(\p\tilde{T}_\ell^+\llcorner\mathbf{B}_1)+\mathbf{M}(\p\tilde{T}_\ell^-\llcorner\mathbf{B}_1)\le\ep_\ell,\\
&\min\{\mathbf{M}(\tilde{T}_\ell^+\llcorner\mathbf{B}_1),\mathbf{M}(\tilde{T}_\ell^+\llcorner\mathbf{B}_1)\}\ge\de\mathbf{M}(\tilde{T}_\ell\llcorner\mathbf{B}_1).
\endaligned
\end{equation}
Up to choosing subsequences, from Federer-Fleming compactness theorem there are 3 currents $\tilde{T}_\infty,\tilde{T}_\infty^+,\tilde{T}_\infty^-$ with integer multiplicities in $\mathbf{B}_1$ so that $\tilde{T}_\ell\rightharpoonup\tilde{T}_\infty$ and $\tilde{T}_\ell^\pm\rightharpoonup\tilde{T}_\infty^\pm$ all in the current sense.
From \eqref{BfThpmurpsi} and Lemma \ref{smallde1}, all 3 currents $\tilde{T}_\infty,\tilde{T}_\infty^+,\tilde{T}_\infty^-$ have multiplicity one, and $\tilde{T}_\infty$ has flat support through the origin in $\mathbf{B}_1$ with $\p\tilde{T}_\infty\llcorner\mathbf{B}_1=0$. Moreover, from \eqref{tildeTellpm} we have
\begin{equation}\aligned\label{tildeTinftypm}
&\mathbf{M}(\p\tilde{T}_\infty^+\llcorner\mathbf{B}_1)=\mathbf{M}(\p\tilde{T}_\infty^-\llcorner\mathbf{B}_1)=0,\\
&\min\{\mathbf{M}(\tilde{T}_\infty^+\llcorner\mathbf{B}_1),\mathbf{M}(\tilde{T}_\infty^+\llcorner\mathbf{B}_1)\}\ge\de\mathbf{M}(\tilde{T}_\infty\llcorner\mathbf{B}_1).
\endaligned
\end{equation}
Then it follows that both supports of $\tilde{T}_\infty^+$ and $\tilde{T}_\infty^-$ are $n$-dimensional round balls in $\mathbf{B}_1$, which gives $\mathbf{M}(\tilde{T}_\infty^+)=\mathbf{M}(\tilde{T}_\infty^-)=\omega_n=\mathbf{M}(\tilde{T}_\infty)$.
This contradicts to the inequality in \eqref{tildeTinftypm}. We complete the proof.
\end{proof}

Compared with Lemma 3.4 in \cite{D2}, we have the following Sobolev inequality on $L\cap\mathbf{B}_{\varrho_*}$ without mean curvature term.
\begin{lemma}
Let $\de=\f12$ and $\varrho_{\f12}$ be defined as in Lemma \ref{Indecomp},
there exists a constant $\de_2>0$ depending only on $n,\ep_{\f12}$ so that for any relatively open $W\subset L\cap\mathbf{B}_{\f12\varrho_{\f12}}$ there holds
\begin{equation}\aligned\label{Sob-special}
\de_2\left(\mathcal{H}^n(W)\right)^{\f{n-1}n}\le \mathcal{H}^{n-1}(\p W).
\endaligned
\end{equation}
\end{lemma}
\begin{proof}
Let us prove the lemma by contradiction. We set $\r=\f12\varrho_{\f12}$ for convenience in this proof. We assume that there is a relatively open $W\subset L\cap\mathbf{B}_{\r}$ and a small constant $\tau>0$ so that
\begin{equation}\aligned\label{pSthomega}
\mathcal{H}^{n-1}(\p W)<\tau\left(\mathcal{H}^n(W)\right)^{\f{n-1}n}.
\endaligned
\end{equation}
Without loss of generality, we assume $\mathcal{H}^n(W)\ge\f12\mathcal{H}^n(L\cap\mathbf{B}_{\r})$, or else we consider $L\cap\mathbf{B}_{\r}\setminus W$ replacing $W$.
In particular, the multiplicity $\vth\equiv1$ on $L\cap\mathbf{B}_{\r}$ from Lemma \ref{smallde1}.

For each $x\in W$, there is a constant $r_{x}>0$ such that
\begin{equation}\aligned\label{HnSrxhalf}
\mathcal{H}^n(W\cap \mathbf{B}_{r_{x}}(x))=\f12\mathcal{H}^n(L\cap \mathbf{B}_{r_{x}}(x))
\endaligned
\end{equation}
and
\begin{equation}\aligned\label{HnSrxhalf*}
\mathcal{H}^n(W\cap \mathbf{B}_{r}(x))>\f12\mathcal{H}^n(L\cap \mathbf{B}_{r}(x))\quad \mathrm{for\ all}\ 0<r<r_{x}.
\endaligned
\end{equation}
From Lemma \ref{smallde1} (for the suitably small $\de_1,\varrho$), we have 
\begin{equation}\aligned\label{Wupperb}
\mathcal{H}^n(W)\le\mathcal{H}^n(L\cap\mathbf{B}_{\r})\le 9\r^n\omega_n/8
\endaligned
\end{equation}
and
\begin{equation}\aligned\label{HnMBrx21}
\mathcal{H}^n(L\cap \mathbf{B}_{r}(x))\ge\f{8}{9}\omega_n r^n\qquad \mathrm{for\ any\ ball}\ \mathbf{B}_r(x)\subset \mathbf{B}_{\r}\ \mathrm{with}\ x\in L.
\endaligned
\end{equation}
Hence, $\mathcal{H}^n(W)\ge\f12\mathcal{H}^n(L\cap\mathbf{B}_{\r})\ge\f{4}{9}\omega_n \r^n$.
With \eqref{HnSrxhalf}-\eqref{HnMBrx21}, for any $0<r\le r_{x}$ we have
\begin{equation}\aligned\label{lowmeaW}
\f49\omega_n r^n\le\f12\mathcal{H}^n(L\cap \mathbf{B}_{r}(x))\le\mathcal{H}^n(W\cap \mathbf{B}_{r}(x))\le\mathcal{H}^n(W).
\endaligned
\end{equation}
The above inequality and \eqref{Wupperb} imply
\begin{equation}\aligned\label{rx2}
r_{x}\le \left(\f{9}{4\omega_n}\mathcal{H}^n(W)\right)^{\f 1{n}}\le \left(\f{81}{32}\r^n\right)^{\f 1{n}}\le \f95\r\le2\r.
\endaligned
\end{equation}
By the definition of $r_{x}$, $\overline{W}\subset\cup_{x\in W}\mathbf{B}_{r_{x}}(x)$ clearly.
From Besicovitch covering lemma and Theorem \ref{Ahlfors}, we can choose a finite collection of balls $\{\mathbf{B}_{r_{p_j}}(p_j)\}_{j=1}^{l}$ with $\overline{W}\subset \bigcup_{1\le j\le l}\mathbf{B}_{r_{p_j}}(p_j)$ such that the number of balls containing any point in $W$ is uniformly bounded, i.e., 
\begin{equation}\aligned
\sharp\{j\in\{1,\cdots,l\}|\ x\in \mathbf{B}_{r_{p_j}}(p_j)\}\le n_*\qquad \mathrm{for\ any}\ x\in W,
\endaligned
\end{equation}
where $n_*$ is a constant depending only on $n$.
Combining \eqref{pSthomega}\eqref{Wupperb}, we have
\begin{equation}\aligned
&\sum_{j=1}^{l}\mathcal{H}^{n-1}(\p W\cap \mathbf{B}_{r_{p_j}}(p_j))\le n_*\mathcal{H}^{n-1}(\p W)\\
\le& n_*\tau\left(\mathcal{H}^{n}(W)\right)^{\f{n-1}n}
\le\f{\tilde{n}_*\tau}{\r}\mathcal{H}^n(W)\le \f{\tilde{n}_*\tau}{\r}\sum_{j=1}^{l}\mathcal{H}^{n}(W\cap \mathbf{B}_{r_{p_j}}(p_j)).
\endaligned
\end{equation}
Here, $\tilde{n}_*$ is a constant depending only on $n$.
From \eqref{rx2}, there exists a point $p_{j_0}\in\{p_1,\cdots,p_l\}\subset \mathbf{B}_{\r/4}$ so that
\begin{equation}\aligned\label{HnpSthS}
\mathcal{H}^{n-1}(\p W\cap \mathbf{B}_{r_{p_{j_0}}}(p_{j_0}))\le \f{\tilde{n}_*\tau}{\r}\mathcal{H}^{n}(W\cap \mathbf{B}_{r_{p_{j_0}}}(p_{j_0}))\le \f{2\tilde{n}_*\tau}{r_{p_{j_0}}}\mathcal{H}^{n}(W\cap \mathbf{B}_{r_{p_{j_0}}}(p_{j_0})).
\endaligned
\end{equation}
This contradicts to Lemma \ref{Indecomp} for the suitably small $\tau>0$. We complete the prooof.
\end{proof}

\begin{lemma}\label{N-P}
For the suitably small $\de_1,\varrho>0$ in Lemma \ref{smallde1},
there exist constants $\de_3,\varrho_*>0$ so that for each $\mathbf{x}\in \mathbf{B}_{\varrho_*/2}\cap L$ and $r\in(0,\varrho_*/2]$ there holds
\begin{equation}\aligned
\de_3\min\{\mu_T(v_{\de_3 r}(\mathbf{x})\cap W),\mu_T( \mathbf{B}_r(\mathbf{x})\setminus W)\}\le \left(\mu_T( \mathbf{B}_{r}(\mathbf{x})\cap\p W)\right)^{\f n{n-1}}
\endaligned
\end{equation}
for any open $W\subset  \mathbf{B}_{\varrho_*}$ with countably $(n-1)$-rectifiable $L\cap\p W$.
\end{lemma}
\begin{proof}
Let us prove it by contradiction. Let $\de_2$ be the constant in \eqref{Sob-special}, and $\varrho_*=\varrho_\de$ with $\de:=\f1{2\omega_n}\left(\f{\de_2}{2n}\right)^n$.
We suppose that there are a small $0<\ep<\de_2/2$, an open $W\subset \mathbf{B}_{\varrho_*}$ with countably $(n-1)$-rectifiable $L\cap\p W$, $r\in(0,\varrho_*/2]$, $\mathbf{x}\in \mathbf{B}_{\varrho_*/2}$ such that
\begin{equation}\aligned\label{Tf1kUk}
\mu_{T}(\p W\cap \mathbf{B}_{r}(\mathbf{x}))
<\ep\left(\min\{\mu_{T}(W\cap \mathbf{B}_{\ep r}(\mathbf{x})),\mu_{T}(\mathbf{B}_{\ep r}(\mathbf{x})\setminus W)\}\right)^{\f{n-1}n}.
\endaligned
\end{equation}
For the fixed $\mathbf{x},r$, we set $T^*=T\llcorner \mathbf{B}_{r}(\mathbf{x})$, $T^+=T\llcorner (\mathbf{B}_{r}(\mathbf{x})\cap W)$ and $T^-=T\llcorner(\mathbf{B}_{r}(\mathbf{x})\setminus W)$. 
Clearly, $$T^*=T^++T^-.$$
From the co-area formula, for almost $t\in(0,r)$ we have
\begin{equation}\aligned\label{pptT*Bt}
\f{\p}{\p t}\mathbf{M}(T^\pm\llcorner \mathbf{B}_t(\mathbf{x}))\ge\mathbf{M}(\p(T^\pm\llcorner \mathbf{B}_t(\mathbf{x})))-\mathbf{M}(\p T^\pm\llcorner \mathbf{B}_t(\mathbf{x})).
\endaligned
\end{equation}
Combining Sobolev inequality \eqref{Sob-special} and \eqref{Tf1kUk}, for almost all $\ep r< t<r$ it follows that
\begin{equation}\aligned
\f{\p}{\p t}\mathbf{M}(T^\pm\llcorner \mathbf{B}_t(\mathbf{x}))>&\de_2\left(\mathbf{M}(T^\pm\llcorner \mathbf{B}_t(\mathbf{x}))\right)^{\f{n-1}n}-\ep\left(\mathbf{M}(T^\pm\llcorner \mathbf{B}_t(\mathbf{x}))\right)^{\f{n-1}n}\\
\ge&\f{\de_2}2\left(\mathbf{M}(T^\pm\llcorner \mathbf{B}_t(\mathbf{x}))\right)^{\f{n-1}n}.
\endaligned
\end{equation}
This implies $\mathbf{M}(T^\pm\llcorner \mathbf{B}_t(\mathbf{x}))>0$ for any $t>\ep$. Then we solve the above differential inequality and get
\begin{equation}\aligned
\mathbf{M}(T^\pm\llcorner \mathbf{B}_t(\mathbf{x}))\ge\left(\f{(t-\ep)\de_2}{2n}\right)^n\qquad\mathrm{for\ each}\ t\in[\ep r,r].
\endaligned
\end{equation}
In particular, from Lemma \ref{smallde1} for the small $\ep>0$ we have
\begin{equation}\aligned\label{T*Btlow}
\mathbf{M}(T^\pm\llcorner \mathbf{B}_r(\mathbf{x}))\ge\f1{2\omega_n}\left(\f{\de_2}{2n}\right)^n\mathbf{M}(T^*\llcorner \mathbf{B}_r(\mathbf{x}))=\de\mathbf{M}(T^*\llcorner \mathbf{B}_r(\mathbf{x})).
\endaligned
\end{equation}
Combining \eqref{Tf1kUk} and \eqref{T*Btlow}, we deduce a contradiction to Lemma \ref{Indecomp} for the suitably small $\ep$. This completes the proof.
\end{proof}

Recalling that $U$ is the open subset of the Ricci-flat K\"ahler manifold $M$ defined at the beginning of this section, and $\k$ is the constant satisfying \eqref{def-k}.
\begin{theorem}\label{NP}
There is a constant $0<\de<\f14$ depending only on $n,\tilde{\La}$ and $U$ so that if $T=(L,\vth,\vec{T})$ is an integral Lagrangian current in $U$ with single-valued harmonic phase $\th$ and $\p T\llcorner U=0$ so that $\mathbf{0}\in L$, $0\le\th\le\tilde{\La}$ for some constant $\tilde{\La}$, and $V=\vth|L|$ satisfies
\begin{equation}\aligned\label{app1cond}
\Th_{\mathbf{0}}(\mu_V,\de)\le1+\de,\qquad\mathrm{and}\qquad \mathbf{m}_\mathbf{0}(\mu_V,\de)\le(1+\de)\mathbf{m}_\mathbf{0}(\mu_V),
\endaligned
\end{equation}
then there holds the following Neumann-Poincar\'e inequality:
\begin{equation}\aligned\label{B1pPITh*}
\int_{\mathbf{B}_r(x)}|f-\bar{f}_{x,r}|d\mu_V\le cr \int_{\mathbf{B}_{cr}(x)}|\na^L f|d\mu_V
\endaligned
\end{equation}
on every ball $\mathbf{B}_{cr}(x)\subset \mathbf{B}_{\de/2}$ with $x\in L$, where $\bar{f}_{x,r}$ is the average of $f$ on $\mathbf{B}_{r}(x)$ w.r.t. $\mu_V$, i.e., $\bar{f}_{x,r}=\fint_{\mathbf{B}_{r}(x)}fd\mu_V$,
and $c>1$ is a constant depending only on $n,\tilde{\La}$ and $U$.
\end{theorem}
\begin{proof}
In rough speaking, the excess of a minimal submanifold is small if its density is close to 1  (see Lemma 6.3 of Chapter 5 in \cite{S}).  The argument also works for the varifold $V$ with generalized mean curvature $H_V$ as follows.
From the proof of Lemma 6.3 of Chapter 5 in \cite{S}, up to a rotation of $\mathbf{R}^n$ there is a general function $\psi\ge0$ on $(0,r_\de]$ with $\lim_{r\to0}\psi(r)=0$ so that 
\begin{equation}\aligned\label{supxRndesi}
\sup\{\mathrm{dist}(\mathbf{x},\mathbf{R}^n):\, \mathbf{x}\in\mathrm{spt}V\cap \mathbf{B}_{s}\}<\psi(\de)s\qquad\mathrm{for\ each}\ s\in(0,\varrho].
\endaligned
\end{equation}
For each $\mathbf{x}=(\mathbf{x}_1,\cdots,\mathbf{x}_m)$, we define $\tilde{\mathbf{x}}=(\mathbf{x}_{n+1},\cdots,\mathbf{x}_m)\in\R^{n-m}$. If $T_\mathbf{x}L$ is an $n$-plane, then
for an orthonormal basis $\{e_i\}_{1\le i\le n}$ of $T_\mathbf{x}L$, we have
\begin{equation}\aligned\label{divLtx}
\mathrm{div}_L\tilde{\mathbf{x}}=&\sum_{1\le i\le n<j\le m}\lan \na_{e_i}(\mathbf{x}_jE_j),e_i\ran=\sum_{1\le i\le n<j\le m}\lan E_j,e_i\ran^2\\
=&\sum_{n<j\le m}|E_j^L|^2=\sum_{n<j\le m}|\na^L\mathbf{x}_j|^2=|\na^L\tilde{\mathbf{x}}|^2,
\endaligned
\end{equation}
where $E_j^L$ denotes the projection into $TL$.
For each $\zeta\in C^\infty_0(\mathbf{B}_s)$, from \eqref{divLtx}, the definition of $\k$ and \eqref{bfHHAM}, with Cauchy-Schwartz inequality we have
\begin{equation}\aligned
&\int \zeta^2|\na^L\tilde{\mathbf{x}}|^2 d\mu_V=\int\mathrm{div}_L(\z^2\tilde{\mathbf{x}})d\mu_V-\int\lan \tilde{\mathbf{x}},\na^L\zeta^2\ran d\mu_V\\
\le&\int \lan\z^2\tilde{\mathbf{x}},H_V\ran d\mu_V+\int \k\z^2|\tilde{\mathbf{x}}| d\mu_V-2\sum_{n<j\le m}\int\zeta \tilde{\mathbf{x}}_j\lan E_j^L,\na\zeta\ran d\mu_V\\
\le& \f1{s^2}\int\z^2|\tilde{\mathbf{x}}|^2 d\mu_V+\f{s^2}2\int\z^2\left(|H_V|^2+\k^2\right) d\mu_V\\
&+2\int|\tilde{\mathbf{x}}|^2|\na\z|^2 d\mu_V+\f12\int\z^2\sum_{n<j\le m}|E_j^L|^2 d\mu_V.
\endaligned
\end{equation}
Let $\zeta=1$ on $\mathbf{B}_{s/2}$, $\zeta=0$ outside $\mathbf{B}_s$ and $|\na\zeta|\le 3/s$, then
\begin{equation}\aligned\label{Bs2snaLtx}
&\int_{\mathbf{B}_{s/2}}|\na^L\tilde{\mathbf{x}}|^2 d\mu_V\le c_ns^{-2}\int_{\mathbf{B}_{s}}|\tilde{\mathbf{x}}|^2 d\mu_V+s^2\int_{\mathbf{B}_{s}}\left(|H_V|^2+\k^2\right) d\mu_V.
\endaligned
\end{equation}
Combining \eqref{Integrable-W} and \eqref{supxRndesi}, we obtain
\begin{equation}\aligned
\qquad E(\mu_V,\mathbf{0},s/2,\mathbf{R}^n)<\psi(\de)\qquad\mathrm{for\ each}\ s\in(0,\varrho].
\endaligned
\end{equation}
From Lemma \ref{N-P}, we complete the proof by a standard argument (see the proof of Theorem 5.5 in \cite{D2}).
\end{proof}

By De Giorgi-Nash-Moser iteration, we can derive the following H\"older and Morrey estimates.
\begin{lemma}
If the conditions in Theorem \ref{NP} hold,
then there are constants $\a\in(0,1)$ and $c>0$ depending only on $n,\tilde{\La}$ and $U$ so that
\begin{equation}\aligned
|\th(\mathbf{x})-\th(\mathbf{x}')|\le c |\mathbf{x}-\mathbf{x}'|^\a
\endaligned
\end{equation}
for each $\mathbf{x},\mathbf{x}'\in L\cap\mathbf{B}_r$ with $0<r\le\de/8$ ,and
\begin{equation}\aligned\label{Ha-reg}
\int_{\mathbf{B}_r}|H_V|^2d\mu_V=\int_{\mathbf{B}_r}|\na^L\th|^2d\mu_V\le cr^{n-2+2\a}.
\endaligned
\end{equation}
\end{lemma}
Due to \eqref{Ha-reg}, we can apply Theorem 1.3 in \cite{BV} by Bourni-Volkmann, and complete the proof of Theorem \ref{Allard}.

\section{Lagrangian currents with single-valued harmonic phases: limits}

Let $\{p_j\}_{1\le j\le j_*}$ be a finite collection of points in ${\mathbb{S}}^1(1)\subset\R^2$, and $\g_{p_j}=\{tp_j\in \R^2:\, t\ge0\}$. Let $\g_*:=\cup_{1\le j\le j_*}\g_{p_j}$, then
$$|\g_*|:=\sum_{1\le j\le j_*}|\g_{p_j}|$$
is a Lagrangian varifold in $\C^1=\R^1\times\R^1$. Let $0^{n-1}$ denote the origin of $\R^{n-1}$ and
$$\G_*=\g_*\times\R^{n-1}:=\{(x_1,x',y_1,0^{n-1}):\, (x_1,y_1)\in\g_*, x'\in\R^{n-1}\}\subset\R^n\times\R^n.$$ 
It's easy to see that $|\G_*|$ is a Lagrangian varifold in $\C^n$. Enlightened by the proof of Theorem 2 of \cite{s-s} by Schoen-Simon, we have the following rigidity result.
\begin{lemma}\label{Rig-Gcodim1}
Let $M_\ell$ be a sequence of Ricci-flat K\"ahler manifolds of uniformly bounded geometry such that $(M_\ell,p_\ell)\to(\C^n,0^{2n})$ in $C^3$ sense as $\ell\to\infty$. For each $\ell$, let $T_\ell$ be an integral Lagrangian $n$-current in $B_{2\ell}(p_\ell)\subset M_\ell$ with single-valued harmonic phases $\th_\ell$ so that $\sup_\ell\sup_{\mathrm{spt} T_\ell}|\th_\ell|<\infty$ and $\sup_\ell\ell^{-n}\mathbf{M}(T_\ell)<\infty$. If  $|T_\ell|$ converges to $|\G_*|$ locally in the sense of measure, then $\g_*$ is a union of lines.
\end{lemma}
\begin{proof}
From the assumption, there is a sequence of diffeomorphisms $\Phi_\ell:\, B_{\sqrt{2}\ell}(0^{2n})\to \Phi_\ell(B_{\sqrt{2}\ell}(0^{2n}))\subset B_{3\ell/2}(p_\ell)$ with $\Phi_\ell(0^{2n})=p_\ell$ so that
\begin{equation*}\aligned
\lim_{\ell\to\infty}\sup_{B_k(0^{2n})}\left(\left|\Phi_\ell-Id\right|+\sum_{1\le i\le 3}|\na^i\Phi_\ell|\right)=0\qquad\mathrm{for\ each\ fixed}\ k\ge1.
\endaligned
\end{equation*}
Denote $T_\ell=(L_\ell,\vth_\ell,\vec{T}_\ell)$, $V_\ell=|T_\ell|=\vth_\ell|L_\ell|$.
Put $\R^{n+1}_y:=\R^{n+1}\times\{y\}$ for each $y\in\R^{n-1}$, and $\mathscr{B}^{\ell}_y:=\Phi_\ell(B_{1}(0^{n+1})\times\{y\})$ for each $y\in B_{\sqrt{2\ell^2-1}}(0^{n-1})\subset\R^{n-1}$. 
By slicing theorem (using $(n-1)$-times), $T_{\ell,y}:=T_\ell\llcorner \mathscr{B}_{y}^{\ell}$ is an integral 1-current with boundary in $\p \mathscr{B}_{y}^{\ell}$ for $\mathcal{H}^{n-1}$-a.e. $y\in B_{\sqrt{2\ell^2-1}}(0^{n-1})$.

Let us prove it by contradiction.
From Fubini's theorem, $|T_{\ell,y}|=V_\ell\llcorner \mathscr{B}^\ell_y$ converges as $\ell\to\infty$ to $\g_*\times\{y\}$ in the sense of Radon measure for $\mathcal{H}^{n-1}$-a.e. $y$.
We assume that $\g_*$ is not a union of lines.
Given $k>1$, there are a sequence of Borel sets $\Om_{\ell,k}\subset B_k(0^{n-1})$ with $\mathcal{H}^{n-1}(B_{k}(0^{n-1})\setminus\Om_{\ell,k})\to0$ as $\ell\to\infty$, and a sequence of connected components $T^*_{\ell,y}$ of $T_{\ell,y}$
 so that 
\begin{equation}\aligned\label{thxx'deG}
\sup_{x,x'\in L^*_{\ell,y}}\left|\th_\ell(x)-\th_\ell(x')\right|\ge\de_{\g_*}\qquad\mathrm{for\ each}\ y\in \Om_{\ell,k},
\endaligned
\end{equation}
where $L^*_{\ell,y}$ is the support of $T^*_{\ell,y}$, and $\de_{\g_*}$ is a positive constant depending only on $\g_*$.

From Theorem \ref{Low-V-Brp} and the argument of $U_k,U_k'$ just before \eqref{naLthk}, for each $\ell$ there is a finite cover $\{U_{\ell,k}\}_{1\le k\le N_{\ell}}$ of $\overline{B}_{3\ell/2}(p_\ell)$ so that for every $k$, $T_\ell\llcorner U_{\ell,k}$ has finite indecomposable components $T_{\ell,k,1},T_{\ell,k,2},\cdots,T_{\ell,k,m_{\ell,k}}$ satisfying that
the phase $\th_{\ell}\in W^{1,2}_{\mu_V}(L_{\ell,k,j})$ is harmonic on each $L_{\ell,k,j}:=\text{spt}\, T_{\ell,k,j}$ w.r.t. $\mu_{V_\ell}$, where $m_{\ell,k}$ is bounded by a constant $m_*$ independent of $\ell,k$. 
From Besicovitch covering lemma, there is an integer $n_*>1$ depending only on $n$ so that
\begin{equation}\aligned
\sharp\{j\in\{1,\cdots,N_\ell\}|\ x\in U_{\ell,k}\}\le n_*\qquad \mathrm{for\ any}\ x\in \overline{B}_{3\ell/2}(p_\ell).
\endaligned
\end{equation}

For each $\ell,k,j$, let $\tilde{\th}_{\ell,k,j}$ be a Lipschitz function on $L_\ell$ so that 
\begin{equation}\aligned\label{appthell|}
\sup_{L_{\ell,k,j}}\left|\th_\ell-\tilde{\th}_{\ell,k,j}\right|^2+\int_{L_{\ell,k,j}}\left|\na^{L_\ell}\th_\ell-\na^{L_\ell}\tilde{\th}_{\ell,k,j}\right|^2d\mu_{V_\ell}<\left(\f1{\ell N_\ell m_*}\right)^2.
\endaligned
\end{equation}
Combining Newton-Leibniz formula for $\tilde{\th}_{\ell,k,j}$, it follows that
\begin{equation}\aligned
\sup_{x,x'\in L^*_{\ell,y}}\left|\th_\ell(x)-\th_\ell(x')\right|\le\sum_{k,j}\int_{L^*_{\ell,y}\cap L_{\ell,k,j}}\left|\na^{L_\ell}\tilde{\th}_{\ell,k,j}\right|d\mathcal{H}^1+\f2{\ell}.
\endaligned
\end{equation}
Then combining \eqref{thxx'deG}, for each $y\in \Om_{\ell,k}$ we get
\begin{equation}\aligned\label{tthBly}
\de_{\g_*}-\f2\ell\le\sup_{x,x'\in L^*_{\ell,y}}\left|\th_\ell(x)-\th_\ell(x')\right|-\f2\ell\le\sum_{k,j}\int_{\mathscr{B}^\ell_y\cap L_{\ell,k,j}}\left|\na^{L_\ell}\tilde{\th}_{\ell,k,j}\right|d\mu_{V_\ell}.
\endaligned
\end{equation}
Set $M_{\ell,r}:=\Phi_\ell(B_1(0^{n+1})\times B_r(0^{n-1}))\subset M_\ell$.
Integrating \eqref{tthBly} over $\Om_{\ell,k}$ and using the co-area formula yield
\begin{equation}\aligned
\left(\de_{\g_*}-\f2\ell\right)\mathcal{H}^{n-1}(\Om_{\ell,r})\le&\sum_{k,j}\int_{M_{\ell,r}\cap L_{\ell,k,j}}\left|\na^{L_\ell}\tilde{\th}_{\ell,k,j}\right|d\mu_{V_\ell}\\
\le&\left(\mu_{V_\ell}(M_{\ell,r})\right)^{1/2}\left(\sum_{k,j}\int_{M_{\ell,r}\cap L_{\ell,k,j}}\left|\na^{L_\ell}\tilde{\th}_{\ell,k,j}\right|^2d\mu_{V_\ell}\right)^{1/2}.
\endaligned
\end{equation}
Combining Lemma \ref{caccith} and \eqref{appthell|} with Minkowski inequality, for large $\ell$ we have
\begin{equation}\aligned\label{kkkkk}
\f{\de_{\g_*}\omega_{n-1}}2r^{n-1}\le&\left(\mu_{V_\ell}(M_{\ell,r})\right)^{1/2}\left(\sum_{k,j}\int_{B_{2r}(p_\ell)\cap L_{\ell,k,j}}\left|\na^{L_\ell}\tilde{\th}_{\ell,k,j}\right|^2d\mu_{V_\ell}\right)^{1/2}\\
\le&\left(\mu_{V_\ell}(M_{\ell,r})\right)^{1/2}\left(\left(\sum_{k,j}\int_{B_{2r}(p_\ell)\cap L_{\ell,k,j}}\left|\na^{L_\ell}\th_\ell\right|^2d\mu_{V_\ell}\right)^{1/2}+\f1\ell\right)\\
\le&\left(\mu_{V_\ell}(M_{\ell,r})\right)^{1/2}\left(m_*n_*\left(\int_{B_{2r}(p_\ell)}\left|\na^{L_\ell}\th_\ell\right|^2d\mu_{V_\ell}\right)^{1/2}+\f1\ell\right)\\
\le&\left(\mu_{V_\ell}(M_{\ell,r})\right)^{1/2}\left(cr^{\f{n-2}2}+\f1\ell\right),
\endaligned
\end{equation}
where $c$ is a general constant independent of $\ell,r$.
From Theorem \ref{Ahlfors}, after a suitable covering argument, we get
$$\mu_{V_\ell}(M_{\ell,r})\le cr^{n-1}.$$
However, the above inequality \eqref{kkkkk} fails for the sufficiently large $r>0$ as $\ell\to\infty$.
Hence, we deduce that $\g_*$ must be a union of lines. This completes the proof.
\end{proof}

Let $(M^n,g,J,\omega)$ be a Ricci-flat K\"ahler manifold with a holomorphic $n$-form $\Om$ on $U$ satisfying $(-1)^{n^2/2}\Om\wedge\overline{\Om}=2^{n}\omega^n/n!$. 
\begin{lemma}\label{constantf}
Let $T=(L,\vth,\vec{T})$ be an integral Lagrangian current in $U$ so that its phase $\th$ has bounded variation with $J\na^L\th$ being the mean curvature $H$ of $|T|$.
If $f$ is a bounded integer-valued function on $L$ of zero bounded variation w.r.t. $\mu_{T}$, then  $f$ is a constant on each indecomposable component of $T$.
\end{lemma}
\begin{proof}
Let $L_*=\{x\in L:\,f(x)=\mathrm{ess\,sup}_{L}f\}$, and $T_*=(L_*,\vth,\vec{T})$.
There exists a constant $\ep_*>0$ so that $f(x)<\mathrm{ess\,sup}_{L}f-\ep_*$ for $\mu_{T}$-a.e. $x\in L\setminus L_*$. There is a $C^1$-function $\phi$ on $\R$ so that $\phi\equiv1$ on $[\mathrm{ess\,sup}_{L_\infty}f,\infty)$ and $\phi\equiv0$ on $(-\infty,\mathrm{ess\,sup}_{L}f-\ep_*)$.
Put $\chi:=\phi\circ f$. Then $\na^{L}\chi=0$ from the assumption $\na^{L}f=0$. 

Given an $(n-1)$-form $\e\in\mathcal{D}^{n-1}(U)$, we define a vector field $X\in\mathfrak{X}_c(U)$ by the dual of 1-form $\lan\e,\Om\ran$. Namely, $\e=\Om(X)=i_X\Om$. 
From \eqref{OmLthae}, we have
\begin{equation}\aligned\label{esqrt-1thdivLX}
e^{-\sqrt{-1}\th}\mathrm{div}_{L}X=e^{-\sqrt{-1}\th}\lan d_{L}\e,\Om\ran=\lan d_{L}\e,\mathrm{vol}_{L}\ran=\lan d\e,\vec{T}\ran
\endaligned\qquad\mu_{T}-a.e..
\end{equation}
From $\na^{L}\chi=0$, we have
$$\na^{L}(\chi e^{-\sqrt{-1}\th})=-\sqrt{-1}\chi e^{-\sqrt{-1}\th}\na^{L}\th=\sqrt{-1}\chi e^{-\sqrt{-1}\th}JH,$$ then
\begin{equation}\aligned
\int_{U}\left(\lan H,X\ran+\mathrm{div}_{L}X\right)\chi e^{-\sqrt{-1}\th} d\mu_{T}=-\sqrt{-1}\int_U\chi e^{-\sqrt{-1}\th}\lan JH,X\ran d\mu_{T}.
\endaligned
\end{equation}
Combining \eqref{esqrt-1thdivLX}, we obtain
\begin{equation}\aligned\label{intdeTinfty}
\int_{U}\lan d\e,\vec{T}\ran d\mu_{T_*}=-\int_U e^{-\sqrt{-1}\th}\lan H+\sqrt{-1}JH,X\ran d\mu_{T_*}.
\endaligned
\end{equation}

We consider a point $p\in L_*$ so that the tangent space $T_pL$ is an $n$-plane. Then we can choose an orthonormal basis $\{e_i\}_{1\le i\le n}$ of $T_pL$ so that $\vec{T}\big|_p=e_1\wedge\cdots\wedge e_n$.
There are a small $\ep_0>0$, and a complex coordinate chart $(B_{\ep_0}(p),Z)$ so that $B_{\ep_0}(p)\subset U$, $Z=(z_1,\cdots,z_n): B_{\ep_0}(p)\to\Phi(B_{\ep_0}(p))\subset\C^n$ is bi-holomorphic with $z_i=(x_i,y_i)$ and $\f{\p}{\p z_i}=\f{\p}{\p x_i}+\sqrt{-1}\f{\p}{\p y_i}$. 
Moreover, we can assume $\f{\p}{\p x_i}=e_i$ at $p$ since $T_pL$ is Lagrangian in $\C^n$.
From the definition of the phase $\th$, 
$$\Om=e^{\sqrt{-1}\th(p)} dz_1\wedge\cdots\wedge dz_n\qquad \text{at}\ p.$$
From the definition of $X$, there are constants $\e_j$ so that
$$X=\sum_j\e_j\f{\p}{\p z_j}=\sum_j\e_j(e_j+\sqrt{-1}Je_j)\qquad \text{at } p.$$ 
We write $H=\sum_iH_iJe_i$ at $p$ for constants $H_i$. Then at $p$ we have
\begin{equation}\aligned
\lan H+\sqrt{-1}JH,X\ran=\sum_{i,j}H_i\e_j\lan Je_i-\sqrt{-1}e_i,e_j+\sqrt{-1}Je_j\ran=0.
\endaligned
\end{equation}
This completes the proof from \eqref{intdeTinfty}.
\end{proof}

We assume that $U$ is isometrically embedded in $\R^m$ with 
$\p U\cap\mathbf{B}_1=\emptyset$.
Let $T_i=(L_i,\vth_i,\vec{T}_i)$ be a sequence of integral Lagrangian currents in $U$ with single-valued harmonic phase $\th_i$ and $\p T_i\llcorner U=0$. 
We assume that there are positive constants $\La,\tilde{\La}$ so that
$$\sup_i\left(\mathbf{M}(T_i)+\mathbf{M}(\p T_i)\right)\le\La,\qquad\sup_i\sup_{L_i}|\th_i|\le \tilde{\La}.$$
Let $V_i=|T_i|$ for each $i$, and $V_\infty=\vth_\infty|L_\infty|$ be the limit of $V_i$ defined at the beginning of \S4.
Up to a subsequence, we assume that $\th_i$ converges to a function $\th_\infty$ on $L_\infty$. Namely,
\begin{equation}\aligned\label{fthitofthinfty}
\lim_{i\to\infty}\int_U f \th_i d\mu_{V_i}=\int_U f\th_\infty d\mu_{V_\infty}\qquad\mathrm{for\ each}\ f\in C_c(U).
\endaligned
\end{equation}
Analog to Theorem \ref{zinftyJxiLVinfty}, we have
\begin{theorem}\label{JnaLinftythinfty}
$\th_\infty\in BV_{\mu_{V_\infty}}(U)$ and
\begin{equation}\aligned
J\na^{L_\infty}\th_\infty=H_{V_\infty}\qquad \mu_{V_\infty}-a.e..
\endaligned
\end{equation}
\end{theorem}
\begin{proof}
Given an open set $U_*\subset\subset U$, we define an integer $\Th_*=\sup_{x\in U_*}\Th_{x}(\mu_V)$.
Let $\ell(\Th_*)=\mathrm{lcm}(1,\cdots,\Th_*)$ denote the least common multiple of 1 to $\Th_*$.
We claim 
\begin{equation}\aligned\label{lTh*thi}
\lim_{i\to\infty}\int_U f e^{2\ell(\Th_*)\sqrt{-1}\th_i} d\mu_{V_i}=\int_U fe^{2\ell(\Th_*)\sqrt{-1}\th_\infty} d\mu_{V_\infty}\qquad\mathrm{for\ each}\ f\in C_c(U).
\endaligned
\end{equation}
If \eqref{lTh*thi} is true, then we can follow the proof of Theorem \ref{zinftyJxiLVinfty}, and deduce $\th_\infty\in BV_{\mu_{V_\infty}}(U)$ and $J\na^{L_\infty}\th_\infty=H_{V_\infty}$ $\mu_{V_\infty}$-a.e..

Now let us prove the claim \eqref{lTh*thi}.
For a small $\ep>0$, let $\{\mathcal{C}_k\}_k$, bi-Lipschitz maps $\{\phi_k\}_k$ and $\{\mathcal{C}^{i}_k\}_k$ be defined as in the proof of Lemma \ref{zinftyViVinftymass}. From \eqref{fthitofthinfty}, it follows that
\begin{equation}\aligned
\lim_{i\to\infty}\left|\int_{\mathcal{C}^i_k}\th_i d\mu_{V_i}-\int_{\mathcal{C}_k}\th_\infty d\mu_{V_\infty}\right|=0.
\endaligned
\end{equation}
Up to a refinement of $\mathcal{C}_k$ for each $k$, there is a point $x_k\in\mathcal{C}_k$ so that
\begin{equation}\aligned\label{C_kthinftyxk}
\sup_{\mathcal{C}_{k}}|\th_\infty-\th_\infty(x_k)|<\ep.
\endaligned
\end{equation}
Moreover, for each $i,k$ there is a point $x^i_k\in\mathcal{C}^i_k$ so that
\begin{equation}\aligned\label{C_kthinftyxki}
\sup_{\mathcal{C}^i_k}\left|e^{2\sqrt{-1}\th_i}-e^{2\sqrt{-1}\th_i(x^i_k)}\right|d\mu_{V_i}<\ep.
\endaligned
\end{equation}
Combining the above 3 inequalities, there are an integer $n_k\in\{1,\cdots,\Th_*\}$ and $n_{k,i}\in\N$ so that
\begin{equation}\aligned\label{nkthinkithinfty}
\left|n_k\th_i(x^i_k)-n_{k,i}\pi-n_k\th_\infty(x_k)\right|<\psi(\ep,i^{-1}).
\endaligned
\end{equation}
Here, $\psi(\ep,i^{-1})$ denotes a general positive function of $\ep,i$ satisfying $\lim_{\ep,i^{-1}\to0}\psi(\ep,i^{-1})=0$.
From \eqref{nkthinkithinfty}, it follows that
\begin{equation}\aligned
\left|e^{2\ell(\Th_*)\sqrt{-1}\th_i(x^i_k)}-e^{2\ell(\Th_*)\sqrt{-1}\th_\infty(x_k)}\right|<\psi(\ep,i^{-1}).
\endaligned
\end{equation}
Combining the definitions of $\mathcal{C}^i_k,\mathcal{C}_k$ and \eqref{C_kthinftyxk}\eqref{C_kthinftyxki}, for each $f\in C_c(U)$
\begin{equation}\aligned
\left|\int_{\mathcal{C}^i_k}f e^{2\ell(\Th_*)\sqrt{-1}\th_i} d\mu_{V_i}-\int_{\mathcal{C}_k} fe^{2\ell(\Th_*)\sqrt{-1}\th_\infty} d\mu_{V_\infty}\right|<\psi(\ep,i^{-1})\mu_{V_\infty}(\mathcal{C}_k)\sup_U|f|.
\endaligned
\end{equation}
By the definitions of $\mathcal{C}_k$ and $\mathcal{C}^i_k$, we can prove the claim \eqref{lTh*thi} by summation the above inequality on $k$. This completes the proof.
\end{proof}
Given a Lipschitz $\phi\in \mathbf{Lip}_c(U)$, we have
\begin{equation}\aligned
0=\int_{U} \lan \na^{L_i}\th_i,\na\phi\ran d\mu_{V_i}=\int_{U} \lan H_i,J\na\phi\ran d\mu_{V_i}.
\endaligned
\end{equation}
Taking the limit (up to choosing the subsequence) implies
\begin{equation}\aligned\label{JnaLinftyJnaphi}
0=\int_{U} \lan J\na^{L_\infty}\th_\infty,J\na\phi\ran d\mu_{V_\infty}=\int_{U} \lan \na^{L_\infty}\th_\infty,\na\phi\ran d\mu_{V_\infty}.
\endaligned
\end{equation}

Given a constant $k\ge0$ and a function $\varphi\in \mathbf{Lip}_c(L_\infty)$ with $\varphi\ge0$. From Lemma \ref{duiduikvar}, we have
\begin{equation}\aligned\label{thinftykVinfty}
\int|\th_\infty|^{k}|\na^{L_\infty}\th_\infty|^2\varphi d\mu_{V_\infty}\le
\int\lan |\th_\infty|^{k}\th_\infty\na^{L_\infty}\th_\infty,\na^{L_\infty}\varphi\ran d\mu_{V_\infty}.
\endaligned
\end{equation}
Combining \eqref{BrxnaVthell}, for $\mathbf{B}_s(p)\subset\mathbf{B}_1$ we get
\begin{equation}\aligned\label{naLinfty|thinfty|}
&\int_{\mathbf{B}_s(p)}\lan \r\na^{L_\infty}\r,\na^{L_\infty} |\th_\infty|^{k+2}\ran d\mu_{V_\infty}\\
\ge&\f{(k+1)(k+2)}2\int_{\mathbf{B}_s(p)}(s^2-\r^2)|\th_\infty|^{k}|\na^{L_\infty}\th_\infty|^2 d\mu_{V_\infty}.
\endaligned
\end{equation}
\begin{remark}
Using \eqref{naLinfty|thinfty|}, we can deduce the monotonicity formula $\mathbf{m}_\mathbf{p}(\mu_{V_\infty},r)$, i.e., \eqref{mon-bfTh} holds with $V,\th,H_V$ replaced by $V_\infty,\th_\infty,\xi$, respectively. Moreover, we can also deduce the Sobolev inequality for functions in the form of $\th^\ell \varphi$ with all $\ell\ge1$ and $\varphi\in C^1_c(\mathbf{B}_{1})$.
\end{remark}

Inspired by Cheeger-Colding \cite{CCo1},
we define some approximate sets related to the singular sets of HSL varifolds as follows.
Let $\mathcal{S}_{L_\infty,\ep,r}$ denote the subsets in $L_\infty$ containing all the points $x$ satisfying
$$d_{GH}\left(L_\infty\cap B_{r}(x),\cup_j P_j\cap B_{r}(0^{2n})\right)\ge\ep r^n$$
for any finite collection of Lagrangian planes $P_j$ through origin in $\C^n.$ Here, $B_{r}(0^{2n})$ denotes the ball in $\C^n$ centered at the origin with radius $r$.
Let
$$\mathcal{S}_{L_\infty,\ep}=\bigcap_{0<r\le1}\mathcal{S}_{L_\infty,\ep,r},
\qquad\mathcal{S}_{L_\infty}=\bigcup_{\ep>0}\mathcal{S}_{L_\infty,\ep}=\bigcup_{\ep>0}\bigcap_{0<r\le1}\mathcal{S}_{L_\infty,\ep,r}.$$

\begin{corollary}\label{GH-sing}
$\mathcal{S}_{L_\infty}$ has Hausdorff dimension $\le n-2$. 
\end{corollary}
\begin{proof}
We only need to show that $\mathcal{S}_{L_\infty,\ep}$ has Hausdorff dimension $\le n-2$ for each $\ep>0$.
From the monotonicity formula $\mathbf{m}_\mathbf{p}(\mu_{V_\infty},r)$, it follows that every tangent cone of $V_\infty$ is a minimal Lagrangian cone in $\C^n$. Then one can adopt a standard argument (see \S 11 in \cite{Gi}, or Lemma 4.4 in \cite{D1} for instance) by combining Federer's dimension reduction argument and Lemma \ref{Rig-Gcodim1}. So we omit the proof.
\end{proof}

\begin{lemma}\label{thinfty-vari-harm}
We suppose that there is an orientation $\vec{T}_\infty$ on $L_\infty$ so that $T_\infty=(L_\infty,\vth_\infty,\vec{T}_\infty)$ satisfies $\p T_\infty\llcorner U=0$. Given $\mathbf{B}_r(p)\subset\mathbf{B}_1$,
if there are constants $0<\de<\tau<r$ so that 
\begin{equation}\aligned\label{thinftytBrpde*}
&\mathbf{M}(\p(T_\infty\llcorner\{\th_\infty>t\})\llcorner \mathbf{B}_r(p))\\
\ge&\de\min\{\mathbf{M}(T_\infty\llcorner(\{\th_\infty<t\}\cap \mathbf{B}_\tau(p))),\mathbf{M}(T_\infty\llcorner(\{\th_\infty>t\}\cap \mathbf{B}_\tau(p)))\}
\endaligned
\end{equation}
for almost every $t\in\R$,
then for any monotonic function $f\in C^1(\R)$ with $|f'|>0$ on $[-\tilde{\La},\tilde{\La}]$ there holds
\begin{equation}\aligned
\int_{\mathbf{B}_\tau(p)}|f\circ\th_\infty-\overline{f\circ\th_\infty}|d\mu_{V_\infty}\le \f{c}{\de}\sup_{k\in\Z}\liminf_{i\to\infty}\int_{\mathbf{B}_r(p)}|\na^{L_i}(f\circ\varphi_{i,k}\circ\th_i)|d\mu_{V_i},
\endaligned
\end{equation}
where $\overline{f\circ\th_\infty}:=\fint_{\mathbf{B}_\tau(p)}f\circ\th_\infty d\mu_{V_\infty}$, $\varphi_{i,k}$ is a Lipschitz function on $\R$ satisfying 
\begin{equation}
\varphi_{i,k}(t)=\left\{\begin{split}
&\sup_{L_i}\th_i\qquad &&\text{if}\  t\ge\sup_{L_i}\th_i+2k\pi/\ell^*\\
&t-2k\pi/\ell^* \qquad &&\text{if}\  \inf_{L_i}\th_i+2k\pi/\ell^*<t<\sup_{L_i}\th_i+2k\pi/\ell^*\\
&\inf_{L_i}\th_i \qquad &&\text{if}\  t\le\inf_{L_i}\th_i+2k\pi/\ell^*\\
\end{split}\right.
\end{equation}with $\ell^*:=2\ell(\Th_*)$,
and $c$ depends only on $n,U,\La,\tilde{\La}$. 
\end{lemma}
\begin{proof}
We may assume $f'>0$ and $\overline{f\circ\th_\infty}=0$ up to a linear transform to $f$. For all $t\in\R$ and $s\in(0,r]$, let 
$$U_{s,t}^+=\{x\in\mathbf{B}_s(p):\, f\circ\th_\infty(x)>t\},\qquad U_{s,t}^-=\{x\in\mathbf{B}_s(p):\, f\circ\th_\infty(x)<t\}.$$
Without loss of generality, we assume $\mu_{V_\infty}(U_{r,0}^+)\le\mu_{V_\infty}(U_{r,0}^-)$. 
From \eqref{thinftytBrpde*} and co-area formula, it follows that
\begin{equation}\aligned\label{|fthinfty|}
 \int_{\mathbf{B}_\tau(p)}|f\circ\th_\infty|d\mu_{V_\infty}=&2\int_{U_{r,0}^+}f\circ\th_\infty d\mu_{V_\infty}=2\int_0^\infty\mu_{V_\infty}(U_{r,t}^+)dt\\
\le&\f2\de\int_0^\infty\mathbf{M}(\p(T_\infty\llcorner\{\th_\infty>f^{-1}(t)\})\llcorner \mathbf{B}_r(p)) dt.
\endaligned
\end{equation}
Noting $\sup_i\sup_{L_i}|\th_i|\le \tilde{\La}$. From co-area formula, we have
\begin{equation}\aligned\label{tktk+1pTiell}
&\sum_{k\in\Z}\int_0^\infty\mathbf{M}(\p(T_i\llcorner\{\th_i>f^{-1}(t)+2k\pi/\ell^*\})\llcorner \mathbf{B}_r(p)) dt\\
=&\sum_{k\in\Z}\int_{f^{-1}(0)+2k\pi/\ell^*}^{f^{-1}(\infty)+2k\pi/\ell^*}\mathbf{M}(\p(T_i\llcorner\{\th_i>\tau\})\llcorner \mathbf{B}_r(p)) f'\left(\tau-\f{2k\pi}{\ell^*}\right) d\tau \\
\le&c_*\sup_{k\in\Z}\int_{\mathbf{B}_r(p)}|\na^{L_i}(f\circ\varphi_{i,k}\circ\th_i)|d\mu_{V_i},
\endaligned
\end{equation}
where $c_*$ is a constant depending only on $n,U,\La,\tilde{\La}$.

From the convergence of Radon measure on Grassmann bundle, we have 
$$\sum_{k\in\Z}T_i\llcorner\{\th_i> f^{-1}(t)+2k\pi/\ell^*\}\rightharpoonup \sum_{k\in\Z}T_\infty\llcorner\{\th_\infty> f^{-1}(t)+2k\pi/\ell^*\}$$ 
on $\mathbf{B}_r(p)$ for almost all $t\in\R$. This implies 
\begin{equation}\aligned
&\sum_{k\in\Z}\mathbf{M}\big(\p (T_\infty\llcorner\{\th_\infty> f^{-1}(t)+2k\pi/\ell^*\})\llcorner \mathbf{B}_r(p)\big)\\
\le&\liminf_{i\to\infty}\sum_{k\in\Z}\mathbf{M}( \p(T_i\llcorner\{\th_i> f^{-1}(t)+2k\pi/\ell^*\})\llcorner \mathbf{B}_r(p)).
\endaligned
\end{equation}
Combining Fatou lemma, we have
\begin{equation}\aligned\label{tktk+1MpTinftyl}
&\int_0^\infty\sum_{k\in\Z}\mathbf{M}\big(\p (T_\infty\llcorner\{\th_\infty> f^{-1}(t)+2k\pi/\ell^*\})\llcorner \mathbf{B}_r(p)\big) dt\\
\le&\int_0^\infty\liminf_{i\to\infty}\sum_{k\in\Z}\mathbf{M}( \p(T_i\llcorner\{\th_i> f^{-1}(t)+2k\pi/\ell^*\})\llcorner \mathbf{B}_r(p)) dt\\
\le&\liminf_{i\to\infty}\int_0^\infty\sum_{k\in\Z}\mathbf{M}( \p(T_i\llcorner\{\th_i> f^{-1}(t)+2k\pi/\ell^*\})\llcorner \mathbf{B}_r(p)) dt.
\endaligned
\end{equation}
Combining \eqref{|fthinfty|}\eqref{tktk+1pTiell}\eqref{tktk+1MpTinftyl}, we complete the proof.
\end{proof}

\begin{theorem}\label{current-n}
There exists an orientation $\vec{T}_\infty$ on $L_\infty$  so that $T_\infty=(L_\infty,\vth_\infty,\vec{T}_\infty)$ is a Hamitonian stationary Lagrangian current in $U$ with bounded single-valued phase $\th^*$ of bounded variation w.r.t. $\mu_{V_\infty}$, where $J\na^{L_\infty}\th^*$ is the mean curvature of $|T_\infty|$, and $\th^*-\th_\infty$ is a constant on each indecomposable component of $T_\infty$.
\end{theorem}
\begin{proof}
Given a small $\ep>0$, from Corollary \ref{GH-sing} there exists a sequence of balls $B_{r_{\ep,j}}(x_{\ep,j})$ in $M$ with $\mathcal{S}_{L_\infty}\subset\cup_jB_{s_{\ep,j}}(x_{\ep,j})$ and all $s_{\ep,j}<\ep$ so that
\begin{equation}\aligned\label{rijn-2j0}
\sum_{j}s_{\ep,j}^{n-3/2}<\ep.
\endaligned
\end{equation}
Given an open $U'\subset\subset U$, let $W_{\ep}:=\cup_jB_{2s_{\ep,j}}(x_{\ep,j})$ be an open subset and $K_{\ep}=\overline{U'}\setminus W_{\ep}$ be a compact subset in $U$.

For each point $p\in K_\ep\cap L_\infty$, there is a sequence $r_i\to\infty$ so that $\lim_{i\to\infty}(\tau_{p,r_i^{-1}})_\# V_\infty$ converges to a stationary varifold $V_{\infty,p}$ in $\R^{m}$ in the sense of Radon measure. Here, $\tau_{p,r_i^{-1}}$ is defined just below \eqref{EmuVprRn}. 
Moreover, $V_{\infty,p}$ has the form of $\sum_{1\le j\le m_p}|\g_j\times\R^{n-1}|$ from Lemma \ref{Rig-Gcodim1}, where each $\g_j$ is a straight line through the origin, and $m_p$ has a uniform upper bound by the density. There is a subsequence $\{k_i\}$ so that for each $\{k^*_i\}$ with $k^*_i\ge k_i$, $\lim_{i\to\infty}(\tau_{p,r_i^{-1}})_\# V_{k^*_i}$ converges to $V_{\infty,p}$ in the sense of Radon measure.

Hence, there are two integral currents $T_{i,k}^+,T_{i,k}^-$ ($T_{i,k}^+\neq0$) in $B_{r_i}(p)$ with $\text{spt}T_{i,k}^\pm\subset \text{spt}T_{k}$ for each $k\ge k_i$ so that
$(\tau_{p,r_i^{-1}})_\#T_{i,k}^\pm$ converges as $i\to\infty$ to integral currents $T_{\infty,p}^\pm$, respectively, $T_{\infty,p}^+$ attains the maximum mass satisfying spt$T_{\infty,p}^+=\text{spt}(V_{\infty,p}\llcorner \mathbf{B}_{1}(0^m))$, and $|T_{\infty,p}^+|+|T_{\infty,p}^-|=V_{\infty,p}\llcorner \mathbf{B}_{1}(0^m)$.
Here, $T_{\infty,p}^+$ has the form $\sum_j\vth_{\g_j}[|\g_j\times\R^{n-1}|]\llcorner \mathbf{B}_{1}(0^m)$ for different straight lines $\g_j$ through the origin with integers $\vth_{\g_j}$. We use $\vec{\g}_j$ denoting the orientation of $[|\g_j\times\R^{n-1}|]$.
Up to modifying $T^\pm_{i,k}$, we can assume that there are components $T^\pm_{i,k,j}$ of $T^\pm_{i,k}$ so that
\begin{equation}\aligned
\lim_{k\ge k_i,i\to\infty}\sup_{1\le j\le m_p}\sup_{\text{spt}T_{i,k,j}^\pm}\{|\vec{T}_{k}\mp\vec{\g}_j|\}=0.
\endaligned
\end{equation}
From Federer-Fleming compactness theorem, for all $i,j$ and $k\ge k_i$ there exist minimizing currents $S_{i,k,j}^\pm$ in $B_{r_i}(p)$ so that 
$\lim_{k\ge k_i,i\to\infty}\sup_{1\le j\le m_p}r_i^{-n}\mathbf{M}(S_{i,k,j}^\pm)=0$, $\p(S_{i,k,j}^\pm-\hat{T}_{i,k,j}^\pm)\llcorner B_{r_i}(p)=0$, 
and $\lim_{k\ge k_i,i\to\infty}\sup_{1\le j\le m_p}r_i^{-1}\sup_{\text{spt}S_{i,k,j}^\pm}|Q_j(x_\a)|=0$ for some orthogonal transformation $Q_j$ on $\R^{m}$ for each $j$.

Up to choosing a subsequence, there are a constant $r_{\ep,p}>0$, an integer $i_{\ep,p}>1$, a finite collection of integral currents $T_{i,\ep,p,j}^+\neq0,T_{i,\ep,p,j}^-$ in $B_{r_{\ep,p}}(p)$ with $\text{spt}T_{i,\ep,p,j}^\pm\subset \text{spt}T_{i}$, and two integral currents $S_{i,\ep,p,j}^\pm$ for all integers $i\ge i_{\ep,p}$ and $j\in\{1,\cdots,m_p\}$
so that
\begin{equation}\aligned\label{sptTiepp+}
d_{\mathbf{B}_{2}(0^m)}\left((\tau_{p,r_{\ep,p}^{-1}})_\#\hat{T}_{i,\ep,p,j}^\pm,T_{\infty,p}^\pm\right)<\ep,\qquad
\max_{1\le j\le j_*}\sup_{\text{spt}T_{i,\ep,p,j}^\pm}\{|\vec{T}_i\mp\vec{\g}_j|\}<\ep,
\endaligned
\end{equation}
and
\begin{equation}\aligned\label{sptSiepp+}
\mathbf{M}(S_{i,\ep,p,j}^\pm)<\ep r_{\ep,p}^n,\ \p \hat{T}_{i,\ep,p,j}^\pm\llcorner B_{r_{\ep,p}}(p)=0,\ \sup_{\text{spt}S_{i,\ep,p,j}^\pm}|Q_j(x_\a)|<\ep r_{\ep,p},
\endaligned
\end{equation}
where $\hat{T}_{i,\ep,p,j}^\pm=T_{i,\ep,p,j}^\pm\llcorner B_{r_{\ep,p}}(p)-S_{i,\ep,p,j}^\pm$, and $d_{\mathbf{B}_{2}(0^m)}$ denotes the pseudometric inducing the flat metric topology (see \S 31 in \cite{S} for details).
Put 
$$T_{i,\ep,p}^\pm=\sum_{1\le j\le m_p}T_{i,\ep,p,j}^\pm\qquad \text{and}\qquad\hat{T}_{i,\ep,p}^\pm=\sum_{1\le j\le m_p}\hat{T}_{i,\ep,p,j}^\pm.$$
We continue the above process for each point $p\in K_\ep\cap L_\infty$.
By finite covering lemma, there is a finite collection of points $\{p_j\}_{j=1}^{m_\ep}\subset K_\ep\cap L_\infty$ so that $K_\ep\cap L_\infty\subset\cup_{1\le j\le m_\ep}B_{r_{\ep,p_j}}(p_j)$. 

Given $p\in\{p_j\}_{j=1}^{m_\ep}$,
for a point $p'\in \{p_j\}_{j=1}^{m_\ep}$, if $B_{r_{\ep,p}}(p)\cap B_{r_{\ep,p'}}(p')\neq\emptyset$,
we can switch $T_{i,\ep,p'}^+$ and $T_{i,\ep,p'}^-$ so that
\begin{equation}\aligned
\mathbf{M}\left(T_{i,\ep,p'}^+\llcorner \text{spt}T_{i,\ep,p}^+\right)\ge\mathbf{M}\left(T_{i,\ep,p'}^-\llcorner \text{spt}T_{i,\ep,p}^+\right).
\endaligned
\end{equation}
From \eqref{sptTiepp+}, it follows that
\begin{equation}\aligned
\mathbf{M}\left((T_{i,\ep,p'}^+-T_{i,\ep,p}^+)\llcorner(B_{r_{\ep,p}}(p)\cap B_{r_{\ep,p'}}(p'))\right)<\psi(\ep)\min\{r_{\ep,p}^n,r_{\ep,p'}^n\},
\endaligned
\end{equation}
where $\psi(\ep)$ is a general function of $\ep$ satisfying $\lim_{\ep\to0}\psi(\ep)=0$.
Noting that the phase $\th_i$ of $T_i$ is single-valued, which determines the orientation of $T_i$.
Hence, we can continue the above switch operation so that if $B_{k,k'}:=B_{r_{\ep,p_k}}(p_k)\cap B_{r_{\ep,p_{k'}}}(p_{k'})\neq\emptyset$, then for all $i\ge\max_{1\le j\le m_\ep} i_{\ep,p_j}$ there holds
\begin{equation}\aligned\label{Tieppkk'+}
\mathbf{M}\left((T_{i,\ep,p_k}^+-T_{i,\ep,p_{k'}}^+)\llcorner B_{p_k,p_{k'}}\right)<\psi(\ep)\min\{r_{\ep,p_k}^n,r_{\ep,p_{k'}}^n\}.
\endaligned
\end{equation}

From Besicovitch covering lemma, there are an integer $m_*$ depending only on $n,U,\La,\tilde{\La}$,
and sub-collections $\Xi_1,\cdots,\Xi_{m_*}\subset\{1\le j\le m_\ep\}$ such that
each $\mathscr{B}_j:=\{B_{r_{\ep,p_k}}(p_k)\}_{ k\in\Xi_j}$ is a pairwise disjoint (or empty) collection, and $\cup_{1\le j\le m_*}\mathscr{B}_j$ still covers $K_\ep\cap L_\infty$.
Denote $\mathcal{B}_j=\cup_{B\in\mathscr{B}_j}B$ and $\widetilde{T}_{i,\ep,p_j}^\pm=\sum_{k\in\Xi_j}\hat{T}_{i,\ep,p_k}^\pm\llcorner B_{r_{\ep,p_k}}(p_k)$.

Let us define $\mathbf{T}_{i,\ep,p_j}^\pm$ from $j=1$ to $j=m_*$ by induction. Let $\mathbf{T}_{i,\ep,p_1}^\pm=\widetilde{T}_{i,\ep,p_1}^\pm$ and $W_1=\emptyset$. If we have defined $\widetilde{T}_{i,\ep,p_k}^\pm$ and $W_k$ for all $1\le k\le j\in\{1,\cdots,m_*-1\}$, then we further define $W_{j+1}=\mathcal{B}_{j+1}\bigcap\cup_{1\le i\le j}\mathcal{B}_{i}$ and
\begin{equation}\aligned
\check{T}_{i,\ep,p_{j+1}}^\pm=\mathbf{T}_{i,\ep,p_j}^\pm+\widetilde{T}_{i,\ep,p_{j+1}}^\pm-\mathbf{T}_{i,\ep,p_j}^\pm\llcorner(W_{j+1}\setminus\text{spt}\widetilde{T}_{i,\ep,p_{j+1}}^\pm)-\widetilde{T}_{i,\ep,p_{j+1}}^\pm\llcorner W_{j+1}.
\endaligned
\end{equation}
Clearly, 
$\p\check{T}_{i,\ep,p_{j+1}}^\pm\llcorner(\cup_{1\le i\le j+1}\mathcal{B}_{i}\setminus W_{j+1})=0.$
From \eqref{sptSiepp+}, there are integral currents $\tilde{S}_{i,\ep,p_{j+1}}^\pm$ so that
\begin{equation}\aligned\label{spttildeSiepp+}
\mathbf{M}(\tilde{S}_{i,\ep,p_{j+1}}^\pm)<\ep\sum_{1\le l\le m_*}\sum_{k\in\Xi_l} r_{\ep,p_k}^n,\quad \p(\check{T}_{i,\ep,p_{j+1}}^\pm-S_{i,\ep,p_{j+1}}^\pm)\llcorner (\cup_{1\le i\le j+1}\mathcal{B}_{i})=0.
\endaligned
\end{equation}
We set 
$$\mathbf{T}_{i,\ep,p_{j+1}}^\pm=\check{T}_{i,\ep,p_{j+1}}^\pm-\tilde{S}_{i,\ep,p_{j+1}}^\pm.$$
Let $L_{i,\ep}^\pm=\cup_{1\le l\le m_*}\cup_{j\in\Xi_l}\text{spt}T_{i,\ep,j}^\pm$, and $T_{i,\ep}^\pm=T_i\llcorner L_{i,\ep}^\pm$.
Combining Theorem \ref{Ahlfors} and \eqref{sptSiepp+}\eqref{spttildeSiepp+}, we have
\begin{equation}\aligned\label{diffbfT-Tiep}
\mathbf{M}\left(\mathbf{T}_{i,\ep}^\pm-T_{i,\ep}^\pm\right)\le& 
\sum_{2\le j\le m_*}\mathbf{M}(\tilde{S}_{i,\ep,p_{j}}^\pm)+\sum_{1\le l\le m_*}\sum_{j\in\Xi_l}\sum_{1\le k\le m_{p_j}}\mathbf{M}(S_{i,\ep,p_j,k}^\pm)\\
\le& \tilde{c}\ep\sum_{1\le l\le m_*}\sum_{k\in\Xi_l} r_{\ep,p_k}^n\le c\ep,
\endaligned
\end{equation}
where $c,\tilde{c}$ are general constants depending only on $n,U,\La,\tilde{\La}$, which may change from line to line.

Put $\mathbf{T}^+_{i,\ep}:=\mathbf{T}_{i,\ep,p_{m_*}}^+$ and $\mathbf{T}^-_{i,\ep}:=-\mathbf{T}_{i,\ep,p_{m_*}}^-$ for short. Then $\p \mathbf{T}_{i,\ep}^\pm\cap K_\ep=0$ from \eqref{spttildeSiepp+}. Combining \eqref{sptTiepp+} and \eqref{Tieppkk'+}, for each $j\in\{1,\cdots,m_*\}$ there exists $\e_j\in\mathcal{D}^{n}(\mathcal{B}_j)$ with $|\e_j|\le1$ so that
\begin{equation}\aligned\label{pmTiepej}
\pm T_{i,\ep}^\pm(\e_j)\ge(1-\psi(\ep))\sum_{k\in\Xi_j}\mathbf{M}\left(T_{i,\ep,p_k}^\pm\right)\qquad\qquad \text{for each large }i.
\endaligned
\end{equation}
There is a sequence of nonnegative smooth functions $\{\phi_{j,\ep}\}_{1\le j\le m_*}$ on $U$ with spt$\phi_{j,\ep}\subset \mathcal{B}_j$ for each $j$ so that $\sum_{1\le j\le m_*}\phi_{j,\ep}\equiv1$ on $K_\ep\cap L_\infty$. Put $\e_\ep:=\sum_{1\le j\le m_*}\phi_{j,\ep}\e_j$. Combining \eqref{Tieppkk'+}\eqref{diffbfT-Tiep}\eqref{pmTiepej}, for all large $i\ge\max_{1\le j\le m_\ep} i_{\ep,p_j}$ we have
\begin{equation}\aligned
&\mathbf{T}_{i,\ep}^+(\e_\ep)=\sum_{j=1}^{m_*}\mathbf{T}_{i,\ep}^+(\phi_{j,\ep}\e_j)\ge\sum_{j=1}^{m_*}\left(T_{i,\ep}^+(\e_j)-T_{i,\ep}^+((1-\phi_{j,\ep})\e_j)\right)-c\ep\\
\ge&\sum_{j=1}^{m_*}\left((1-\psi(\ep))\sum_{k\in\Xi_j}\mathbf{M}\left(T_{i,\ep,p_k}^\pm\right)-\sum_{k\in\Xi_j}\mathbf{M}\left(|T_{i,\ep,p_k}^+|\llcorner(1-\phi_{j,\ep})\right)-c\psi(\ep)\right)-c\ep\\
\ge&\sum_{j=1}^{m_*}\left(\sum_{k\in\Xi_j}\mathbf{M}(|T_{i,\ep,p_k}^+|\llcorner\phi_{j,\ep})-\psi(\ep)\sum_{k\in\Xi_j}\mathbf{M}(T_{i,\ep,p_k}^+)\right)-c\psi(\ep)\\
\ge&\mathbf{M}(|T_{i,\ep}^+|\llcorner K_\ep)-c\psi(\ep).
\endaligned
\end{equation}
Moreover, we have the similar estimate for $\mathbf{T}_{i,\ep}^-$. Put $\mathbf{T}_{i,\ep}=\mathbf{T}_{i,\ep}^+-\mathbf{T}_{i,\ep}^-$, and $T_{i,\ep}=T_{i,\ep}^+-T_{i,\ep}^-$. Then
\begin{equation}\aligned
\mathbf{T}_{i,\ep}(\e_\ep)=&\mathbf{T}_{i,\ep}^+(\e_\ep)-\mathbf{T}_{i,\ep}^-(\e_\ep)\ge\mathbf{M}(|T_{i,\ep}^+|\llcorner K_\ep)+\mathbf{M}(|T_{i,\ep}^-|\llcorner K_\ep)-c\psi(\ep)\\
=&\mathbf{M}(|T_{i,\ep}|\llcorner K_\ep)-c\psi(\ep).
\endaligned
\end{equation}

Noting $\p \mathbf{T}_{i,\ep}\llcorner K_\ep=0$ and \eqref{rijn-2j0}.
There is a subsequence $\{\ell_i\}$ so that
the integral Lagrangian current $T_i^*:=\mathbf{T}_{\ell_i,1/i}$
converges to an integral Lagrangian current $T_\infty^*=(L_\infty,\vth_\infty,\vec{T}_\infty^*)$ satisfying $|T_\infty^*|=V_\infty$ and $\p T_\infty^*\llcorner U=0$. 
We define a function $\th_i^*$ on $L_{\ell_i}$ by $\th_i^*=\th_{\ell_i}+\pi$ on $L_{\ell_i,1/i}^-$ and $\th_i^*=\th_{\ell_i}$ on $L_{\ell_i}\setminus L_{\ell_i,1/i}^-$. We assume that $e^{\sqrt{-1}\th_i^*}$ converges to a function $\z_\infty$ on $L_\infty$.
From the argument in Lemma \ref{zinftyViVinftymass}, $e^{2j\sqrt{-1}\th_i^*}\to \z_\infty^{2j}$
in the sense of measure for each $j\in\Z$. 
Combining Theorem \ref{zinftyJxiLVinfty}, we get $\z_\infty\in BV_{\mu_{V_\infty}}(U)$ and $\z_\infty=e^{\sqrt{-1}\th^*}$ with
$$J\na^{L_\infty}\th^*=J\na^{L_\infty}\th_\infty=H_{V_\infty}\qquad \mu_{V_\infty}-a.e.$$
for some multi-valued real measurable function $\th^*$. 
Hence, $T^*_\infty$ is Hamiltonian stationary Lagrangian from \eqref{JnaLinftyJnaphi}. 
Combining \eqref{lTh*thi}, we deduce
$$e^{2\ell(\Th_*)\sqrt{-1}(\th^*-\th_\infty)}=1\quad  \mu_{V_\infty}-a.e..$$
In other words, $\f{\ell(\Th_*)}{\pi}(\th^*-\th_\infty)$ is an integer-valued function in $BV_{\mu_{V_\infty}}(U)$.
From Lemma \ref{constantf}, $\th^*-\th_\infty$ is constant on each indecomposable component of $T_\infty$.
Let $\vec{T}_i^*$ denote the orientation of $T_i^*$, then
\begin{equation}\aligned
\int_U fe^{\sqrt{-1}\th^*} d\mu_{T_\infty^*}=&\lim_{i\to\infty}\int_U f\lan\Om,\vec{T}_i^*\ran d\mu_{T_i^*}=\lim_{i\to\infty}T_i^*(f\Om)\\
=&T^*_\infty(f\Om)=\int_U f\lan\Om,\vec{T}_\infty^*\ran d\mu_{T_\infty^*}.
\endaligned
\end{equation}
This means that $\th^*$ is a phase of $T_\infty^*$.
 We complete the proof.
\end{proof}
\begin{remark}\label{TitoTinftyW}
$\th_*-\th_\infty$ in Theorem \ref{current-n} may be not zero. For example, let $T_i=[|\R^1\times\{0\}|]-[|\R^1\times\{i^{-1}\}|]$ in $\R^2$, then for the orientation (vector) $(1,0)$, the phase $\th_i$ of $T_i$ satisfies $\th_i\equiv0$ on $\R^1\times\{0\}$, and $\th_i\equiv\pi$ on $\R^1\times\{i^{-1}\}$. Clearly, $T_i\rightharpoonup0$ in the current sense, and we can choose $T_\infty=2[|\R^1\times\{0\}|]$ with phase $\th_\infty\equiv0$. Then $|T_i|\to |T_\infty|$ and $\th_i\to\pi/2$ both locally in the sense of measure.
\end{remark}

\section{Regularity and rigidity of Lipschitz HSLs in Euclidean space} 

For a function $u\in C^1(B_1)\cap W^{2,n}(B_1)$,
let $L$ denote a graph $\{(x,Du(x)):x\in B_1\}$ in $B_1\times\R^n$. Then $[|L|]$ is a multiplicity one Lagrangian current, whose phase $\th$ satisfies
\begin{equation}\aligned\label{LE}
\th:=\mathrm{tr}(\arctan D^2u)=\sum_{i=1}^n\arctan\la_i\quad (\text{mod } 2\pi)\quad a.e.,
\endaligned
\end{equation}
where $\la_1,\cdots,\la_n$ denote the eigenvalues of $D^2u$ at $C^2$-points of $u$, and we have regarded $\th$ as a measurable function on $B_1$ by identifying $\th(x)=\th(x,Du(x))$ for almost every $x\in B_1$. Let $\na^L$ denote the Levi-Civita connection on $L$ a.e..

Now we further assume that the phase $\th\in W^{1,2}(B_1)$, and 
$\th$ is harmonic on $L$, i.e., 
\begin{equation}\aligned
\int_L\left\lan\na^L\th,\na^L\phi\right\ran d\mathcal{H}^n=0\qquad\text{for any}\ \phi\in C^\infty_c(B_1\times\R^n).
\endaligned
\end{equation}
In other words, $u$ is a weak solution to geometric Hamiltonian stationary equation \eqref{HSL}.
Let $\La\ge1$ be a constant satisfying
\begin{equation}\aligned\label{DEF-La}
\mathcal{H}^n(L\cap\mathbf{B}_{r}(\mathbf{x}))\le\La\qquad \text{for each}\ \mathbf{B}_r(\mathbf{x})\subset B_1\times\R^n.
\endaligned
\end{equation}
From Theorem \ref{Ahlfors}, there is a general constant $c_\La\ge1$ depending only on $n,\La$ so that
\begin{equation}\aligned
c_\La^{-1}\le s^{-n}\mathcal{H}^n(L\cap\mathbf{B}_s(\mathbf{x}))\le c_\La r^{-n}\mathcal{H}^n(L\cap\mathbf{B}_r(\mathbf{x}))
\endaligned
\end{equation}
for each $0<s<r$. For a point $\mathbf{x}=(x,Du(x))\in L$ so that $u$ is $C^2$ at $x$, the tangent cone of $L$ at $\mathbf{x}$ is a Lagrangian $n$-plane with multiplicity one. From Theorem \ref{Allard}, $L$ is smooth in a small neighborhood of $\mathbf{x}$, i.e., $u$ is $C^2$ in a small neighborhood of $x$.
Hence, the regular points are dense, relatively open in $L$.

Let $\tau_{\mathbf{x},t}=\f1t(\cdot-\mathbf{x})$ for all $\mathbf{x}\in\R^{2n}$ and $t>0$.
\begin{lemma}\label{smoothLu}
Assume $u\in C^{1,1}(B_1)$. If the harmonic phase $\th$ satisfies 
\begin{equation}\aligned\label{C-SC-phase}
(n-2)\pi/2\le\th<n\pi/2\qquad a.e.
\endaligned
\end{equation}
on $B_1$, then the Lagrangian graph $L$ is smooth.
\end{lemma}
\begin{proof}
We assume that the singular set of $L$ is non-empty, and pick a singular point $\mathbf{x}=(x,Du(x))\in L$.
From compactness of varifolds and $u\in C^{1,1}(B_1)$, there is a sequence $r_i\to0$ so that $(\tau_{\mathbf{x},r_i})_\# L$ converges to a Lagrangian stationary varifold $V_{\mathbf{x}}$ of multiplicity one in $\R^{2n}$, where the support $L_{\mathbf{x}}$ of $V_{\mathbf{x}}$ is an entire graph of the function $Dw$ for some function $w\in C^{1,1}(\R^n)$. 
From \eqref{Wpl*inf} and \eqref{Wpl*to0}, $L_{\mathbf{x}}$ is a cone with constant phase $\th_*$. From Theorem \ref{Allard} and \eqref{C-SC-phase}, it follows that $\f{n-2}2\pi\le\th_*<\f n2\pi$ a.e..
We claim that
\begin{center}
$L_{\mathbf{x}}$ is flat.
\end{center}
If the claim fails, then after Federer's dimension reduction argument, we get a non-flat, graphical, special Lagrangian cone $C$ in $\R^{2n}$ with 
phase $\th_C$ satisfying $\f{n-2}2\pi\le\th_C<\f n2\pi$, where $C=\{(x',l(x''):\, x'\in C', x''\in\R^{n-k}\}$, $C'$ is a non-flat, regular, graphical, special Lagrangian cone in $\R^{2k}$ and $l$ is a linear function on $\R^{n-k}$ for some $k>1$.
In particular, $C'$ has phase $\th_{C'}$ satisfying $\f{k-2}2\pi\le\th_{C'}<\f k2\pi$.
From \cite{Y2}, it follows that $C'$ is flat.
This is a contradiction. Hence, the singular set is empty and we complete the proof by Theorem \ref{Allard}.
\end{proof}

\begin{lemma}\label{Cri-supCri-Du} 
Let $u\in C^\infty(B_2)$, and $L$ be a smooth HSL graph of the function $Du$ on $B_2\subset\R^n$ in $\R^{2n}$. If the phase of $L$ satisfies \eqref{C-SC-phase} on $B_2$, then the volume of $L\cap(B_{1}\times\R^n)$ is uniformly bounded by a constant depending only on $n,\sup_{B_2}|Du|$.
\end{lemma}
\begin{proof}
Let $v=\sqrt{\det\left(I_n+(D^2u)^2\right)}$ with the identity matrix $I_n$ of order $n$, and $\k:=\max\{1,\sup_{B_2}|Du|\}$.
Let $\si_k=\si_k(D^2u)$ denote the $k$-th elementary symmetric polynomial of $D^2u$.
Recalling a divergence formula (see p. 495 in \cite{WdY} for instance):
\begin{equation}\aligned\label{divStruc}
\int_{\R^n}\sum_{i=1}^n\f{\p\si_{k}}{\p u_{ij}}\f{\p\varphi}{\p x_i}=0\qquad \mathrm{for\ each}\ j=1,\cdots,n,
\endaligned
\end{equation}
where $\varphi$ is a Lipschitz function on $B_2$ with compact support in $B_2$.
Combining with \eqref{divStruc}, for $k=\{1,\cdots,n\}$ we have
\begin{equation}\aligned\label{kesikuuu}
k\int_{B_{2\varrho}}\varphi\si_k=\int_{B_{2\varrho}}\varphi\sum_{i,j=1}^n\f{\p\si_k}{\p u_{ij}}\f{\p^2u}{\p x_i\p x_j}=-\int_{B_{2\varrho}}\sum_{i,j=1}^n\f{\p\varphi}{\p x_i}\f{\p\si_k}{\p u_{ij}}\f{\p u}{\p x_j}.
\endaligned
\end{equation}
From (3.2) in \cite{WdY}, it follows that
\begin{equation}\aligned\label{vlai}
\sum_{i=1}^n\f{v}{1+\la_i^2}=\cos\th\sum_{1\le2k+1\le n}(-1)^k(n-2k)\si_{2k}-\sin\th\sum_{1\le2k\le n}(-1)^k(n-2k+1)\si_{2k-1}.
\endaligned
\end{equation}
Noting $\si_k\ge0$ (see Lemma 2.1 in \cite{WdY}), and $0<\p_{\la_i}\si_k<(n-k+1)\si_{k-1}$ (see page 495 in \cite{WdY}) for each $1\le k\le n-1$.
Let $\e$ be a Lipschitz function on $B_2$ with $\e\equiv 1$ on $B_{3/2}$, $|D\e|\le10$ and $\e\equiv0$ outside $B_{8/5}$. Using \eqref{kesikuuu} with $k=n-1$, we have
\begin{equation}\aligned
\int_{B_{3/2}}\sum_{i=1}^n\f{v}{1+\la_i^2}\le c_n\sum_{k=0}^{n-1} \int_{B_2}k\e\si_{k}\le c_n\sum_{k=0}^{n-2} \int_{B_{8/5}}k\si_{k}.
\endaligned
\end{equation}
Here, $c_{n}\ge1$ denotes a general constant depending only on $n$, which may change from line to line. By induction, one concludes that
\begin{equation}\aligned\label{fvlai11}
\int_{B_{3/2}}\sum_{i=1}^n\f{v}{1+\la_i^2}\le c_n\k^{n-1}.
\endaligned
\end{equation}

Let $\xi$ be a Lipschitz function satisfying $\xi=1$ on $B_{1}$, $\xi=0$ outside $B_{3/2}$ and $|D\xi|\le 2$.
We also regard $\xi$ as a function on $L$ by identifying $\xi(x)=\xi(x,Du(x))$.
Integrating by parts in conjuction with Cauchy-Schwartz inequality implies
\begin{equation}\aligned\label{bnbui2xi2*111}
-\int_L u_i\xi^4\De_L u_i=&\int_L|\na^L u_i|^2\xi^4+4\int_L u_i\xi^3\lan\na^L u_i,\na^L\xi\ran\\
\ge&\f12\int_L|\na^L u_i|^2\xi^4-8\int_L u_i^2|\na^L\xi|^2\xi^2.
\endaligned
\end{equation}
From \eqref{HL-JnaTh}, for each $i$ one has
\begin{equation}\aligned\label{DeLuigijthj}
\De_L u_i=\lan H_L,E_{n+i}\ran=\lan J\na^L\th,E_{n+i}\ran=\lan \na^L\th,E_i\ran=\sum_{j=1}^ng^{ij}\f{\p\th}{\p x_j}.
\endaligned
\end{equation}
By harmonicity of $\th$, \eqref{DeLuigijthj}, and Cauchy-Schwartz inequality, 
\begin{equation}\aligned\label{xi4Deui}
-\int_L u_i\xi^4\De_L u_i=&-\int_L u_i\xi^4\sum_{j=1}^ng^{ij}\f{\p\th}{\p x_j}=-\int_L\xi^4\lan\na^L\th,\na^L u\ran=\int_L u\lan\na^L\th,\na^L \xi^4\ran\\
=&4\int_L u\xi^3\xi_i\De_L u_i=-4\int_L\lan\na^L(u\xi^3\xi_i),\na^L u_i\ran\\
=&-4\int_L\lan3u\xi_i\na^L\xi+\xi\na^L(u\xi_i),\na^L u_i\ran\xi^2\\
\le&\f14\int_L|\na^L u_i|^2\xi^4+16\int_L\left|3u\xi_i\na^L\xi+\xi\na^L(u\xi_i)\right|^2.
\endaligned
\end{equation}
We may assume $u(0)=0$. From \eqref{fvlai11}\eqref{bnbui2xi2*111} and \eqref{xi4Deui}, we get
\begin{equation}\aligned\label{fvlai22}
\int_L|\na^L u_i|^2\xi^4\le& c_n\k^2\int_L |\na^L\xi|^2(\xi^2+\xi_i^2)+c_n\int_L\xi^2\left|\na^L(u\xi_i)\right|^2\\
\le& c_n\k^2\int_{B_{3/2}}\sum_{i=1}^n\f{v}{1+\la_i^2}\le c_n\k^{n+1}.
\endaligned
\end{equation}
Combining \eqref{fvlai11}\eqref{fvlai22}, we can obtain
\begin{equation*}\aligned
\int_{B_1} v\le\f1n\int_{B_1}\sum_{i=1}^n\f{\la_i^2v}{1+\la_i^2}+\f1n\int_{B_1}\sum_{i=1}^n\f{v}{1+\la_i^2}\le c_n\k^{n+1}.
\endaligned
\end{equation*}
This completes the proof.
\end{proof}

\begin{theorem}\label{Cri-superCri}
If $u\in C^{1,1}(B_1)$ and $L=\{(x,Du(x)):x\in B_1\}$ has harmonic phase $\th$ satisfying \eqref{C-SC-phase} on $L$, then $L$ is smooth and the norm of second fundamental form of $L$ is uniformly bounded on $\mathbf{B}_{1/2}$ by a constant depending only on $n,\sup_{B_1}|Du|$.
\end{theorem}
\begin{proof}
Let us prove it by contradiction. Given a constant $K\ge0$.
Let $u_\ell\in C^{1,1}(B_1)$, and $L_\ell$ be the Lagrangian graph of $Du_\ell$ with phase $\th_\ell\ge\f{n-2}2\pi$ and $\sup_{B_1}|Du_\ell|\le K$ so that $\lim_{\ell\to\infty}\sup_{\mathbf{B}_{1/2}}|A_{L_\ell}|=\infty$ with the second fundamental form $A_{L_\ell}$ of $L_\ell$ ($L_\ell$ is smooth from Lemma \ref{smoothLu}).
There exists a sequence of points $\mathbf{p}_\ell\in \mathbf{B}_1\cap L_\ell$ such that
\begin{equation}\aligned\label{rkgotoinfty}
r_\ell:=\left(1-|\mathbf{p}_\ell|\right)|A_{L_\ell}|(\mathbf{p}_\ell)=\sup_{\mathbf{B}_1\cap L_\ell}\left(1-|\mathbf{x}|\right)|A_{L_\ell}|(\mathbf{x}) \rightarrow\infty\qquad \text{as } \ell\rightarrow\infty. 
\endaligned
\end{equation}
Denote $\mathbf{p}_\ell=(p_\ell,Du_\ell(p_\ell))$.
Put $\tau_\ell=1-|\mathbf{p}_\ell|>0$, then $\mathbf{B}_{\tau_\ell}(\mathbf{p}_\ell)\subset\mathbf{B}_1$. Let $R_\ell=2r_\ell/\tau_\ell$, and $\tilde{L}_\ell$ be a scaling of a part of $L_\ell$ defined by
$$\tilde{L}_\ell=\{R_\ell(\mathbf{x}-\mathbf{p}_\ell)\in\R^n\times\R^n|\, \mathbf{x}\in L_\ell\cap \mathbf{B}_{\tau_\ell/2}(\mathbf{p}_\ell)\},$$
then $\tilde{L}_\ell$ is a special Lagrangian graph in $\mathbf{B}_{r_\ell}$ with $\p\tilde{L}_\ell\cap\mathbf{B}_{r_\ell}=\emptyset$.
Let $A_{\tilde{L}_\ell}$ be the second fundamental form of $\tilde{L}_\ell$ in $\R^{2n}$.
Since $\f{\tau_\ell}2\le1-|\mathbf{x}|$ for every $\mathbf{x}\in \mathbf{B}_{\f{\tau_\ell}2}(\mathbf{p}_\ell)$, by the definition of $r_\ell$ we have
\begin{equation}\aligned\label{BdSii}
\sup_{\tilde{L}_\ell}|A_{\tilde{L}_\ell}|=&\f1{R_\ell}\sup_{\mathbf{B}_{\f{\tau_\ell}2}(\mathbf{p}_\ell)\cap L_\ell}|A_{L_\ell}|
\le\f1{R_\ell}\f{2}{\tau_\ell}\sup_{\mathbf{B}_{\f{\tau_\ell}2}(\mathbf{p}_\ell)\cap L_\ell}\left(1-|\mathbf{x}|\right)|A_{L_\ell}|(\mathbf{x})\\
\le&\f{2}{R_\ell\tau_\ell}\sup_{\mathbf{B}_1\cap L_\ell}\left(1-|\mathbf{x}|\right)|A_{L_\ell}|(\mathbf{x})=\f{2r_\ell}{R_\ell\tau_\ell}=1,
\endaligned
\end{equation}
and
\begin{equation}\aligned\label{BdSii0}
|A_{\tilde{L}_\ell}|(0)=\f1{R_\ell}|A_{L_\ell}|(\mathbf{p}_\ell)
=\f1{R_\ell}\f{1}{\tau_\ell}\sup_{\mathbf{B}_1\cap L_\ell}\left(1-|\mathbf{x}|\right)|A_{L_\ell}|(\mathbf{x})=\f{r_\ell}{R_\ell\tau_\ell}=\f12.
\endaligned
\end{equation}
Up to choosing the subsequence, we assume that $\tilde{L}_\ell$ converges to a complete immersed non-flat HSL submanifold $L_\infty\subset\R^{2n}$ with single-valued phase $\th_\infty$. 
From $\f{n-2}2\pi\le\th_\ell<\f n2\pi$, it follows that $\f{n-2}2\pi\le\th_\infty\le\f n2\pi$.
Let $\{z_k\}$ be a sequence of points in $L_\infty$ so that
\begin{equation}\aligned
\sup_{L_\infty}\th_\infty=\lim_{k\to\infty}\th_\infty(z_k).
\endaligned
\end{equation}

Let $T=[|L_\infty|]$ with phase $\th=\th_\infty$.
Given $r>0$, let $\{T_{k,r}\}_k$ be a sequence of indecomposable components of $\{T\llcorner\mathbf{B}_{2r}(z_k)\}_k$ with $z_k\in\mathrm{spt}T_{k,r}$ and $V_{k,r}=|T_{k,r}|$.
We claim
\begin{equation}\aligned
\liminf_{k\to\infty}\max\left\{\f{\mu_{V_{k,r}}(\mathbf{B}_{r}(z_k))}{\omega_n r^n},\f{\mathbf{m}_{z_k}(\mu_{V_{k,r}},r)}{\mathbf{m}_{z_k}(\mu_{V_{k,r}})}\right\}>1.
\endaligned
\end{equation}
Or else, $T_{k,r}$ converges to a ball in an $n$-plane of multiplicity one up to a subsequence, which violates to \eqref{BdSii}\eqref{BdSii0} and Theorem \ref{Allard}.

Given a positive sequence $\{r_k\}$, we say it \emph{smallest} satisfying a property if for any positive sequence $\{r_k'\}$ satisfying the property, there holds 
$$\limsup_{k\to\infty}\f{r_k}{r_k'}\le1.$$
Let $\de\in(0,1/4]$ be the constant in Theorem \ref{Allard} with $\be=n\pi/2$ there.
There are a constant $0<\de^*<<\de$ and a smallest positive sequence $\{\r_k\}$ satisfying 
\begin{equation}\aligned\label{de*rkpk*}
\liminf_{k\to\infty}\max\left\{\f{\mu_{V_k}(\mathbf{B}_{\r_k}(z_k))}{\omega_n\r_k^n},\f{\mathbf{m}_{z_k}(\mu_{V_k},\r_k)}{\mathbf{m}_{z_k}(\mu_{V_k})}\right\}\ge1+\de^*,
\endaligned
\end{equation}
where $V_k=|T_k|$ and $\{T_k\}$ is a sequence of indecomposable components of $\{T\llcorner\mathbf{B}_{2\r_k}(z_k)\}$ with $z_k\in\mathrm{spt}T_k$.

For each $k$,
let $\tilde{T}_k=(\tau_{z_k,\r_k})_\#T_k$, $\tilde{V}_k=|\tilde{T}_k|$ and $\tilde{\th}_k=\th\circ(\tau_{z_k,\r_k})^{-1}$.
Up to choosing subsequences, we may assume that $\tilde{V}_k$ converges to an integral Lagrangian varifold $V^*=(L^*,\vth^*)$, and $\tilde{\th}_k$ converges to a function $\th^*$ satisfying $\f{n-2}2\pi\le\th^*\le\f n2\pi$. From Theorem \ref{Ahlfors} and Lemma \ref{Cri-supCri-Du}, $V^*$ has Euclidean volume growth.
From Theorem \ref{current-n}, 
there is an orientation $\vec{T}^*$ so that $T^*=(L^*,\vth^*,\vec{T}^*)$ is a HSL current in $\R^{2n}$ with $\p T^*=0$.
From Theorem \ref{Allard} and \eqref{de*rkpk*}, the support of $\tilde{T}_k$ in $\mathbf{B}_{1/2}$ converges to $L^*\cap\mathbf{B}_{1/2}$ smoothly. Hence,
$\th^*(0)=\sup_{L^*}\th^*$, and $\th^*$ is a constant on $L^*\cap\mathbf{B}_{1/2}$ by the maximum principle for the harmonic $\th^*$.

Let $L^*_0$ denote the maximal connected relatively open subset of $L^*$ containing the origin so that $\th^*$ is a constant on $ L^*_0$. 
Set $T^*_0=(L^*_0,\vth^*,\vec{T}^*)$. We claim 
\begin{equation}\label{remthL*}
\p T^*_0=0.
\end{equation}
Let us prove \eqref{remthL*} by contradiction.
We assume $\p T^*_0\neq0$. 
There exist a point $x_*\in\text{spt}\p T^*_0$ and two constants $\ep_*,\tilde{\ep}_*>0$ so that
\begin{equation}\aligned
\mathbf{M}(\p T^*_0\llcorner \mathbf{B}_{\ep_*}(x_*))\ge\tilde{\ep}_*.
\endaligned
\end{equation}
Let $T^*_t=T^*\llcorner\{\th^*>\th^*(0)-t\}$ for each $t>0$.
From $T_t^*\rightharpoonup T_0^*$ as $t\to0$, we have
\begin{equation}\aligned
\mathbf{M}(\p T^*_0\llcorner \mathbf{B}_{\ep_*}(x_*))\le\liminf_{t\to0}\mathbf{M}(\p T^*_t\llcorner \mathbf{B}_{\ep_*}(x_*))
\endaligned
\end{equation}
from the semi-continuity of mass. Hence, there is a constant $t_*>0$ so that
\begin{equation}\aligned
\tilde{\ep}_*\le\mathbf{M}(\p T^*_0\llcorner \mathbf{B}_{\ep_*}(x_*))\le2\mathbf{M}(\p T^{*}_t\llcorner \mathbf{B}_{\ep_*}(x_*))\qquad\text{for each }t\in[0,t_*].
\endaligned
\end{equation}
Without loss of generality, we assume that $(\p T^{*}_{t_*}-\p T^{*}_0)\llcorner\mathbf{B}_{\ep_*}(x_*)\neq0$, or else we further consider $\p T^{*}_t$ for $t>t_*$.
Hence, there is a small ball $\mathbf{B}_{\si}(y)\subset \mathbf{B}_{\ep_*}(x_*)$ with $y\in\text{spt}\p T^*_0$ and $0<\si<<\ep_*$ so that $(T^*-T^{*}_s)\llcorner\mathbf{B}_{\si}(y)=0$ for each $s\ge t_*$. Therefore, there is a constant $\de^*>0$ so that
\begin{equation}\aligned
\mathbf{M}(\p T^*_t\llcorner \mathbf{B}_{\ep_*+\si}(y))
\ge\de^*\min\{\mathbf{M}((T^*-T^*_t)\llcorner\mathbf{B}_\si(y)),\mathbf{M}(T^*_t\llcorner \mathbf{B}_\si(y))\}
\endaligned
\end{equation}
for all $t\in\R$.

Let $f_\ep(\tau)=\log(\th^*(0)+\ep-\tau)$ for each $\ep\in(0,1]$.
Now, we can apply Lemma \ref{thinfty-vari-harm} to get the following Neumann-Poincar\'e inequality 
\begin{equation}\aligned\label{fepth*mean}
\int_{\mathbf{B}_{\si}(y)}|f_\ep\circ\th^*-\overline{f_\ep\circ\th^*}|d\mu_{T^{*}}\le c_*\sup_{k\in\Z} \liminf_{i\to\infty}\int_{\mathbf{B}_{\ep_*+\si}(x_*)}|\na^{\Si_i} (f_\ep\circ\varphi_{i,k}\circ\tilde{\th}_i)|d\mu_{\tilde{T}_i},
\endaligned
\end{equation}
where $\overline{f_\ep\circ\th^*}$ is the average of $f_\ep\circ\th^*$ on $\mathbf{B}_{\si}(y)$ w.r.t. $\mu_{T^{*}}$, $c_*>0$ is a constant, and $\Si_i$ is the support of $\tilde{T}_i$. Here, 
\begin{equation}
\varphi_{i,k}\circ\tilde{\th}_i=\left\{\begin{split}
&\sup_{L_i}\th_i\qquad &&\text{if}\  \tilde{\th}_i\ge\sup_{L_i}\th_i+2k\pi/\ell^*\\
&\tilde{\th}_i-2k\pi/\ell^* \qquad &&\text{if}\  \tilde{\th}_i<\sup_{L_i}\th_i+2k\pi/\ell^*
\end{split}\right.
\end{equation}
for each $k\in\Z_-$, and $\ell^*$ is the positive integer defined as in Lemma \ref{thinfty-vari-harm} with $\tilde{\La}=n\pi/2$.
Hence, $\th^*(0)+\ep-\varphi_{i,k}\circ\tilde{\th}_i$ is a positive subharmonic function on $\Si_i$ for large $i$.
Then, \eqref{caccith*} holds for $\ell=0$ with $\th$ replaced by $\varphi_{i,k}\circ\tilde{\th}_i$. From the monotonicity of $f_\ep$ and Cauchy-Schwartz inequality, we conclude that
\begin{equation}\aligned\label{fepth*deri}
\liminf_{i\to\infty}\int_{\mathbf{B}_{\ep_*+\si}(x_*)}|\na^{\Si_i} (f_\ep\circ\varphi_{i,k}\circ\tilde{\th}_i)|d\mu_{\tilde{T}_i}\le c_{**},
\endaligned
\end{equation}
where $c_{**}$ is a constant independing of $\ep\in(0,1]$ and $k\in\Z$.
Put $\phi_\ep=f_\ep\circ\th^*$. Compared with \eqref{thinftykVinfty}, we get
\begin{equation}\aligned
|\na^{L^*}\phi_\ep|^2\le-\De_{L^*}\phi_\ep \qquad \text{in the distribution sense}.
\endaligned
\end{equation}
Combining \eqref{fepth*mean} and \eqref{fepth*deri}, we can empoly De Giorgi-Nash-Moser iteration to obtain the mean value inequality on a ball $\mathbf{B}_{\si/2}(y)$ for $\th^*_\ep:=\th^*(0)-\th^*+\ep$. Namely,
\begin{equation}\aligned
\f1{\mu_{T^{*}}(B_{\si/2}(y))}\int_{B_{\si/2}(y)}\th^*_\ep d\mu_{T^{*}}\le\tilde{c}\th^*_\ep=\tilde{c}\ep,
\endaligned
\end{equation}where $\tilde{c}$ is a constant independent of $\ep$.
Letting $\ep\to0$ gives that $\th^*$ is a constant on $\mathbf{B}_{\si/2}(y)\cap L^*$.
This contradicts to the definition of $L^*_0$. We have proven the claim \eqref{remthL*}.

From \eqref{de*rkpk*}, $T^*_0$ has multiplicity one on $L^*_0$.
From Theorem 6.2 in \cite{D3}, either $L^*_0$ is flat, or $L^*_0$ is an entire graph in $\R^{2n}$ of the function $Du^*$ for some $u^*\in C^\infty(\R^n)$. 
For the graph $L^*_0$, we consider tangent cones of $L^*_0$ at infinity, and deduce that the function $Du^*$ has linear growth using Theorem 6.2 in \cite{D3} again. From Proposition 2.1 in \cite{WdY}, we get the flatness of $L^*_0$ up to Federer's dimension reduction argument.

From \eqref{de*rkpk*}, $T^*$ is indecomposable in $\mathbf{B}_2$, i.e., $V^*\llcorner\mathbf{B}_2=|T_0^*|\llcorner\mathbf{B}_2$. From Theorem \ref{Allard}, the support of $\tilde{T}_k$ in $\mathbf{B}_{3/2}$ converges to $L^*_0\cap\mathbf{B}_{3/2}$ smoothly. Then \eqref{de*rkpk*} gives
\begin{equation}\aligned
\max\left\{\f{\mu_{V^*}(\mathbf{B}_{1})}{\omega_n},\f{\mathbf{m}_{0}(\mu_{V^*},1)}{\mathbf{m}_{0}(\mu_{V^*})}\right\}\ge1+\de^*.
\endaligned
\end{equation}
This contradicts to the flat $L_0^*$ and $V^*\llcorner\mathbf{B}_2=|T_0^*|\llcorner\mathbf{B}_2$.
Hence, we complete the proof.
\end{proof}
From Theorem \ref{Cri-superCri}, we immediately have the following corollary.
\begin{corollary}
Let $L$ be an n-dimensional smooth entire HSL graph in $\R^{2n}$ with phase $\th\ge\f{n-2}2\pi$. If the defining function of $L$ has linear growth,
then $L$ is flat.
\end{corollary}
\begin{remark}
If $L$ is an n-dimensional smooth entire HSL graph in $\R^{2n}$ with phase $\th$ satisfying $\f{n-2}2\pi<\inf_L\th<\f n2\pi$, then the defining function of $L$ has linear growth after a small rotation. Warren \cite{Wm} has already proved the flatness of $L$ in this case.
\end{remark}

\begin{proof}[Proof of Theorem \ref{HessHS}]
From Theorem \ref{Cri-superCri}, we only need to consider the case $D^2u\ge-1/\sqrt{3}$ a.e. on $B_1$.
By a blowing up argument as in Lemma \ref{smoothLu}, it follows that $u$ is smooth using Theorem 3.1 in \cite{WmY0}.
Let us prove the curvature estimate by contradiction. 

Let $u_\ell\in C^\infty(B_1)$ with $D^2u_\ell\ge-1/\sqrt{3}$ a.e., and $L_\ell$ be the Lagrangian graph of $Du_\ell$ with $Du_\ell(0)=0$ so that the second fundamental form $A_{L_\ell}$ of $L_\ell$ satisfies 
$\lim_{\ell\to\infty}\sup_{\mathbf{B}_{1/2}}|A_{L_\ell}|=\infty$.
We follow the beginning of the proof of Theorem \ref{Cri-superCri}, and obtain a complete smooth non-flat HSL submanifold $L_\infty$ with phase $\th_\infty\in[-\f n6\pi,\f n2\pi]$. 
We consider a rotation $\mathfrak{R}:\,\R^n\times\R^n\to\R^n\times\R^n$ defined by $\mathfrak{R}(x,y)=(\hat{x},\hat{y})$ with
\begin{equation}\label{labxby}
\left\{\begin{split}
&\hat{x}=\f{\sqrt{3}}2 x+\f12 y\\
&\hat{y}=-\f12 x+\f{\sqrt{3}}2 y\\
\end{split}\right..
\end{equation}
From (9.10) in \cite{D3}, $\mathfrak{R}(L_\ell)$ is a graph of $D\tilde{u}_\ell$ with $-\sqrt{3}\le D^2\tilde{u}_\ell\le\sqrt{3}$. Hence, $\mathfrak{R}(L_\infty)$ is an entire Lagrangian graph of $D\tilde{u}_\infty$ with $-\sqrt{3}\le D^2\tilde{u}_\infty\le\sqrt{3}$ on $\R^n$. 
In particualr, there holds a Neumann-Poincar\'e inequality on $L_\infty$. For the harmonic function $\th:=\sup_{L_\infty}\th_\infty-\th_\infty\ge0$, we have the Harnack's inequality
\begin{equation}\aligned
\th(x)\le c\,\th(y),
\endaligned
\end{equation}
where $c>1$ is a constant. From $\inf_{L_\infty}\th=0$ and the above inequality, we get $\th\equiv0$ on $L_\infty$.
Namely, $\mathfrak{R}(L_\ell)$ is an entire special Lagrangian graph over $\R^n$ with $-\sqrt{3}\le D^2\tilde{u}_\infty\le\sqrt{3}$ on $\R^n$.  
Hence, $\tilde{u}_\infty$ is quadratic from the proof Theorem 1.1 in \cite{WmY0} (see also \cite{D1}).
However, this contradicts to the non-flat $L_\infty$. We complete the proof.
\end{proof}

From Theorem \ref{HessHS}, we immediately have the following corollary.
\begin{corollary}
Let $L$ be an n-dimensional smooth entire HSL graph of the function $Du$ in $\R^{2n}$. If $D^2u\ge-1/\sqrt{3}$ a.e. on $\R^n$,
then $L$ is flat.
\end{corollary}

\section{Appendix I: Smoothing Lipschitz functions on countably rectifiable sets}

Let $M$ be a complete Riemannian manifold and $U$ be an open set in $M$. 
Let $V$ be a rectifiable $n$-varifold in $U$ with integer multiplicity and support $L$.
We assume the corresponding Radon measure $\mu_V$ satisfying $\mu_V(U)<\infty$.
\begin{lemma}\label{Lip-to-smooth}
Let $f$ be a Lipschitz function on $L$, then for any $\ep>0$, there is a smooth function $f_\ep$ on $U$ so that 
\begin{equation}\aligned
\int_U\left(|f_\ep-f|^2+|\na^L(f_\ep-f)|^2\right)d\mu_V<\ep.
\endaligned
\end{equation}
\end{lemma}
\begin{proof}
By Nash's isometric embedding theorem,
we assume that $U$ is a smooth submanifold (with boundary) in $\R^m$.
From the countably $n$-rectifiable $L$, there is a sequence of $n$-dimensional embedded $C^1$-submanifolds $\{S_j\}_{j\ge1}$ of $U\cap \R^m$ so that $L\subset\cup_{j\ge0}S_j$, where $\mathcal{H}^n(S_0)=0$. We define $\tilde{S_1}=S_1$, and $\tilde{S_j}=S_j\setminus\cup_{1\le i\le j-1}S_i$ by induction for $j\ge2$. Then $\{\tilde{S_j}\}_{j\ge1}$ are mutually disjoint. For any $\ep>0$, there is an integer $n_\ep>0$ so that $\sum_{j> n_\ep}\mu_V(L\cap\tilde{S_j})<\ep/2$. Put $\hat{S_1}=\tilde{S_1}$. There is a sequence of relatively open subsets $\hat{S_j}\subset\tilde{S_j}$ so that $\mu_V(\tilde{S_j}\setminus\hat{S_j})<2^{-j-1}\ep$ and $\de_j:=d(\hat{S_j},\cup_{1\le i\le j-1}\tilde{S}_i)>0$. Put $\de_\ep=\min_{1\le j\le n_\ep}\de_j$ and $S=\cup_{1\le j\le n_\ep}\hat{S_j}$. Then
\begin{equation}\aligned\label{muThatSj}
\mu_V(U\setminus S)\le\sum_{j> n_\ep}\mu_V(L\cap\tilde{S_j})+\sum_{j=1}^{n_\ep}\mu_V(\tilde{S_j}\setminus\hat{S_j})<\f{\ep}2+\sum_{j=1}^{n_\ep}2^{-j-1}\ep<\ep.
\endaligned
\end{equation}
In particular, $f$ is Lipschitz on $S$.

Let $\si$ be a symmetric smooth nonnegative function on $\R$ with compact support in $(-1,1)$.
Put $\si_{\ep}(s)=\left(\int_{\R^{n}}\si(|\mathrm{x}|^2)d\mathrm{x}\right)^{-1}\ep^{-n}\si(\ep^{-2}s)$ for all $\ep>0$. Then 
\begin{equation}\aligned\label{sits=1}
\mathcal{H}^{n-1}(\mathbb{S}^{n-1}(1))\int_0^\ep s^{n-1}\si_{\ep}(s^2)ds=\int_{\R^{n}}\si_\ep(|\mathrm{x}|^2)d\mathrm{x}=1.
\endaligned
\end{equation}
Let  $\mu_S=\mathcal{H}^n\llcorner S$.
For each $t<\de_\ep/2$, we define a smooth function $f_t$ on $U$ by letting
\begin{equation}\aligned\label{juan}
f_t(y)=\f{\int \si_t(|y-x|^2)f(x) d\mu_S(x)}{\int\si_t(|y-x|^2)d\mu_S(x)}\qquad\qquad \mathrm{for\ each}\ y\in L.
\endaligned
\end{equation} 
We consider a point $p\in L\cap S$ so that $\na^L f(p)$ exists.
Given a vector $\mathrm{v}\in T_pS$, we assume $T_pS=\R^n\times\{0^{m-n}\}\subset\R^{m}$ up to a rotation. 
Let $X$ be a smooth vector field in $U$ with $X(p)=\mathrm{v}$.
Since
$f(x)-f(p)=(x-p)\cdot\na^S f(p)+o(|x-p|)$,
we have
\begin{equation}\aligned\label{ep0Xsiepv}
\lim_{t\to0}\int X\si_t(|y-x|^2)\big|_{y=p} &(f(x)-f(p)) d\mu_S=-\int_{B_1(0^n)} \na_{\mathrm{v}}\si(|\mathrm{x}|^2)\lan\mathrm{x},\na^Sf(p)\ran d\mathcal{H}^n\\
=&\int_{B_1(0^n)} \si(|\mathrm{x}|^2)\lan\mathrm{v}, \na^Sf(p)\ran d\mathcal{H}^n=\lan\mathrm{v}, \na^Sf(p)\ran,
\endaligned
\end{equation} 
where the directional differential $X$ of $\si_t$ is w.r.t. $y$, and
we have used integration by parts in the second equality in \eqref{ep0Xsiepv}.
On the other hand,
\begin{equation}\aligned\label{ep0Xsiepv0}
&\lim_{t\to0}\int \si_t(|y-x|^2)\big|_{y=p}(f(x)-f(p)) d\mu_S\int X\si_t(|y-x|^2)\big|_{y=p} d\mu_S\\
=&\lim_{t\to0}\int \si_t(|x-p|^2)\f{(x-p)\cdot \na^Sf(p)+o(|x-p|)}{t}d\mu_S\int t X\si_t(|y-x|^2)\big|_{y=p} d\mu_S\\
=&-\int_{B_1(0^n)} \si(|\mathrm{x}|^2)\mathrm{x}\cdot \na^Sf(p)  d\mathcal{H}^n\int_{B_1(0^n)} \na_{\mathrm{v}}\si(|\mathrm{x}|^2)d\mathcal{H}^n=0.
\endaligned
\end{equation} 
A direct calculation gives
\begin{equation}\aligned\label{Xfty=p}
Xf_t(y)\big|_{y=p}=&\f{\int X\si_t(|y-x|^2)\big|_{y=p}f(x) d\mu_S(x)}{\int\si_t(|p-x|^2)d\mu_S(x)}\\
&-\f{\int \si_t(|p-x|^2)f(x) d\mu_S(x)\int X\si_t(|y-x|^2)\big|_{y=p} d\mu_S(x)}{\left(\int\si_t(|p-x|^2)d\mu_S(x)\right)^2}\\
=&\f{\int X\si_t(|y-x|^2)\big|_{y=p}(f(x)-f(p)) d\mu_S(x)}{\int\si_t(|p-x|^2)d\mu_S(x)}\\
&-\f{\int \si_t(|p-x|^2)(f(x)-f(p)) d\mu_S(x)\int X\si_t(|y-x|^2)\big|_{y=p} d\mu_S(x)}{\left(\int\si_t(|p-x|^2)d\mu_S(x)\right)^2}.
\endaligned
\end{equation} 
Combining \eqref{ep0Xsiepv} and \eqref{ep0Xsiepv0}, we get
\begin{equation}\aligned\label{epXthlep}
\lim_{t\to0}Xf_t(y)\big|_{y=p}=\lan\mathrm{v}, \na^Sf(p)\ran.
\endaligned
\end{equation} 
This implies $\lim_{t\to0}\na^Sf_t(p)=\na^Sf(p)$. 

Let us consider a general point $y$ in $L\cap S$ where $\na^Sf$ may not exist at $y$.
Let $l_f$ denote the Lipschitz constant of $f$ on $L\cap S$. From \eqref{Xfty=p}, we get
\begin{equation}\aligned
\big|Xf_t\big|(y)\le2l_ft\f{\int \big|X\si_t(|y-x|^2)\big| d\mu_S(x)}{\int\si_t(|y-x|^2)d\mu_S(x)}.
\endaligned
\end{equation} 
Noting that the tangent space $T_xS$ varies continuously.
Hence, there is a constant $c_n>0$ depending only on $n$ so that $|\na^S f_t|\le c_nl_f$ on $L\cap S$. Since $f_t\to f$ and $\na^Sf_t\to\na^Sf$ as $t\to0$ both in the sense of measure, we conclude that
\begin{equation}\aligned
\lim_{t\to0}\int_L\left(|f_t-f|^2+|\na^L(f_t-f)|^2\right)d\mu_S=0.
\endaligned
\end{equation}
Combining \eqref{muThatSj}, for any $\ep>0$, there is a constant $0<t_\ep<\de_\ep/2$ so that 
\begin{equation}\aligned
\int_U\left(|f_{t_\ep}-f|^2+|\na^L(f_{t_\ep}-f)|^2\right)d\mu_V<\ep.
\endaligned
\end{equation}
This completes the proof up to a choice of the index $t$ in $f_t$.
\end{proof}

\section{Appendix II: Almost monotonicity for varifolds in manifolds}

Let $M$ be an $n$-dimensional complete Riemannian manifold. Let $U\subset M$ be an open set, and $V=\vth|L|$ be an integral $k$-varifold in $U$ with generalized mean curvature $H\in L^1_{\mu_V}$ with $1\le k<n$.
From Nash's isometric embedding theorem,
we assume that $U$ is a submanifold in $\R^m$ with $m\ge n$ and $\p U\cap \mathbf{B}_1=\emptyset$. Let $\mathbf{B}_r$ denote the ball in $\R^m$ with radius $r$ and centered at the origin.

Let $\bn$ denote the Levi-Civita connection of $\R^m$, and $\r(\mathbf{x})=|\mathbf{x}|$, then 
\begin{equation}\aligned
&\mathrm{div}_{P}\mathbf{x}=k\qquad\mathrm{and}\qquad \mathbf{x}=\r\bn\r\qquad\qquad \text{for a }k \text{-plane } P\subset \R^m.
\endaligned
\end{equation}
Let $P^N$ denote the orthogonal $(m-k)$-plane to $P$.
For each $\e\in C^1(\R,\R^+)$ and $f\in C^1(\mathbf{B}_1,\R^+)$,
from (3.25) and (3.26) both in \cite{CM1} we get 
\begin{equation}\aligned\label{efxapp}
\mathrm{div}_{P}\left(\e(\r) f\mathbf{x}\right)=&k\e(\r)f+\lan\na_P f,\mathbf{x}\ran\e(\r)+\r\e'(\r)f|\na_P\r|^2\\
=&k\e(\r)f+\lan\na_P f,\mathbf{x}\ran\e(\r)+\r\e'(\r)f-\r\e'(\r)f\left|\na_{P^N}\r\right|^2.
\endaligned
\end{equation}

Now we adapt the method in Proposition 3.7 of Colding-Minicozzi's book \cite{CM1} to study almost monotonicity for $V$.
Let $\phi$ be a nonnegative cut-off function with $\phi'\le0$, $\phi\equiv1$ on $[0,1/2]$ and
spt$\phi\subset[0,1]$. We fix $0<s<1$ for the moment and let $\e(\r)=\phi(\r/s)$ so that
$$\r\e'(\r)=-s\f{\p}{\p s}\phi\left(\f {\r}s\right).$$
We denote $\na_{P^N}$ by $\na^N$ if $P\subset \R^m$ denotes the tangent plane $T_\mathbf{x}L$ whenever it exists.
Let $\mathbf{A}_U$ denote the second fundamental form of $U$ in $\R^m$, and $\k$ be a constant satisfying $|\mathbf{A}_U(Y,Y)|\le\k|Y|^2/n$ for any $Y$ tangent to $U$.
We also see the mean curvature $H$ as a vector in $\R^m$ by adding zero components in the normal direction to $U$.
Let $X=\e(\r) f\mathbf{x}$, using \eqref{efxapp} we get
\begin{equation}\aligned
&-\int \lan H+\mathrm{tr}_L\mathbf{A}_{U},\e(\r) f\mathbf{x}\ran d\mu_V=\int_U\mathrm{div}_LX d\mu_V\\
=&\int\left(k\e(\r)f+\lan\na^L f,\mathbf{x}\ran\e(\r)+\r\e'(\r)f-\r\e'(\r)f\left|\na^N\r\right|^2\right)d\mu_V\\
=&\int\left(k\phi\left(\f{\r}s\right)-s\f{\p}{\p s}\phi\left(\f {\r}s\right)+s\f{\p}{\p s}\phi\left(\f {\r}s\right)\left|\na^N\r\right|^2\right)fd\mu_V+\int\lan\na^L f,\mathbf{x}\ran\phi\left(\f {\r}s\right)d\mu_V.
\endaligned
\end{equation}
Then
\begin{equation}\aligned\label{phis-kk*}
\f{d}{ds}\left(s^{-k}\int\phi\left(\f{\r}s\right)fd\mu_V\right)=& s^{-k-1}\int s\f{\p}{\p s}\phi\left(\f {\r}s\right)\left|\na^N\r\right|^2f d\mu_V\\
&+s^{-k-1}\int\lan f(H+\mathrm{tr}_L\mathbf{A}_{U})+\na^L f,\mathbf{x}\ran\phi\left(\f {\r}s\right)d\mu_V.
\endaligned
\end{equation}
Let $\phi$ increase to the characteristic function of $[0, 1]$ and we integrate \eqref{phis-kk*} for $s$ from $r$ to $R<1$. With the monotone convergence theorem, we get
\begin{equation}\aligned\label{monofHkdf}
\f{1}{R^k}\int_{\mathbf{B}_R}fd\mu_{V}-\f{1}{r^k}\int_{\mathbf{B}_r}fd\mu_{V}\ge-\int_r^R\f{1}{s^{k+1}}\int_{\mathbf{B}_s}(sf(|H|+\k)+\lan\na^L f,\mathbf{x}\ran) d\mu_Vds.
\endaligned
\end{equation}

On the other hand, from \eqref{phis-kk*} it follows that
\begin{equation}\aligned\label{s-nphiVinfty}
\f{d}{ds}\left(\f{e^{\k s}}{s^k}\int\phi\left(\f{\r}s\right)fd\mu_V\right)\ge&\int e^{\k \r}\r^{-k}\f{\p}{\p s}\phi\left(\f {\r}s\right)\left|\na^N\r\right|^2fd\mu_V\\
&+\f{e^{\k s}}{s^{k+1}}\int\lan fH+\na^L f,\mathbf{x}\ran\phi\left(\f {\r}s\right) d\mu_V\\
=\f{d}{ds}\left(\int\f{e^{\k \r}}{\r^k}\phi\left(\f {\r}s\right)\left|\na^N\r\right|^2fd\mu_V\right)&+\f{e^{\k s}}{s^{k+1}}\int\lan fH+\na^L f,\mathbf{x}\ran\phi\left(\f {\r}s\right) d\mu_V.
\endaligned
\end{equation}
Noting $\mathbf{x}=\r\bn\r$. Compared with \eqref{monofHkdf}, we have
\begin{equation}\aligned\label{Almost-mono-Vinfty}
\f{e^{\k R}}{R^k}\int_{\mathbf{B}_R}fd\mu_{V}-\f{e^{\k r}}{r^k}\int_{\mathbf{B}_r}fd\mu_{V}\ge&\int_{\mathbf{B}_R\setminus \mathbf{B}_r}\f{e^{\k \r}}{\r^k}\left|\na^N\r\right|^2fd\mu_V\\
&+\int_r^R\f{e^{\k s}}{s^{k+1}}\int_{\mathbf{B}_s}\lan fH+\na^L f,\r\na\r\ran d\mu_Vds,
\endaligned
\end{equation}
where $\na$ denotes the Levi-Civita connection of $M$.
Analog to the above argument, we get
\begin{equation}\aligned\label{Almost-mono-Vinfty**}
\f{e^{-\k R}}{R^k}\int_{\mathbf{B}_R}fd\mu_{V}-\f{e^{-\k r}}{r^k}\int_{\mathbf{B}_r}fd\mu_{V}\le&\int_{\mathbf{B}_R\setminus \mathbf{B}_r}\f{e^{-\k \r}}{\r^k}\left|\na^N\r\right|^2fd\mu_V\\
&+\int_r^R\f{e^{-\k s}}{s^{k+1}}\int_{\mathbf{B}_s}\lan fH+\na^L f,\r\na\r\ran d\mu_Vds.
\endaligned
\end{equation}

\bibliographystyle{amsplain}

\end{document}